\documentclass[11pt]{article}
\usepackage{color}
\usepackage{geometry}                
\usepackage[parfill]{parskip}    
\usepackage{graphicx}
\usepackage{subfig}	
\usepackage{amssymb}			
\usepackage{epstopdf}
\usepackage{amsmath,bbm}
\usepackage{tabularx,url}
\usepackage{algorithmicx}
\usepackage[ruled]{algorithm}
\usepackage{algpseudocode,enumerate}
\usepackage[numbers]{natbib}
\usepackage{hyperref}
\usepackage{cleveref}
\usepackage{amsthm,changepage,lipsum}
\usepackage{accents}
\usepackage{pgfplots}
\usepackage{tikz}
\usepackage[customcolors,norndcorners]{hf-tikz}

\providecommand{\keywords}[1]{\textbf{{Keywords }} #1}
\providecommand{\subjclass}[1]{\textbf{{Mathematics subject classification }} #1}

\newtheorem{mydef}{Definition}
\newtheorem{mytheorem}{Theorem}
\newtheorem*{mytheorem*}{Theorem}
\newtheorem{corollary}{Corollary}
\newtheorem*{corollary*}{Corollary}
\newtheorem{lemma}{Lemma}

\newcommand{\vect}[1]{\boldsymbol{#1}}

\numberwithin{equation}{section}

\algrenewcommand\algorithmicindent{1.0em}%

\allowdisplaybreaks

\title{Sparse Harmonic Transforms:  A New Class of Sublinear-time Algorithms for Learning Functions of Many Variables}

\author{Bosu Choi\thanks{Department of Mathematics, Michigan State University \texttt{(choibosu@msu.edu)}}
\and
Mark A. Iwen\thanks{Department of Mathematics, and Department of Computational Mathematics, Science, and Engineering(CMSE), Michigan State University \texttt{(iwenmark@msu.edu)}}%
\and
Felix Krahmer\thanks{Department of Mathematics, Technische Universit{\"a}t M{\"u}nchen, Germany \texttt{(felix.krahmer@tum.de)}}
}
\date{}

\begin{document}

\maketitle
\vspace{-12pt}
\begin{abstract}
In this paper we develop fast and memory efficient numerical methods for learning functions of many variables that admit sparse representations in terms of general bounded orthonormal tensor product bases.  Such functions appear in many applications including, e.g., various Uncertainty Quantification (UQ) problems involving the solution of parametric PDE that are approximately sparse in Chebyshev or Legendre product bases \cite{chkifa2016polynomial, rauhut2017compressive}. We expect that our results provide a starting point for a new line of research on sublinear-time solution techniques for UQ applications of the type above which will eventually be able to scale to significantly higher-dimensional problems than what are currently computationally feasible.

More concretely, let $\mathcal{B}$ be a finite Bounded Orthonormal Product Basis (BOPB) of cardinality $|\mathcal{B}| = N$.  Herein we will develop methods that rapidly approximate any function $f$ that is sparse in the BOPB, that is, $f: \mathcal{D} \subset \mathbb{R}^D \rightarrow \mathbb{C}$ of the form
\[f(\vect{x}) = \sum_{b \in \mathcal{S}} c_b \cdot b(\vect{x})\]
with $\mathcal{S} \subset \mathcal{B}$ of cardinality $|\mathcal{S}| = s \ll N$.

Our method adapts the CoSaMP algorithm \cite{needell2009cosamp} to use additional function samples from $f$ along a randomly constructed grid $\mathcal{G} \subset \mathbb{R}^D$ with universal approximation properties in order to rapidly identify the multi-indices of the most dominant basis functions in $\mathcal{S}$ component by component during each CoSaMP iteration.
It has a runtime of just $(s \log N)^{\mathcal{O}(1)}$, uses only $(s \log N)^{\mathcal{O}(1)}$ function evaluations on the fixed and nonadaptive grid $\mathcal{G}$, and requires not more than $(s \log N)^{\mathcal{O}(1)}$ bits of memory.  We emphasize that nothing about $\mathcal{S}$ or any of the coefficients $c_b \in \mathbb{C}$ is assumed in advance other than that $\mathcal{S} \subset \mathcal{B}$ has $|\mathcal{S}| \leq s$.  Both $\mathcal{S}$ and its related coefficients $c_b$ will be learned from the given function evaluations by the developed method.  

For $s\ll N$, the runtime $(s \log N)^{\mathcal{O}(1)}$  will be less than what is required to simply enumerate the elements of the basis $\mathcal{B}$; thus our method is the first approach applicable in a general BOPB framework that falls into the class referred to as {\it sublinear-time}. This and the similarly reduced sample and memory requirements set our algorithm apart from previous works based on standard compressive sensing algorithms such as basis pursuit which typically store and utilize full intermediate basis representations of size $\Omega(N)$ during the solution process.

\end{abstract}



\section{Introduction} \label{Introduction}
\setcounter{equation}{0}


One encounters the problem of multivariate function integration, approximation, interpolation, and learning from a relatively small number of function evaluations in application areas ranging from computational physics to mathematical finance.  A common class of examples in the Uncertainty Quantification (UQ) literature \cite{Smith:2013:UQT:2568154, Xiu:2010:NMS:1893088}, for example, involves the approximation of Quantities of Interest (QoI) that are assumed to be continuous functions of a potentially large number of parameters.  Consequently, uncertainty in the input parameters leads to uncertainty in the QoI outputs which, in turn, necessarily depends on how the QoI behaves as a function of the input parameters.  
In order to understand the uncertainty in the QoI outputs one is therefore forced to approximate the QoI as a function.  
This typically requires multivariate function integration and interpolation, usually via quadrature methods \cite{Dahlquist:2008:NMS:1383510}, sparse grid approaches \cite{bungartz2004sparse}, or (quasi-)Monte Carlo methods \cite{leobacher2014introduction,caflisch1998monte}, depending on the number of parameters (i.e., variables) on which the QoI depends.  In any case, all of these approaches typically must assume that the QoI is a highly smooth function of its input parameters in order to guarantee efficiency and accuracy, though smoothness alone cannot generally save one from the curse of dimensionality \cite{hinrichs2014curse} (i.e., from an exponential sampling and runtime dependence on the number of function variables, $D$).

More recently, sparsity of the quantity of interest in a given Bounded Orthonormal Product Basis (BOPB) has been identified as an appropriate model assumption for UQ problems involving solutions of parametric elliptic partial differential equations \cite{schwab2006karhunen,chkifa2016polynomial,2017arXiv170101671B,rauhut2017compressive, Adcock2017, adcock2017infinite}. This observation allows for a formulation of the QoI approximation problem in the language of compressive sensing (CS), a paradigm introduced in the signal recovery literature in the early 2000's (starting with \cite{donoho2006compressed, Candes2006}, cf.~\cite{foucart2013mathematical} for a comprehensive introduction to the field). Namely, when the function evaluations are performed for random choices of the input parameters, the problem is a special case of the problem of recovering a function that admits a sparse representation in a Bounded Orthogonal System from a small number of function evaluations. This problem in general terms, that is, without assuming a product structure as we encounter it here, has been of interest to the CS community almost from the beginning (see, e.g., \cite{R07, R12,KW14}). 

Building on these general results, a number of more recent works have studied the same questions specifically for the important case of multivariate functions which exhibit sparsity in high-dimensional Chebyshev and Legendre product bases (see. e.g., \cite{rauhut2017compressive,chkifa2016polynomial}).  These methods still store and utilize full intermediate basis representations during the solution process, however.  In order for the problems to still be feasible, they often make additional assumptions on the structure of the sparsity which imply that the degrees of the polynomial basis functions with large coefficients are relatively small.  This has the effect of reducing the overall sampling complexity and size of the basis which, in turn, allows for faster approximation of the QoI function with less required memory.  A simple example is the case where a function of a very large number of variables is assumed to actually only depend on a small subset of them \cite{devore2011approximation,gilbert2008group} (see also \cite{rauhut2017compressive, 2017arXiv170101671B} which achieve a similar effect in the UQ setting when $D \gg s$ via a combination of Petrov-Galerkin approximation and weighted $\ell_1$ minimization techniques).  Methods which assume hyperbolic cross \cite{shen2010sparse,dung2016hyperbolic} or lower set \cite{chkifa2016polynomial} structures on the energetic basis function indexes provide additional examples.  

The connection between UQ and BOPB-sparse function recovery established in these recent works paves the way for us to devise the first sublinear-time compressive sensing methods for
general BOPB frameworks in this paper.  More precisely, we are able to decouple the runtime and memory requirements necessary in order to learn a given BOPB-sparse function from the overall BOP basis size one must initially consider.  In short, we develop extremely fast and memory efficient compressive sensing algorithms for such problems. Besides an improved theoretical performance, also the empirical performance of our method improves over previous approaches. In particular, the enhanced memory efficiency allows us to tackle much larger problem sizes than in previous works.  We expect that the results presented in this paper will trigger follow-up works on sublinear-time solution techniques for UQ applications of the type above which will eventually be able to scale to significantly higher-dimensional problems than what are currently computationally feasible.

Though its focus is on the recovery of functions which exhibit sparsity in an arbitrary BOPB, the method developed herein is a direct descendant of previously existing sublinear-time compressive sensing algorithms developed in the mathematics and computer science communities for data stream processing and sketching applications \cite{gilbert2006sublinear,Gilbert:2007:OSF:1250790.1250824,doi:10.1137/100816705,iwen2014compressed,Gilbert:2017:FSR:3058789.3039872}.  Unlike these previous methods, however, the compressive sensing matrices we are forced to use herein are necessarily solely derived from highly structured combinations of Bounded Orthonormal System (BOS) sampling matrices (see \S\ref{sec:SublinCompSense} for details).  As a result, our recovery algorithm cannot make direct use of any of the group testing and random hashing techniques commonly utilized by such sublinear-time compressive sensing methods.  Instead, we appeal to compressive sensing results concerning the restricted isometry constants of random sampling matrices derived from a BOS in order to develop general energy-based hashing techniques which capitalize on the tensor product basis structure of any given BOP basis $\mathcal{B}$.  These new energy hashing techniques are then used to rapidly identify a given BOPB-sparse function's support set $\mathcal{S} \subset \mathcal{B}$ using the algorithms discussed in Sections \ref{secEntryID} and \ref{secPairing}.

Similarly, the sublinear-time compressive sensing method developed herein can also be viewed as a significantly generalized high dimensional Sparse Fourier Transform (SFT) algorithm \cite{iwen2013improved,indyk2014sparse,kapralov2016sparse,choi2016multi,potts2016sparse,potts2017multivariate,Volkmer2017thesis,morotti2017explicit}.  In particular, the support identification techniques developed for arbitrary BOP bases in Sections \ref{secEntryID} and \ref{secPairing} bear a high-level resemblance to the dimension incremental support identification techniques recently proposed by both Potts and Volkmer et al. \cite{potts2016sparse, kammerer2017high} and Choi et al. \cite{choi2016multi,choi2019multiscale} for the multivariate Fourier basis (see \S\ref{sec:EntryIDandPairingComp} for more details).  Unlike the method proposed herein, however, the aforementioned high dimensional SFT's all use the specific structure of the Fourier basis in fundamental ways which makes their results difficult to directly extend to general BOP bases.  Furthermore, with the notable exception of \cite{iwen2013improved,morotti2017explicit}, none of them provide universal recovery guarantees for all Fourier compressible functions.  
As a result, we need to develop entirely new sublinear-time support identification methods which only depend on general BOP basis structure herein.  For a more detailed discussion of how the SFT results yielded as special cases of the main result herein compare to previous SFT methods for the Fourier basis, we refer the reader to \S\ref{equ:SFTcompareCoro1} below.

\subsection{The Compressive Sensing Problem for BOPB-Sparse Functions}
\label{sec:CompSense}

Toward a more exact problem formulation, let $p \in \mathbb{N}$ be any natural number and $[N]: = \{0, 1, 2, \dots, N-1 \}$ for all $N \in \mathbb{N}$.  The set of functions, $\left\{ T_{k}: \mathcal{D} \rightarrow \mathbb{C}  \right \}_{k \in [N]}$ forms a {\em Bounded Orthonormal System} (BOS) with respect to a probability measure $\sigma$ over $\mathcal{D} \subset \mathbb{R}^p$ with BOS constant $K :=\max_{k \in [N]}\|T_{k}\|_\infty \geq 1$ if $K < \infty$, and
\begin{equation*}
\langle T_{k}, T_{l}  \rangle_{(\mathcal{D},\sigma)} := \int_{\mathcal{D}} T_{k}(\vect{x}) \overline{T_{l}(\vect{x})} d\sigma(\vect{x}) = \delta_{k, l} = 
\begin{cases}
1 & \textrm{if}~k = l\\
0 & \textrm{if}~k \neq l
\end{cases}
\end{equation*}
holds for all $ k, l \in [N]$. Now let $\mathcal{B}_j := \{T_{j,k}:\mathcal{D}_j\rightarrow \mathbb{C}\}_{k\in[M]}$ form a {BOS} with respect to a probability measure $\nu_j$ on $\mathcal{D}_j \subset\mathbb{R}$, with constant $\widetilde{K}_j$ for each $j\in[D]$.  Then, the BOPB functions $\mathcal{B} := \left \{T_{\vect{n}}:  \mathcal{D} \rightarrow \mathbb{C} \right\}_{\vect{n} \in [M]^D }$, defined by 
\begin{equation}
T_{\vect{n}}(\vect{x}):=\prod _{j\in[D]} T_{j;n_j}(x_{j})
\label{def:T_n}
\end{equation}
again form a {BOS} with constant 
\begin{equation*}
K:=\max_{\vect{n} \in [M]^D} ||{T}_{\vect{n}}||_\infty = \prod _{j\in[D]}  \widetilde{K}_j
\end{equation*}
with respect to the probability measure $\vect{\nu}:=\otimes_{j\in[D]} \nu_j$ over $\mathcal{D}:= \times_{j\in[D]} \mathcal{D}_j \subset \mathbb{R}^D$. Throughout this paper, we assume for simplicity that $T_{j;0} \equiv 1$ for all $j$. This assumption is true for the large class of orthonormal polynomials including Trigonometric polynomials, Chebyshev polynomials, Legendre polynomials, Gegenbauer polynomials, Jacobi polynomials, etc.

Herein we consider BOPB-sparse functions $f:  \mathcal{D} \rightarrow \mathbb{C}$ of the form
\begin{equation}
f(\vect{x}):= \sum_{\vect{n}\in \mathcal{S} \subset \mathcal{I} \subseteq [M]^D } c_{\vect{n}} T_{\vect{n}}(\vect{x})
\label{equ:ExactlySparsefunc}
\end{equation}
where  $|\mathcal{S}| = s \ll |\mathcal{I}| \leq |\mathcal{B}| = N = M^D$.  Following \cite{devore2011approximation,rauhut2017compressive,2017arXiv170101671B} we will take $\mathcal{I}$ to be the subset of $[M]^D$ containing at most $d \leq D$ nonzero entries.  

Considering the recovery of $f$ using standard compressive sensing methods \cite{donoho2006compressed,foucart2013mathematical} when $\mathcal{I} = [M]^D$, one can simply independently draw $m'_1$ points, $\mathcal{G}^E := \left\{ \vect{t}_1, \dots, \vect{t}_{m'_1} \right\} \subset \mathcal{D}$, according to $\vect{\nu}$ and then sample $f$ at those points to obtain
\begin{equation}
\vect{y^{\rm E}} = f\left( \mathcal{G}^E \right) := \left( f\left(\vect{t}_1 \right),f\left(\vect{t}_2\right),\dots,f\left(\vect{t}_{m'_1}\right)\right)^T \in \mathbb{C}^{m'_1}.
\label{equ:DefyE}
\end{equation}
Our objective becomes the recovery of $f$ using only the samples $\vect{y^{\rm E}}$.

Let the $m'_1 \times M^D$ {\it random sampling matrix} $\Phi \in \mathbb{C}^{m'_1 \times M^D}$ have entries given by
\begin{equation}
\Phi_{\ell,\vect{n}}=T_{\vect{n}}(\vect{t}_\ell).
\label{def:A}
\end{equation}
We can now form the underdetermined linear system 
$$\vect{y^{\rm E}} = \begin{pmatrix} f \left(\vect{t}_1 \right) \\ f \left(\vect{t}_2 \right) \\ \vdots \\ f \left(\vect{t}_{m'_1}\right) \end{pmatrix}= \begin{pmatrix} T_{\vect{n}_1}(\vect{t}_1) & T_{\vect{n}_2}(\vect{t}_1) & \cdots  & \cdots & T_{\vect{n}_{M^D}}(\vect{t}_1)\\ T_{\vect{n}_1}(\vect{t}_2) & T_{\vect{n}_2}(\vect{t}_2) & \cdots  & \cdots & T_{\vect{n}_{M^D}}(\vect{t}_2) \\
\vdots & \vdots & & \ddots & \vdots
\\T_{\vect{n}_1}(\vect{t}_{m'_1}) & T_{\vect{n}_2}(\vect{t}_{m'_1}) & \cdots & \cdots & T_{\vect{n}_{M^D}}(\vect{t}_{m'_1})\end{pmatrix}\vect{c} = \Phi \vect{c},$$
where $\vect{c} \in \mathbb{C}^{M^D}$ contains the basis coefficients $c_{\vect{n}}$ of $f$, and the index vectors $\vect{n}_{1},\dots,\vect{n}_{M^D} \in [M]^D$ are ordered, e.g., lexicographically.  Note that this linear system is woefully underdetermined when $m'_1 \ll M^D$.  When $\vect{c}$ has only $s \ll M^D$ nonzero entries as it does here, however, the compressive sensing literature tells us that  $\vect{c}$ can still be recovered using significantly fewer than $M^D$ function evaluations as long as the normalized random sampling matrix $\Phi$ has the {\em Restricted Isometry Property} (RIP) of order $2s$ \cite{foucart2013mathematical}.\\

\begin{mydef}[See Definition 6.1 in \cite{foucart2013mathematical}]
The $s^{\rm th}$ restricted isometry constant $\delta_s$ of a matrix $\widetilde{\Phi}\in \mathbb{C}^{m \times N}$ is the smallest $\delta \geq 0$ such that 
\begin{equation*}
(1-\delta)\|\vect{c}\|_2^2\leq \left\|\widetilde{\Phi}\vect{c} \right\|_2^2 \leq (1+\delta)\|\vect{c}\|_2^2
\end{equation*}
holds for all $s$-sparse vectors $\vect{c}\in \mathbb{C}^N$. The matrix $\widetilde{\Phi}$ is said to satisfy the {RIP} of order $s$ if $\delta_s \in (0,1)$.
\end{mydef}

Furthermore, one can show that random sampling matrices have the restricted isometry property with high probability even when $\left| \mathcal{G}^E \right|$ is relatively small.\\

\begin{mytheorem}[See Theorem 12.32 and Remark 12.33 in \cite{foucart2013mathematical}]
\label{thm:BOS_RIP}
Let $A\in \mathbb{C}^{m\times N}$ be the random sampling matrix associated to a {BOS} with  constant $K\geq 1$. If, for $\delta, p \in (0,1)$, 
\begin{equation*}
m\geq a K^2\delta^{-2}s \cdot \max \{\log^2(4s) \log(8N) \log(9m), \log(p^{-1})\},
\end{equation*}
then with probability at least $1-p$, the restricted isometry constant $\delta_s$ of $\widetilde{A}=\frac{1}{\sqrt{m}}A$ satisfies $\delta_s \leq \delta$. The constant $a>0$ is universal.
\end{mytheorem}

Note that Theorem~\ref{thm:BOS_RIP} effectively decouples the number of samples that one must acquire/compute in order to recover any BOPB-sparse $f$ from the overall BOP basis size $\left| \mathcal{B} \right| = M^D$.  It guarantees that a random sampling set of size $\left| \mathcal{G}^E \right| = m'_1 = \mathcal{O}(K^2 \cdot s \cdot D \cdot \log^4(KMD))$ suffices.  The main obstacle to reducing the sampling complexity (i.e., $m'_1$) at this point becomes the BOS sampling constant $K$.  To see why, consider, e.g., the cosine BOPB where for all $j\in[D]$ in \eqref{def:T_n} we set $T_{j;n}(x)=\sqrt{2}\cos\left(nx\right)$ for $n\geq 1$ and $T_{j;0}(x)=1$ in \eqref{def:T_n}.  This leads to a BOS with $K={2}^{D/2}$ with respect to uniform probability measure $\vect{\nu}$ over $\mathcal{D}=[0,2\pi]^D$.  Now we can see that we still face the curse of dimensionality since $K^2 = 2^D$ even for this fairly straightforward BOPB.  Nonetheless, it expresses itself in a dramatically reduced fashion:  $2^D$ is still a vast improvement over $M^D$ for even moderately sized $M > 2$.

As previously mentioned, to further reduce the sampling complexity from scaling like $2^{\mathcal{O}(D)}$ previous work has focussed on developing efficient methods for effectively reducing the basis size to a smaller subset of the total basis $\mathcal{B}$ (see, e.g., \cite{chkifa2016polynomial, rauhut2017compressive}).  To see how this might work in the context of our simple cosine BOPB example above, we can note that the BOPB elements in \eqref{def:T_n} can be rewritten as
$$T_{\vect{n}}(\vect{x}) := {2^{\| \vect{n} \|_0/2}} \prod_{j=0}^{D-1}\cos \left(n_j x_j \right) $$
in that case.  It now becomes obvious that limiting the basis functions to those with indexes in $\mathcal{I} := \{\vect{n}\in [M]^D ~|~ \|\vect{n}\|_0 \leq d \leq D\}$ leads to a reduced BOS constant of $K=2^{d/2} \leq 2^{D/2}$ for the resulting reduced basis, as well as to a smaller basis cardinality of size ${D \choose d} M^d = \mathcal{O}\left((\frac{DM}{d})^d \right)$.

In particular, the utility of the assumption that the $s$ non-negligible basis indexes of $f$, $\mathcal{S} \subset [M]^D$, also belong to the reduced index set $\mathcal{I}$ above is supported in some UQ applications where it is known that, e.g., the solutions of some parametric PDE are not only approximately sparse in some BOP bases such as the Chebyshev or Legendre product bases, but also that most of their significant coefficients  correspond to index vectors $\vect{n} \subset \mathbb{N}^D$ with relatively small (weighted) $\ell_p$-norms \cite{chkifa2016polynomial, rauhut2017compressive}.  In certain simplified situations this essentially implies that $\mathcal{S} \subset \mathcal{I}$ as discussed above.  As a result, we will assume throughout this paper that $\mathcal{S}\subset \mathcal{I}$ so that $N=|\mathcal{I}|={D \choose d} M^d \leq M^D$.\footnote{Additionally, we will occasionally assume that our total grid size $|\mathcal{G}|$ below always satisfies $|\mathcal{G}| \leq N^c$ for some absolute constant $c \geq 1$ in order to simplify some of the logarithmic factors appearing in our big-O notation.  This will certainly always be the case for any standardly used (trigonometric) polynomial BOPB (such as Fourier and Chebyshev product bases) whenever $sKDM < N$.}


Even when $s \cdot K^2 \ll N$ so that the number of required samples $m'_1$ is small compared to the reduced basis size $\left| \mathcal{I} \right| = N$, however, all existing standard compressive sensing approaches for recovering $f$ still need to compute and store potentially fully populated intermediate coefficient vectors $\vect{c}' \in \mathbb{C}^{N}$ at some point in the process of recovering $f$.  As a result, all existing approaches are limited in terms of the reduced basis sizes $\mathcal{I}$ they can consider by both their memory needs and runtime complexities.  In this paper we develop new methods that are capable of circumventing these memory and runtime restrictions for a general class of practical BOP bases.\footnote{Here we note that preconditioning and well chosen sampling distributions are crucial for many BOP bases.  For example, the BOS constant for the standard Legendre basis is $K = \sqrt{2M+1}^D$ which implies that a naive application of Theorem~\ref{thm:BOS_RIP} may require more than $M^D$ (or $M^d$) samples.  However, preconditioning can effectively reduce this BOS constant to $K = \sqrt{3}^d$ in practice \cite{R12}.}  As a result, we make it possible to recover a new class of extremely high-dimensional BOPB-sparse functions which are simply too complicated to be approximated by other means.  We are now prepared to discuss our main results.

\subsection{Main Results}
\label{sec:MainResults}

The proposed sublinear-time algorithm is a greedy pursuit method motivated by CoSaMP\cite{needell2009cosamp}, HTP\cite{foucart2011hard}, and their sublinear-time predecessors \cite{gilbert2006sublinear,Gilbert:2007:OSF:1250790.1250824}.  In particular, it is obtained from CoSaMP by replacing CoSaMP's support identification procedure with a new sublinear-time support identification procedure.  See Algorithm~\ref{alg:main} in Section~\ref{sec:proposed} for pseudocode and other details.  Our main result demonstrates the existence of a relatively small grid of points $\mathcal{G} \subset \mathcal{D}$ which allows Algorithm~\ref{alg:main} to recover any given BOPB-sparse function $f$ in sublinear-time from its evaluations on $\mathcal{G}$.  We refer the reader to Section~\ref{sec:UniversalGrids} for a detailed description of the grid set $\mathcal{G}$ and its use in Algorithm~\ref{alg:main}. The following theorem is a simplified version of Theorem~\ref{mainThm} in Section~\ref{sec:proposed}.\\   


\begin{mytheorem*}[Main Result]
Suppose that  $\left\{ T_{\vect{n}}~\big|~\vect{n} \in \mathcal{I} \subseteq [M]^D \right\}$ is a BOS where each basis function $T_{\vect{n}}$ is defined as per \eqref{def:T_n}, and $T_{j;0}\equiv 1$ for all $j \in [D]$.
Let $\mathcal{F}_{s}$ be the subset of all functions $f \in {\rm span} \left\{ T_{\vect{n}}~\big|~\vect{n} \in \mathcal{I} \right\}$ whose coefficient vectors are $s$-sparse, and let $\vect{c}_f \in \mathbb{C}^{\mathcal{I}}$ denote the $s$-sparse coefficient vector for each $f \in \mathcal{F}_{s}$.  Fix $p\in \left(0,1/3 \right)$, a precision parameter $\eta > 0$, $1 \leq d \leq D$, and $\displaystyle K =\sup_{\vect{n} \in [M]^D \text{ s.t.} \|\vect{n}\|_{0}\leq d } \|T_{\vect{n}}\|_{\infty}$.
Then, one can randomly select a set of i.i.d. Gaussian weights $\mathcal{W} \subset \mathbb{R}$ for use in \eqref{def:fk}, and also randomly construct a compressive sensing grid, $\mathcal{G} \subset \mathcal{D}$, whose total cardinality $\left| \mathcal{G} \right|$ is
$\mathcal{O}\left( \left( s D \mathcal{L}' K^2 + s^3D K^4 \right) \max \left\{d^4 \log^4(s) \log^4({D^2M}), \log^2(\frac{D}{p}) \right\} \right)$, 
such that the following property holds $\forall f \in \mathcal{F}_{s}$ with probability greater than $1-3p$:
\begin{adjustwidth}{0.25in}{0.25in}
Let $\vect{y} = f(\mathcal{G})$ consist of samples from $f \in \mathcal{F}_{s}$ on $\mathcal{G}$.  If Algorithm \ref{alg:main} is granted access to $\vect{y}$, $\mathcal{G}$, and $\mathcal{W}$, then it will produce an $s$-sparse approximation $\vect{a} \in \mathbb{C}^{\mathcal{I}}$ s.t. 
\begin{equation*}
\|\vect{c}_f-\vect{a}\|_2 \leq C\eta,
\end{equation*}
where $C > 0$ is an absolute constant.
\end{adjustwidth}
Furthermore, the total runtime complexity of Algorithm \ref{alg:main} is always\\ $\mathcal{O}\Big(\Big( \left( s^2 D^2 \mathcal{L} K^2 + s^5 D^2 K^4   \right) \max \left\{d^4\log^4(s) \log^4({D^2M}), \log^2(\frac{D}{p}) \right\} \Big)\times \log \frac{\|\vect{c}_f\|_2}{\eta}\Big)$.  
\end{mytheorem*}

Note that Algorithm \ref{alg:main} will run in sublinear-time whenever $s^5 D^2 \mathcal{L} K^4 d^4 \ll |\mathcal{I}|$ (neglecting logarithmic factors).  Here and in the theorem above the parameters $\mathcal{L}$ and $\mathcal{L}'$ depend on your choice of numerical method for computing the inner product between a sparse function in the span of each one-dimensional BOS $\mathcal{B}_j = \left\{ T_{j;m}~\big|~m\in [M] \right\}$.  More specifically, let $\mathcal{L}'_j$ represent the number of function evaluations one needs in order to compute all $M$-inner products $\left\{ \langle g, T_{j;\widetilde{n}} \rangle \right\}_{\widetilde{n} \in [M]}$ in $\mathcal{O}(\mathcal{L})$-time for any given function $g: \mathcal{D}_j \rightarrow \mathbb{C}$ belonging to the span of $\mathcal{B}_j$ that is also $s$-sparse in $\mathcal{B}_j$.  We then set $\mathcal{L}' := \max_{j \in [D]} \mathcal{L}'_j$.  For example, if each BOS $\mathcal{B}_j$ consists of orthonormal polynomials whose degrees are all bounded above by $M$ then quadrature rules such as Gaussian quadrature or Chebyshev quadrature give $\mathcal{L}= \mathcal{O}(M^2)$ and $\mathcal{L}'=\mathcal{O}(M)$ \cite{Dahlquist:2008:NMS:1383510}.  If each $\mathcal{B}_j$ is either the standard Fourier, sine, cosine, or Chebyschev basis then the Fast Fourier Transform (FFT) can always be used to give $\mathcal{L}=\mathcal{O}(M\log M)$ and $\mathcal{L}'=\mathcal{O}(M)$ \cite{Dahlquist:2008:NMS:1383510}.  

Moreover, there are several sublinear-time sparse Fourier transforms as well as sparse harmonic transforms for other bases which could also be used to give other valid $\mathcal{L}'$ and $\mathcal{L}$ combinations \cite{gilbert2014recent,bittens2017deterministic,merhi2017new,gilbert2005improved,iwen2007empirical,iwen2008deterministic,bailey2012design,hassanieh2012simple,hu2015rapidly,iwen2010combinatorial,segal2013improved,iwen2013improved,clw}. These typically have $\mathcal{O}(s^c \log^{c'} M)$ runtime and sampling complexities for small positive absolute constants $c$ and $c'$. 
As a result, one can obtain much stronger results than the main theorem above when $s \ll M$ and every one-dimensional BOS $\mathcal{B}_j$ is either the Fourier, sine, cosine, or Chebyshev basis.  The following corollary of our main theorem is obtained by using deterministic one-dimensional SFT results from \cite{iwen2013improved} and \cite{hu2015rapidly} in order to compute all of the nonzero inner products in lines 6 -- 13 of Algorithm~\ref{alg:Entry}.  They lead to $\mathcal{L}'$ and $\mathcal{L}$ values in Section~\ref{sec:proposed}'s Theorem~\ref{mainThm} of size $\mathcal{O}(s^2 \log^4 M)$. \\ 

\begin{corollary}
Suppose that  $\left\{ T_{\vect{n}}~\big|~\vect{n} \in \mathcal{I} \subseteq [M]^D \right\}$ is a BOS where each basis function $T_{\vect{n}}$ is defined as per \eqref{def:T_n}, and where every one-dimensional BOS $\mathcal{B}_j$ is either the Fourier, sine, cosine, or Chebyshev basis.  
Let $\mathcal{F}_{s}$ be the subset of all functions $f \in {\rm span} \left\{ T_{\vect{n}}~\big|~\vect{n} \in \mathcal{I} \right\}$ whose coefficient vectors are $s$-sparse, and let $\vect{c}_f \in \mathbb{C}^{\mathcal{I}}$ denote the $s$-sparse coefficient vector for each $f \in \mathcal{F}_{s}$.  Fix $p\in \left(0,1/3 \right)$, a precision parameter $\eta > 0$, $1 \leq d \leq D$, and let $\displaystyle K =\sup_{\vect{n} \in [M]^D \text{ s.t.} \|\vect{n}\|_{0}\leq d } \|T_{\vect{n}}\|_{\infty}$.  Then, one can randomly select a set of i.i.d. Gaussian weights $\mathcal{W} \subset \mathbb{R}$ for use in \eqref{def:fk}, and also randomly construct a compressive sensing grid, $\mathcal{G} \subset \mathcal{D}$, whose total cardinality $\left| \mathcal{G} \right|$ is\\
$\mathcal{O}\left(s^3 D \log^4 (M) K^4\max \left\{d^4 \log^4(s) \log^4({D^2M}), \log^2(\frac{D}{p}) \right\} \right)$, 
such that the following property holds $\forall f \in \mathcal{F}_{s}$ with probability greater than $1-3p$:
\begin{adjustwidth}{0.25in}{0.25in}
Let $\vect{y} = f(\mathcal{G})$ consist of samples from $f \in \mathcal{F}_{s}$ on $\mathcal{G}$.  If Algorithm \ref{alg:main} is granted access to $\vect{y}$, $\mathcal{G}$, and $\mathcal{W}$, then it will produce an $s$-sparse approximation $\vect{a} \in \mathbb{C}^{\mathcal{I}}$ s.t. 
\begin{equation*}
\|\vect{c}_f-\vect{a}\|_2 \leq C\eta,
\end{equation*}
where $C > 0$ is an absolute constant.
\end{adjustwidth}
Furthermore, the total runtime complexity of Algorithm \ref{alg:main} is always\\ $\mathcal{O}\Big(\Big( s^5 D^2 \log^4 (M) K^4   \max \left\{d^4\log^4(s) \log^4({D^2M}), \log^2(\frac{D}{p}) \right\} \Big)\times \log \frac{\|\vect{c}_f\|_2}{\eta}\Big)$.  
  \label{cor:fou}
\end{corollary}

Note that the runtime dependance achieved by the corollary above scales sublinearly with $M$, quadratically in $D$, and at most polynomially in the parameter $d \leq D$ used to determine $\mathcal{I}$.  We also remind the reader that the BOS constant $K$ for the Fourier basis is $1$.  As a result, the $K$ dependence in the runtime complexity vanishes entirely when the BOPB in question is the multidimensional Fourier basis.\footnote{Though the resulting $\mathcal{O}\left( s^5 D^2 d^4 {\rm polylog}(MDs\|\vect{c}\|_2/\eta p) \right)$-runtime achieved by Corollary~\ref{cor:fou} for the multidimensional Fourier basis is strictly worse than the best existing noise robust and deterministic sublinear-time results for that basis \cite{iwen2013improved} (except perhaps when $s^3 d^4 \ll D^2$), we emphasize that it is achieved with a different and significantly less specialized grid $\mathcal{G}$ herein.}  Finally, there are also sublinear-time sparse transforms for one-dimensional Legendre polynomial systems \cite{hu2015rapidly}, though the theoretical results for sparse recovery therein require additional support restrictions beyond simple sparsity.  Thus, Corollary~\ref{cor:fou} can also be extended to restricted types of Legendre-sparse functions in order to achieve sublinear-in-$M$ runtimes.  A detailed development of such results is left for future work, however.

\subsection{A Comparison of Corollary~\ref{cor:fou} to Prior SFT Algorithms}
\label{equ:SFTcompareCoro1}

\begin{figure}
\begin{center} 
\resizebox{\hsize}{!}{
\begin{tabular}{|c|c|c|c|c|c|}
\hline
\textbf{SFT Method} & \textbf{Runtime Complexity}  & \textbf{Sampling Complexity} & \textbf{$\mathcal{I} \subseteq [M]^D$} & {\bf MB} & {\bf Error Guarantees} \\ \hline
Corollary~\ref{cor:fou} & $s^5 d^4 D^2 \cdot \log^5(s) \log^8({DM})$ & $s^3 d^4 D \cdot \log^4(s) \log^8({DM}) $ & Any & Yes & Uniform Exact\\ \hline
Thm 8 in \cite{iwen2013improved} & $s^2 D^4 \cdot \log^4(MD)$ & $s^2 D^4 \cdot \log^4(MD)$ & $[M]^D$ & No & Det. Uniform Best-$s$ Term\\ \hline
Thm 12 in \cite{morotti2017explicit} & $s^2 M^2 D^3 \cdot \log (M)$ & $s^2 M D^3 \cdot \log (M)$ & $[M]^D$ & No & Det. Uniform Best-$s$ Term\\ \hline
Thm 8 in \cite{iwen2013improved} & $s D^4 \cdot \log^4(MD)$ & $s D^4 \cdot \log^4(MD)$ & $[M]^D$ & No & Nonuniform Best-$s$ Term\\ \hline
Thm 3.5 in \cite{kapralov2016sparse} & $s 2^{\mathcal{O}(D^2)}D^{D+3} \cdot \log^{D+3} M$ & $s 2^{\mathcal{O}(D^2)} D \cdot \log M~*$ & $[M]^D$ & No & Nonuniform Best-$s$ Term\\ \hline
Thm 4.6 in \cite{kammerer2017high} & $s^3 M D^2 \cdot \log^4(s)  \log^2(DM) ~*$ & $s^3 M D \cdot \log^3(s) \log(DM) ~*$ & Any & No & Nonuniform Exact \\ \hline
{Thm 4.7 \cite{kammerer2017high} abt \cite{potts2016sparse}} & $s^5 D(s+M\log s) \cdot \log^3(s)  \log^3(DM) ~*$ & $s^5 M D \cdot \log^3(s) \log^3(DM)$ & Any & No & Nonuniform Exact \\ \hline
Alg in \cite{potts2016sparse} & $s^3 D + s^2 M D \cdot \log(sM)$~ $\dagger$ & $s^2 M D$~ $\dagger$ & Any & No & Empirical/Noise Robust \\ \hline
Mod. Alg \cite{potts2016sparse} & $s^3 D$~ $\dagger$ & $s D + M D$~ $\dagger$ & Any & No & Empirical/Noise Robust \\ \hline
Alg in \cite{choi2016multi} & {Avg. Case:}~~$s D \cdot \log (s)$ & Avg. Case:~~$s D$ & $[M]^D$ & No & Empirical/Low Noise \\ \hline
Alg in \cite{choi2019multiscale} & {Avg. Case:}~~$s D \cdot \log (s) \log (M)$ & Avg. Case:~~$s D \cdot \log (M)$ & $[M]^D$ & No & Empirical/Allows Noise \\ \hline
\end{tabular}}
\end{center}
\caption{\small A comparison of Corollary~\ref{cor:fou} with several existing SFT results from \cite{iwen2013improved,potts2016sparse,choi2016multi,kapralov2016sparse,choi2019multiscale,morotti2017explicit,kammerer2017high}.  To make the runtime and sampling complexity bounds easily comparable we have set all probabilities of failure $p$ for algorithms involving randomized constructions to be $p = \left(\frac{1}{DM} \right)^c$ for a universal constant $c \geq 1$ (except for \cite{kapralov2016sparse} which formulates its guarantees for $p = 0.02$), and have also used that $\log(M)$ and $\log({D^2M})$ are both $\mathcal{O}\left(\log({DM})\right)$ throughout.  We also note that $K = 1$ in the Fourier basis setting considered in the table.  To again aid in easier table comparisons we focussed on the recovery of functions with $\|\vect{c}_f\|_2 \leq c s$ and $\eta^{-1} \leq c$ for a universal constant $c \geq 1$ (representing machine precision) to simplify the $\log \left( \frac{\|\vect{c}_f\|_2}{\eta} \right)$ term in the runtime complexity of Corollary~\ref{cor:fou}.  Furthermore, we used that $\log(|\mathcal{I}|) \leq d \log(DM)$ which follows from the assumption herein that $N=|\mathcal{I}|={D \choose d} M^d \leq (DM)^d$.  Finally,~*~in the table denotes that we dropped $\log |\log(\cdot)|$ factors in the complexity bounds from \cite{kammerer2017high,kapralov2016sparse} to help save space in the table, and $\dagger$ indicates that \cite{potts2016sparse} formulates it's bounds under the assumption that $\sqrt{M} \lesssim s \lesssim M^D$ (as opposed to assuming that $s \lesssim M^D$ is just larger than some absolute constant).}
\label{fig:SFTcompfig}
\end{figure}

As mentioned above, the methods developed herein can be considered as generalizations of existing SFT techniques to BOP bases. Figure~\ref{fig:SFTcompfig} compares Corollary~\ref{cor:fou} in the case of the multidimensional Fourier basis to several existing SFT results for periodic functions of $D$-variables \cite{iwen2013improved,potts2016sparse,choi2016multi,kapralov2016sparse,choi2019multiscale,morotti2017explicit,kammerer2017high}.  The second and third columns of the table give the worst case runtime and sampling complexities of the methods, respectively, with the exception of the rows pertaining to \cite{choi2016multi,choi2019multiscale} which list average case complexities for generic signals.  Note that all of the quoted runtimes are sublinear in that they scale like $o(M^D)$ for sufficiently large $M$ and $D$.  The fourth column indicates whether each SFT paper considers improving its performance for smaller index sets $\mathcal{I} \subset [M]^D$ with $|\mathcal{I}| \ll M^D$, or not.  As indicated there, the majority of previous results only considered $\mathcal{I} = [M]^D$ with the notable exceptions of \cite{potts2016sparse,kammerer2017high}.  The fifth columns indicates if the methods allow for the easy extension of the methods and results to Mixed Bases (MB), or not.  As can be seen from the table, only Corollary~\ref{cor:fou} provides guarantees for, e.g., mixed Fourier/Chebyshev product bases at the expense of increasing $K$.

The final column in Figure~\ref{fig:SFTcompfig} summarizes the theoretical error guarantees proven about each method in the table.  There, ``Best-$s$ Term'' refers to methods that prove theoretical best-$s$ term approximation guarantees of the type considered by Cohen et al. in \cite{cohen2009compressed} for all periodic functions with Fourier series coefficients that decay rapidly enough.  Similarly, ``Exact'' therein refers to methods that provide strictly weaker guarantees regarding the recovery of functions which are exactly $s$-sparse in Fourier (i.e., of the form of \eqref{equ:ExactlySparsefunc}).\footnote{Best $s$-term approximation guarantees imply the exact recovery of exactly $s$-sparse functions.}  Note that such exactly $s$-sparse functions correspond to multivariate trigonometric polynomials with exactly $s$ nonzero terms in the SFT context.  Both of these types of guarantees can themselves be either ``Uniform'' (i.e., providing sampling sets of the type discussed below in \S\ref{sec:UniversalGrids} that work {\it for all} functions of the given class with high probability), or ``Nonuniform'' (i.e., providing sampling sets that work for {\it any one arbitrary} function of the given class with high probability, but not necessarily for all of them).  In addition, two of the results in \cite{iwen2013improved,morotti2017explicit} provide entirely deterministic and explicit sampling constructions with no probability of failure in achieving their respective approximation guarantees.  These two error guarantees are denoted with the prefix ``Det.'' for {\it Det}erministic.  

Finally, several other methods \cite{potts2016sparse,choi2016multi,choi2019multiscale} provide theoretical analysis which doesn't ultimately guarantee that they can approximate an arbitrary exactly sparse function \eqref{equ:ExactlySparsefunc} to machine precision with high probability.  All of these methods provide extensive empirical tests to demonstrate that such failures are indeed rare, however, and so are denoted in the last column of Figure~\ref{fig:SFTcompfig} by the term ``Empirical'' along with a description of the additive errors they are observed to tolerate on their function samples (``Noise Robust'' indicates good tolerance to arbitrary perturbations including random noise, ``Allows Noise'' indicates tolerance to random i.i.d. mean 0 noise, and ``Low Noise'' indicates a tolerance to small errors on the level of numerical roundoff).  We refer the readers to the original papers for additional details on their respective theoretical guarantees and empirical performance.

\subsection{Randomly Constructed Grids with Universal Approximation Properties}
\label{sec:UniversalGrids}

Fix a BOP basis $\mathcal{B}$ and sparsity level $s$.  We will call any set $\mathcal{G} \subset \mathcal{D}$ a {\em compressive sensing grid} if and only if $\exists$ a set of weights $\mathcal{W}$ s.t. $\forall$ $f:\mathcal{D} \rightarrow \mathbb{C}$ that are $s$-sparse in $\mathcal{B}$
$${\rm Algorithm~\ref{alg:main}~ with~ weights}~ \mathcal{W}~{\rm can ~recover} ~f~ {\rm from ~ its~evaluations~ on}~ \mathcal{G} $$
is true.  As mentioned above, our main results demonstrate the existence of relatively small compressive sensing grids by randomly constructing highly structured sets of points that are then shown to be compressive sensing grids with high probability.  We emphasize that our use of probability in this paper is entirely constrained to $(i)$ the initial choice of the grid $\mathcal{G}$ given a BOP basis $\mathcal{B}$ and sparsity level $s$, and to $(ii)$ the entirely independent and one-time initial choice of a set of random Gaussian weights $\mathcal{W}$ for use in \eqref{def:fk} (i.e., as part of the initialization phase for Algorithm~\ref{alg:main}).  Algorithm~\ref{alg:main} is entirely deterministic once both $\mathcal{G}$ and $\mathcal{W}$ have been chosen.

The compressing sensing grids $\mathcal{G}$ utilized herein will be the union of three distinct sets of points in $\mathcal{D}$.  The first set of points is the set $\mathcal{G}^E \subset \mathcal{D}$ which has already been introduced in Section~\ref{sec:CompSense} as the set of sampling points at which $f$ is evaluated in order to obtain $\vect{y^{\rm E}}$ in \eqref{equ:DefyE}.  This set of points is used in Algorithm~\ref{alg:main} in order to estimate the basis coefficients for the basis elements identified by Algorithms~\ref{alg:Entry} and~\ref{alg:Pairing}.  The second and third sets included in $\mathcal{G}$, $\mathcal{G}^I \subset \mathcal{D}$ and $\mathcal{G}^P \subset \mathcal{D}$, are utilized by Algorithm~\ref{alg:Entry} and Algorithm~\ref{alg:Pairing}, respectively.  
Here we will summarize the main ideas behind their construction (their precise definition is given in Section \ref{sec:SublinCompSense} below).
	
The set $\mathcal{G}^I$ is the union of $D$ grids $\mathcal{G}_j^I$, $j\in [D]$, where each $\mathcal{G}_j^I$ is designed to allow for the identification of the $j$-th entry of the energetic index vectors in $\mathcal{S}$ from \eqref{equ:ExactlySparsefunc}.  Each $\mathcal{G}_j^I$ set consists of all combinations of an appropriate quadrature set for the $j$-th entry (to allow for an approximation of one-dimensional integrals in that variable) crossed with a random sampling set for all other entries (necessary for another approximate numerical integration in these directions that singles out the  $j$-th entry of interest).  Viewed another way, each $\mathcal{G}_j^I$ set contains the points necessary to exactly integrate the functions of the $j$-th variable $x_j$ one obtains from $f$ in \eqref{equ:ExactlySparsefunc} after fixing the other $D-1$ variables to be random constants, for several different collections of random constants.

Similarly, the set $\mathcal{G}^P$ is also a union of $D-1$ grids $\mathcal{G}_j^P$, one for each $j\in [D] \setminus \{0\}$. Here the grids are designed to allow for the sequential buildup of the energetic index vectors contained in $\mathcal{S}$ from \eqref{equ:ExactlySparsefunc} component by component. More precisely, the samples in $\mathcal{G}_j^P$ will be used  to pair the first $j$ components of each element of $\mathcal{S}$ with its correct $(j+1)$-th component.  To generate the grid $\mathcal{G}_j^P$ we generate a set $W_j$ of independent samples of a random vector in $\mathbb{C}^j$ and a set $Z_j$ of independent samples of a random vector in $\mathbb{C}^{D-j}$. We then choose $\mathcal{G}_j^P:=W_j\times Z_j$.  The underlying idea is that this construction allows one to reduce the original $D$-dimensional problem to a $(j+1)$-dimensional problem with energetic index vectors given by the first $j+1$ components of the energetic index vectors $\mathcal{S}$ from the full problem.  This is achieved by using the product structure, and the random construction of $Z_j$ which allows us to approximately integrate over the last $D-j-1$ variables of $f$ from \eqref{equ:ExactlySparsefunc}.  Then we can use the randomness of $W_j$ to identify the active frequencies among all combinations of the $j$-dimensional component vectors identified in the previous sequential step, and the $(j+1)$-th components of $\mathcal{S}$ identified using $\mathcal{G}_{j+1}^I$ above.

As we saw in Section~\ref{sec:MainResults}, it turns out that $\mathcal{G} := \mathcal{G}^E \cup \mathcal{G}^I \cup \mathcal{G}^P$ will be a compressive sensing grid with high probability even when each component set is chosen to have a relatively small cardinality.  The vast majority of the remainder of this paper will be dedicated to proving this fact.  We will begin in Section \ref{Preliminaries} by introducing additional notation that is used throughout the rest of the paper, and by interpreting our function evaluations on $\mathcal{G}$,
\begin{equation} 
\vect{y} = (\vect{y^{\rm E}}, \vect{y^{\rm I}}, \vect{y^{\rm P}})^T \in \mathbb{C}^{m'_1 + m'_2 + m'_3}
\label{Def:SamplesFromf}
\end{equation}
as standard compressive sensing measurements.  Next, in Section~\ref{sec:proposed}, Algorithm~\ref{alg:main} is discussed in detail and the main theorem above is proven with the help of a key technical lemma (i.e., Lemma~\ref{lem:replCosamp}) that guarantees the accuracy of our proposed support identification method.  Lemma~\ref{lem:replCosamp} is then proven in Section~\ref{sec:Analysis}.  Finally, a numerical evaluation is carried out in Section \ref{Empirical Evaluation} that demonstrates that Algorithm~\ref{alg:main} both behaves as expected, and is robust to noisy function evaluations.  The paper then concludes after a short discussion concerning future work in Section~\ref{Future Work}.

\section{Preliminaries} \label{Preliminaries}
\setcounter{equation}{0}

In this section we introduce the notation that will be used in the rest of this paper as well as the problem for which we will develop our proposed algorithm. We denote by $\mathbb{N}$ the set of natural numbers, $\mathbb{R}$ the set of real numbers, and $\mathbb{C}$ the set of complex numbers.  Let $[N]: = \{0, 1, 2, \dots, N-1 \}$ for $N \in \mathbb{N}$.

\subsection{Notation and Preliminaries} \label{Notation and Problem Setting}

In this paper all letters in boldface (other than probability measures such as $\vect{\nu}$) will always represent vectors.  
Vectors whose entries are indexed by {\it index vectors} in, e.g., $[M]^D$ will be assumed to have their entries ordered lexicographically for the purposes of, e.g., matrix-vector multiplications. Thus, we say either $\vect{v}\in \mathbb{C}^{[M]^D}$ or $\vect{v}\in \mathbb{C}^{M^D}$ when we want to emphasize that each entry $v_{\vect{n}}$ of $\vect{v}$ is corresponding to its index vector $\vect{n}$, or when we perform, e.g.,   matrix-vector multiplications, respectively. We further define the $\ell_0$ pseudo-norm of a vector ${\vect v}$ by $\|\vect{v}\|_0 := |\{i :{v}_i\neq 0\}|$ where the index $i$ refers to the $i^{\rm th}$ entry of the vector (in lexicographical order).  If $v \in \mathbb{C}$ is a scalar then we will also use the $\ell_0$-notation to correspond to the indicator function defined by
\begin{equation}
\|v\|_0 := 
\begin{cases}
1 & \textrm{if}~v \neq 0\\
0 & \textrm{if}~v = 0
\end{cases}.\label{equ:l0asIndicator}
\end{equation}

We will consider functions $f: \mathcal{D} \rightarrow \mathbb{C}$ given in a BOS product basis expansion below so that
\begin{equation}
f(\vect{x}):= \sum_{\vect{n}\in \mathcal{I} \subseteq [M]^D } c_{\vect{n}} T_{\vect{n}}(\vect{x})
\label{def:f}
\end{equation}
where $\vect{c} \in  \mathbb{C}^{[M]^D}$.
We will further assume that $f$ is approximately sparse in this BOS product basis. That is, we will assume that there exists some index set $\mathcal{S} \subset \mathcal{I} $ for an a priori known index set $\mathcal{I}\subseteq [M]^D$ such that $\mathcal{S}$ has the property that both $(i)$ $|\mathcal{S}| = s \ll |\mathcal{I}| \leq M^D$, and that $(ii)$ the set of coefficients $\mathcal{C} := \{ c_{\vect n} ~\big |~ {\vect n} \in \mathcal{S} \} \subset \mathbb{C}$ dominates $f$'s $\ell_2$-norm in the sense that
\begin{equation*}
\sum_{\vect{n}\in \mathcal{S} \subset \mathcal{I}} |c_{\vect{n}}|^2 \gg \sum_{\vect{n} \in [M]^D \setminus \mathcal{S}} |c_{\vect{n}}|^2 =: \epsilon^2,
\end{equation*}
for a relatively small number $\epsilon \geq 0$.  We emphasize here that absolutely nothing about $\mathcal{S}$ is known to us in advance beyond the fact that it is a subset of $\mathcal{I}$, and has cardinality at most $s$.  We must learn the identity of its elements ourselves by sampling $f$.

Our analysis herein will focus on the case where $\mathcal{I}$ is given by 
$$\mathcal{I} := \left\{\vect{n}\in [M]^D ~\big|~ \|\vect{n}\|_0 \leq d \right\}$$ 
for some $d\leq D$ (cf. Section \ref{sec:CompSense}).  Note that this includes, for $d=D$, the special case where $\mathcal{I} = [M]^D$.  We will call the index vectors $\vect{n} \in\mathcal{S}$ {\em energetic}. Our goal is to recover $\mathcal{S}$ and the associated coefficients $\mathcal{C}$ as rapidly as possible using only evaluations/samples from $f$.  This will, in turn, necessitate that we sample $f$ at very few locations in $\mathcal{D}$.  In this paper we will mainly focus on providing theoretical guarantees for the case where $\epsilon=0$ (i.e., for provably recovering $f$ that are exactly $s$-sparse in a {BOS} product basis).  Numerical experiments in Section \ref{Empirical Evaluation} demonstrate that the method also works when $\epsilon > 0$, however.  We leave theoretical guarantees in the case of $\epsilon > 0$ for future consideration.

\subsection{Definitions Required for Support Identification}
\label{sec:SuppIDDefs}

As with most compressive sensing and sparse approximation problems we will see that identifying the function $f$'s support $\mathcal{S}$ is the most difficult part of recovering $f$.  As a result our proposed iterative algorithm spends the vast majority of its time in every iteration recovering as many energetic $\vect{n}=(n_0,n_1, \cdots,n_{D-1}) \in \mathcal{S}$ as it can.  Only after doing so does it then approximate a sparse vector $\vect{c}\in \mathbb{C}^{[M]^D}$ containing nonzero coefficients $c_{\vect{n}}$ for each discovered $\vect{n}\in \mathcal{S}$. Here, each $\vect{n}$ will be referred to as an {\it index vector} of an entry in $\vect{c}$.  Let ${\rm supp}(\vect{v}) \subseteq [M]^D$ represent the set of index vectors whose corresponding $v_{\vect{n}}$ entries are nonzero. We introduce the following notation in order to help explain our algorithm in the subsequent sections of the paper. 

For a given $\vect{v}\in \mathbb{C}^{[M]^D}$, $j \in [D]$, and $\widetilde{n} \in [M]$ the vector $\vect{v}_{j;\widetilde{n}} \in \mathbb{C}^{[M]^{D-1}}$ indexed by $\vect{k} \in [M]^{D-1}$ is defined by
\begin{align}
\left({v}_{j;\widetilde{n}}\right)_{\vect{k}}= 
   \begin{cases}
       v_{\vect{n}}, &\text{ if } \vect{n} = (k_0, \dots, k_{j-1},\widetilde{n},k_j,\dots, k_{D-2}) \\
       0 &\text{ otherwise } \\
   \end{cases}. \label{def:vecentryfixed}
\end{align}
Note that $\vect{v}_{j;\widetilde{n}}$ will only ever have at  most
\begin{equation}
N' := {D-1 \choose d - \| \widetilde{n} \|_0}M^{\min \left\{ d - \| \widetilde{n} \|_0, D-1 \right\}} \leq \left(\frac{e(D-1)M}{\max \{ d - 1, 1 \} } \right)^d
\label{equ:Dimvec1}
\end{equation}
nonzero entries if $v_{\vect{n}} = 0$ for all $\vect{n} \in [M]^D$ with $\| \vect{n} \|_0 > d$ by assumption.  Here, as throughout the remainder of the paper, we define ${p \choose q }$ to be $1$ whenever  $q \geq p$ or $q < 0$ (also recall the definition of $ \| \widetilde{n} \|_0$ from \eqref{equ:l0asIndicator} above).\footnote{The $\min \left\{ d - \| \widetilde{n} \|_0, D-1 \right\}$ in the exponent of the $M$ in \eqref{equ:Dimvec1} handles the case when $d = D$ and $\widetilde{n} = 0$.}
Similarly, for a given $\vect{v}\in \mathbb{C}^{[M]^D}$, $j \in [D]$, and $\vect{\widetilde{n}} \in [M]^{j+1}$ the vector $\vect{v}_{j;(\vect{\widetilde{n}},\cdots)} \in \mathbb{C}^{[M]^{D-j-1}}$ indexed by $\vect{k} \in [M]^{D-j-1}$ is defined by 
\begin{align}
\left({v}_{j;(\vect{\widetilde{n}},\cdots)}\right)_{\vect{k}}=
   \begin{cases}
       v_{\vect{n}},&\text{ if }\vect{n} = (\vect{\widetilde{n}},\vect{k}) \\
       0 &\text{ otherwise}\\
   \end{cases}. \label{def:vecprefixfixed}
\end{align}
The following lemma bounds the total number of nonzero entries that $\vect{v}_{j;(\vect{\widetilde{n}},\cdots)}$ can have given that $v_{\vect{n}} = 0$ whenever $\| \vect{n} \|_0 > d$.  Note that for $j=0$, $\vect{v}_{j;\widetilde{n}}=\vect{v}_{j;({\widetilde{n}},\cdots)}$ so that \eqref{equ:Dimvec1} follows as a special case.\\

\begin{lemma}
Let $\vect{v}\in \mathbb{C}^{[M]^D}$, $j \in [D]$, and $\vect{\widetilde{n}} \in [M]^{j+1}$ with $\| \vect{\widetilde{n}} \|_0 \leq d$.  Suppose that $v_{\vect{n}} = 0$ whenever $\| \vect{n} \|_0 > d$.  Then ${v}_{j;(\vect{\widetilde{n}},\cdots)}$ can have at most
\begin{equation}
\widetilde{N}_j :=  {D-j-1 \choose d - \| \vect{\widetilde{n}} \|_0}M^{\min \left\{ d - \| \vect{\widetilde{n}} \|_0, D-j-1 \right\}} \leq \left(\frac{e(D-j-1)M}{\max \{ d -j- 1, 1 \} } \right)^d
\label{lem:estNj}
\end{equation}
nonzero entries.

\end{lemma}

\begin{proof}
Since $v_{\vect{n}} = 0$ whenever $\| \vect{n} \|_0 > d$ it must be the case that $\left({v}_{j;(\vect{\widetilde{n}},\cdots)}\right)_{\vect{n}} = 0$ whenever $\vect{n} = (\vect{\widetilde{n}}, \vect{\widetilde{n}}')$ has $\| \vect{\widetilde{n}}' \|_0 > d - \| \vect{\widetilde{n}} \|_0$, where $\vect{\widetilde{n}}' \in \mathbb{C}^{D-j-1}$.  As a result, if $D-j-1 > d - \| \vect{\widetilde{n}} \|_0$ then there are at most ${D-j-1 \choose d - \| \vect{\widetilde{n}} \|_0}$ entry combinations left in $\vect{n}$ which can be nonzero, each of which can take on $M$ different values.  If, on the other hand, $d - \| \vect{\widetilde{n}} \|_0 \geq D-j-1$ then all of the remaining $D-j-1$ values of $\vect{n}$ can each take on $M$ different values.  
\end{proof}

Motivated by the definition of $\vect{v}_{j;\widetilde{n}}$ in \eqref{def:vecentryfixed} we further define 
$$\mathcal{I}_{j;\widetilde{n}}:= \left\{\vect{n}\in \mathcal{I}~\big|~ n_j=\widetilde{n} \right\} \subset [M]^D,$$ 
and denote the restriction matrix 
that projects vectors in $\mathbb{C}^{[M]^D}$ onto each $\mathcal{I}_{j;\widetilde{n}}$ (considered as a subset of $\mathbb{C}^{[M]^{D-1}}$) by $P_{j;\widetilde{n}} \in \{ 0,1 \}^{M^{D-1} \times M^D}$. 
That is, we consider each $P_{j;\widetilde{n}}$ matrix to have rows indexed by $\vect{l} \in [M]^{D-1}$, columns indexed by $\vect{k} \in [M]^D$, and entries defined by 
\begin{equation}
(P_{j;\widetilde{n}})_{\vect{l}, \vect{k}} := \begin{cases} 1 &\textrm{ if } \vect{k} = (l_0, \dots, l_{j-1},\widetilde{n},l_j, \dots, l_{D-2})\\
0 &\textrm{ otherwise }
\end{cases}.
\label{equ:ProjMatrix}
\end{equation}
As a result, we have that $P_{j;\widetilde{n}}\vect{v}=\vect{v}_{j;\widetilde{n}}$ for all $\vect{v} \in \mathbb{C}^{[M]^D}$. 

The fast support identification strategy we will employ in this paper will effectively boil down to rapidly approximating the norms of various $\vect{c}_{j;\widetilde{n}}$ and $\vect{c}_{j;(\vect{\widetilde{n}},\cdots)}$ vectors for carefully chosen collections of $\widetilde{n} \in [M]$ and $\vect{\widetilde{n}} \in [M]^{j+1}$.  This, in turn, will be done using as few evaluations of $f$ in \eqref{def:f} as possible in order to estimate inner products and norms of other proxy functions constructed from $f$.  As a simple example, note that $\| \vect{c} \|_2$ can be estimated by using samples from $f$ in order to approximate $\| f \| ^2_{L^2(\mathcal{D},\vect{\nu})}$ since 
$$\| f \| ^2_{L^2(\mathcal{D},\vect{\nu})}=\sum_{\vect{n}\in \mathcal{I} \subseteq [M]^D} |c_{\vect{n}}|^2.$$ 

A bit less trivially, for $j \in [D]$ and $\widetilde{n} \in [M]$ one can also define the function $\left\langle f, T_{j;\widetilde{n}} \right\rangle_{(\mathcal{D}_j,\nu_j)}:  \mathcal{D}'_j \rightarrow \mathbb{C}$ with domain $\mathcal{D}'_j:= \times_{k \neq j} \mathcal{D}_k$ by having $\left\langle f, T_{j;\widetilde{n}} \right\rangle_{(\mathcal{D}_j,\nu_j)} \left( \vect{w} \right)$ evaluate to
$$\int_{\mathcal{D}_j} f(w_0, \dots, w_{j-1}, z ,w_{j+1}, \dots, w_{D-2} ) \overline{T_{j;\widetilde{n}}(z)} ~d\nu_j(z)$$
for all $\vect{w} \in \mathcal{D}'_j$.
Let $\vect{\nu}'_j:=\otimes_{k \neq j} \nu_k$.  It is not too difficult to see that 
\begin{equation}
\left \| \langle f, T_{j;\widetilde{n}} \rangle_{(\mathcal{D}_j, \nu_j)} \right \|^2_{L^2(\mathcal{D}'_j,\vect{\nu}'_j)} = \sum_{\vect{n} \in \mathcal{I} \text{ s.t. } n_j=\widetilde{n}} |c_{\vect{n}}|^2 = \| \vect{c}_{j;\widetilde{n}} \|_2^2
\label{equ:EngEstEntryID}
\end{equation}
in this case.  Similarly, for some $j \in [D]$ and $\vect{\widetilde{n}} \in [M]^{j+1}$ one can define the function $\langle f, {T}_{j;\vect{\widetilde{n}} }\rangle_{\left( \times_{i\in[j+1]} \mathcal{D}_i,  \otimes_{i\in[j+1]} \nu_i  \right)}$ from $\mathcal{D}''_j := \times_{k > j} \mathcal{D}_k$ into $\mathbb{C}$ by letting $\langle f, T_{j;\vect{\widetilde{n}}} \rangle_{\left( \times_{i\in[j+1]} \mathcal{D}_i,  \otimes_{i\in[j+1]} \nu_i  \right)} (\vect{w})$ equal
\begin{equation*}
 \int_{\times_{i\in[j+1]} \mathcal{D}_i} f(\vect{z} ,w_{0}, \dots, w_{D-j-2} )\prod_{k \in [j+1]}  \overline{T_{k;\widetilde{n}_k}(z_k)} ~d\left(\otimes_{i\in[j+1]} \nu_i \right)(\vect{z})
\end{equation*}
for all $\vect{w} \in \mathcal{D}''_j$.  Let $\vect{\nu}''_j:=\otimes_{k > j} \nu_k$.  Analogously to the situation above we then have that
\begin{equation}
\left\|  \langle f, {T}_{j;\vect{\widetilde{n}} }\rangle_{L^2\left( \times_{i\in[j+1]} \mathcal{D}_i,  \otimes_{i\in[j+1]} \nu_i  \right)} \right\|^2_{L^2 \left( \mathcal{D}''_j, \vect{\nu}''_j \right)} ~=~ \| \vect{c}_{j;(\vect{\widetilde{n}},\cdots)} \|_2^2.
\label{equ:EngEstPairing}
\end{equation}
As we shall see below, both \eqref{equ:EngEstEntryID} and \eqref{equ:EngEstPairing} will be implicitly utilized in order to allow the estimation of such $\| \vect{c}_{j;\widetilde{n}} \|_2^2$ and $\| \vect{c}_{j;(\vect{\widetilde{n}},\cdots)} \|_2^2$ norms, respectively, using just a few nonadaptive samples from $f$.

\subsection{The Proposed Method as a Sublinear-Time Compressive Sensing Algorithm}
\label{sec:SublinCompSense}

Note that $\vect{c} \in \mathbb{C}^{[M]^D}$ from \eqref{def:f}, as well as its restrictions $\vect{c}_{j;\vect{\widetilde{n}}} \in \mathbb{C}^{[M]^{D-1}}$ and $\vect{c}_{j;(\vect{\widetilde{n}},\cdots)} \in \mathbb{C}^{[M]^{D-j-1}}$ for any $\vect{\widetilde{n}} \in [M]^{j+1}$, will all be at most $s$-sparse under the assumption that $\epsilon=0$.  As mentioned above, this means that recovering $f$ from a few function evaluations is essentially equivalent to recovering $\vect{c}$ using random sampling matrices.  Given this, the method we propose in the next section can also be viewed as a sublinear-time compressive sensing algorithm which uses a highly structured measurement matrix $A$ consisting of several concatenated random sampling matrices.  More explicitly, the measurements $\vect{y} \in \mathbb{C}^{m'_1 + m'_2 + m'_3}$ utilized by Algorithm~\ref{alg:main} below consist of function evaluations (i.e., recall \eqref{Def:SamplesFromf}) which can be represented in the concatenated form
\begin{equation}
\vect{y} = \begin{bmatrix} \vect{y^{\rm E}} \\ \vect{y^{\rm I}} \\ \vect{y^{\rm P}} \end{bmatrix} 
~=~ \begin{bmatrix} \Phi \vect{c} \\ A^I \vect{c} \\ A^P \vect{c} \end{bmatrix} ~=~  \begin{bmatrix} \Phi \\ A^I \\ A^P \end{bmatrix} \vect{c} ~=~ A\vect{c}
\label{equ:sublinCS}
\end{equation}
for subvectors $\vect{y^{\rm E}} \in \mathbb{C}^{m'_1}$, $\vect{y^{\rm I}} \in \mathbb{C}^{m'_2}$, $\vect{y^{\rm P}} \in \mathbb{C}^{m'_3}$ (recall \S\ref{sec:UniversalGrids}, and see below for technical details), and structured sampling matrices $\Phi \in \mathbb{C}^{m'_1 \times M^D}$, $A^I \in \mathbb{C}^{m'_2 \times M^D}$, and $A^P \in \mathbb{C}^{m'_3 \times M^D}$.

In \eqref{equ:sublinCS} the matrix $\Phi$ is a standard random sampling matrix with the RIP formed as per \eqref{def:A}.  It and its associated samples $\vect{y^{\rm E}} = \Phi \vect{c}$ are used to estimate the entries of $\vect{c}$ indexed by the index vectors contained in the identified energetic support set $T$ in line 13 of Algorithm~\ref{alg:main}.  The matrices $A^I$ and $A^P$ are both used for support identification.

In particular, the matrix $A^I$ is a random sampling matrix corresponding to the sampling set $\mathcal{G}^I$, which is constructed as follows.
	Let $\mathcal{U}_j := \left\{ u_{j,0}, \dots, u_{j,\mathcal{L}'_j-1} \right\} \subset \mathcal{D}_j$ be the set of $\mathcal{L}'_j$ points at which one can evaluate any given $\mathcal{B}_j$-sparse function $g: \mathcal{D}_j \rightarrow \mathbb{C}$ in the span of $\mathcal{B}_j$ in order to compute all $M$-inner products $\left\{ \langle g, T_{j;\widetilde{n}} \rangle \right\}_{\widetilde{n} \in [M]}$ in $\mathcal{O}(\mathcal{L})$-time.  Also, let $I_{\geq j}:  [D-1] \rightarrow \{0,1\}$ be the indicator function that is zero when $\kappa < j$, and one when $\kappa \geq j$.
	For each $j \in [D]$ we will then define $\mathcal{G}_j^I \subset \mathcal{D}$ to be the set of $m \mathcal{L}'_j$ randomly generated grid points given by 
	$$\vect{x}'_{j,\ell,k} = \left( (x_{j,\ell})_0, (x_{j,\ell})_1, \dots, (x_{j,\ell})_{j-1}, u_{j,k}, (x_{j,\ell})_{j}, \dots, (x_{j,\ell})_{D-2} \right) ~\forall (\ell,k) \in [m]\times[\mathcal{L}'_j],$$
	where each $(x_{j,\ell})_\kappa \in \mathcal{D}_{\kappa+I_{\geq j}(\kappa)}$ is an independent realization of a random variable $\sim \nu_{\kappa + I_{\geq j}(\kappa)}$ for all $\kappa \in [D-1]$.  We now take $\mathcal{G}^I$ to be the union of these sets so that
	$$\mathcal{G}^I := \bigcup_{j \in [D]} \mathcal{G}_j^I = \bigcup_{j \in [D]} \left\{ \vect{x}'_{j,\ell,k} \right\}_{(\ell,k) \in [m]\times[\mathcal{L}'_j]}.$$
	Finally, similar to \eqref{equ:DefyE}, we will also define $f$'s evaluations on $\mathcal{G}^I$ to be $\vect{y^{\rm I}} \in \mathbb{C}^{m'_2}$ where
	$$\vect{y^{\rm I}} =  f\left( \mathcal{G}^I \right) := \left( f\left(\vect{x}'_{0,0,0} \right), f\left(\vect{x}'_{0,0,1}\right),\dots,f\left(\vect{x}'_{D-1,m-1,\mathcal{L}'_{D-1}-1}\right)\right)^T.$$
	The measurements $\vect{y^{\rm I}}$ are
	 used to try to identify all of the energetic basis functions in each input dimension, i.e., the sets 
	\begin{equation}
	\mathcal{N}'_j := \left\{ \widetilde{n} \in [M] ~\big|~ \exists \vect{n} \in \mathcal{S} ~{\rm with}~ n_j =  \widetilde{n} \right\} \subseteq [M]
	\label{equ:TrueNjsupport}
	\end{equation}
	for each $j \in [D]$ (see Algorithm~\ref{alg:Entry}).

 The matrix $A^P$ is a random sampling matrix corresponding to the sampling set $\mathcal{G}^P$, whose precise definition will again be based on several different subsets for each $j \in [D] \setminus \{ 0 \}$.  For each fixed $(j,\ell,k) \in [D] \setminus \{ 0 \} \times [m_1] \times [m_2]$ let $\vect{w}_{j,\ell} \in W_j:=\times_{i \in [j+1]}\mathcal{D}_i$ and $\vect{z}_{j,k} \in Z_j:=\times^{D-1}_{i = j+1}\mathcal{D}_i$ be chosen independently at random according to $\otimes_{i\in[j+1]} \nu_i$ and $\otimes^{D-1}_{i = j+1} \nu_i$, respectively.\footnote{When $j = D-1$ the vector $\vect{z}_{j,k}$ is interpreted as a null vector satisfying $(\vect{w}_{j,\ell},\vect{z}_{j,k}) = \vect{w}_{j,\ell}$ $\forall(\ell, k)$.}  We then define $\mathcal{G}_j^P \subset \mathcal{D}$ to be the set of $m_1 m_2$ randomly generated grid points given by 
 	$$\mathcal{G}_j^P :=  \left\{ (\vect{w}_{j,\ell},\vect{z}_{j,k})~\big|~ (\ell,k) \in [m_1] \times [m_2] \right\} ~\forall j \in [D] \setminus \{ 0 \}.$$
 	As above, we now let $\mathcal{G}^P$ be the union of these sets so that
 	$$\mathcal{G}^P := \bigcup_{j \in [D] \setminus \{ 0 \}} \mathcal{G}_j^P$$
 	and consider $f$'s evaluations on $\mathcal{G}^P$, denoted by $\vect{y^{\rm P}} \in \mathbb{C}^{m'_3}$.  The resulting measurements
 	\[\vect{y^{\rm P}} =  f\left( \mathcal{G}^P \right) := \left( f\left(\vect{w}_{1,0},\vect{z}_{1,0} \right), f\left(\vect{w}_{1,0},\vect{z}_{1,1}\right),\dots,f\left(\vect{w}_{D-1,m_1-1},\vect{z}_{D-1,m_2-1}\right)\right)^T\] and its associated samples $\vect{y^{\rm P}} = A^P \vect{c}$ are then used in Algorithm~\ref{alg:Pairing} to help build up the estimated support set $T \subset [M]^D$ from the previously identified $\mathcal{N}'_j$-sets.  See \S\ref{sec:proposed} below for additional details.

The matrix $A^I$ above is built using the matrix Kronecker products $\widetilde{A}_{j} \otimes L_{j}$ for $j \in [D]$, where $\widetilde{A}_{j} \in \mathbb{C}^{ m \times [M]^{D-1} }$ is the random sampling matrix defined in \eqref{eqn:defA}, and $L_{j} \in \mathbb{C}^{\mathcal{L}'_j \times [M]}$ is a sampling matrix associated with $\mathcal{U}_j \subset \mathcal{D}_j$ from Section \ref{sec:UniversalGrids} defined as
\begin{equation}
\left(L_{j}\right)_{q,n}:=T_{j;n}(u_{j,q}), \qquad q\in[\mathcal{L}'_j] \text{ and } n\in[M].
\label{defL}
\end{equation}
The matrix $A^P$, on the other hand, is constructed using $B_j \otimes C_j$ for all $j \in [D] \setminus \{ 0 \}$ where each ${B}_j \in \mathbb{C}^{m_1 \times [M]^{j+1}}$ is the random sampling matrix defined below in \eqref{eqn:defB}, and each ${C}_j \in \mathbb{C}^{m_2 \times [M]^{D-j-1}}$ the random sampling matrix defined in \eqref{eqn:defC}.
In particular, we have that
\begin{equation}
A^I := \begin{bmatrix} \widetilde{A}_{0} \otimes L_{0} \\ \vdots \\ \widetilde{A}_{D-1} \otimes L_{D-1} \end{bmatrix}, ~{\rm and}~
A^P :=  \begin{bmatrix} B_{1} \otimes C_{1} \\ \vdots \\  B_{D-1} \otimes C_{D-1}  \end{bmatrix}.
\label{defSamMat}
\end{equation}
From a set of random Gaussian weights $\mathcal{W} = \left\{ g_{\ell}^k \right\}_{\ell \in [m], k \in [L]}$, a matrix $G\in \mathbb{C}^{L\times m}$ is defined as
\begin{equation}
\left(G\right)_{k,\ell}:=g^k_{\ell}, \qquad k\in[L]  \text{ and } \ell\in[m],
\label{defG}
\end{equation}
where $g^k_{\ell}$'s are the Gaussian weights from (\ref{def:fk}).

Briefly contrasting the proposed approach interpreted as a sublinear-time compressive sensing method via \eqref{equ:sublinCS} against previously existing sublinear-time algorithms for Compressive Sensing (CS) (see, e.g., \cite{gilbert2006sublinear,Gilbert:2007:OSF:1250790.1250824,doi:10.1137/100816705,iwen2014compressed,Gilbert:2017:FSR:3058789.3039872}), we note that no previous sublinear-time CS methods exist which utilize measurement matrices solely derived from general BOS random sampling matrices.  This means that the associated recovery algorithms developed herein can not directly take advantage of the standard group testing, hashing, and error correcting code-based techniques which have been regularly employed by such methods, making the development of fast reconstruction techniques and their subsequent analysis quite challenging.  Nonetheless, we will see that we can still utilize at least some of the core ideas of these methods by sublinearizing the runtime of one of their well known superlinear-time relatives, CoSaMP \cite{needell2009cosamp}.

\section{The Proposed Method} 
\label{sec:proposed}
\setcounter{equation}{0}

\begin{algorithm}[h]
\caption{Sublinearized CoSaMP}
\label{alg:main}
\begin{algorithmic}[1]
\Procedure{$\mathbf{Sublinear Recovery Algorithm}$}{}\\
{\textbf{Input: }}{Sampling matrices $\Phi,~ A^I,~A^P$ (implicitly, via the samples in $\mathcal{G}$ that determine their rows), samples $\vect{y^{\rm E}}=\Phi\vect{c}, ~ \vect{y^{\rm I}} = A^I \vect{c},~\vect{y^{\rm P}} = A^P \vect{c}$, a sparsity estimate $s$, and a set of i.i.d. Gaussian weights $\mathcal{W}$}\\
{\textbf{Output: }}{$s$-sparse approximation $\vect{a}$ of $\vect{c}$}
\State $\vect{a}^0=\vect{0}$ \hfill \{Initial approximation\}
\State $\vect{v}^{I} \gets \vect{y^{\rm I}}$, $\vect{v}^{P} \gets \vect{y^{\rm P}}$ 
\State $t\gets 0$
\Repeat
\State \{The next line calls Algorithm~\ref{alg:Entry} \dots \}
\State $\mathcal{N}_j ~\forall j \in [D] \gets$ {\bf EntryIdentification}($\vect{v}^{I}$, $\mathcal{W}$) \hfill \{Support identification step \# 1\}
\State \{The next line calls Algorithm~\ref{alg:Pairing} \dots\}
\State $\Omega \gets$ {\bf Pairing}($\vect{v}^{P}$, $\mathcal{N}_j ~\forall j \in [D]$) \hfill \{Support identification step \# 2\}
\State $T\gets \Omega\cup {\rm supp}(\vect{a}^{t})$ \hfill \{Merge supports\}
\State $\vect{b}_{T} \gets \Phi^{\dagger}_{T} \vect{y^{\rm E}}$ \hfill \{Local estimation by least-squares\}
\State $t\gets t+1$
\State $\vect{a}^t\gets \left(\vect{b}_{T} \right)_s$ \hfill \{Prune to obtain next approximation\}
\State $\vect{v}^{I}\gets \vect{y^{\rm I}} - A^I\vect{a}^{t}$, $\vect{v}^{P}\gets \vect{y^{\rm P}} - A^P\vect{a}^{t}$ \hfill \{Update current samples\} 

\Until{halting criterion true}
\EndProcedure
\end{algorithmic}
\end{algorithm}

In this section we introduce and discuss our proposed method. Roughly speaking, our algorithm can be considered as a greedy pursuit algorithm (see, e.g., \cite{daubechies2004iterative, foucart2011hard, needell2009cosamp, needell2010signal, zhang2011sparse}) with a faster support identification technique that takes advantage of the structure of BOS product bases. In particular, we will focus on the CoSaMP algorithm \cite{needell2009cosamp} herein.  Note that support identification is the most computationally expensive step of the CoSaMP algorithm.  Otherwise, CoSaMP is already a sublinear-time method for any type of BOS basis one likes.  Our overall strategy, therefore, will be to hijack the CoSaMP algorithm as well as its analysis by removing its superlinear-time support identification procedure and replacing it with a new sublinear-time version that still satisfies the same iteration invariant as the original.  See Algorithm \ref{alg:main} for pseudocode of our modified CoSaMP method.  Note that most its steps are identical to the original CoSaMP algorithm except for the two ``Support identification" steps, and the ``Update current samples" and ``halting criterion" lines.  Thus, our discussion will mainly focus on these three parts.  

Like CoSaMP, Algorithm~\ref{alg:main} is a greedy approximation technique which makes locally optimal choices during each iteration. In the $t$-th iteration, it starts with an $s$-sparse approximation $\vect{a}^{t}$ of $\vect{c}$ and then tries to approximate the at most $2s$-sparse residual vector $\vect{r}:=\vect{c}-\vect{a}^{t}$.  The two ``Support identification" steps begin approximating $\vect{r}$ by finding a support set $\Omega \subset \mathcal{I}$ of cardinality at most $2s$ which contains the set $\left \{\vect{n} ~\big|~ |{r}_{\vect{n}}|^2 \geq \frac{\|\vect{r}\|_2^2}{\alpha^2 s} \right \}$ (i.e., $\Omega$ contains the indices of the entries where most of the energy of $\vect{r}$ is located).  These support identification steps constitute the main modification made to CoSaMP in this paper and are discussed in more detail in Sections \ref{secEntryID} and \ref{secPairing} below.  After support identification, in the ``Merge supports" step, a new support set $T$ of cardinality at most $3s$ is then formed from the union of $\Omega$ with the support of the current approximation $\vect{a}^{t}$.  At this stage $T$ should contain the overwhelming majority of the important (i.e., energetic) index vectors in $\mathcal{S}$.  As a result, restricting the columns of the sampling matrix $\Phi$ to those in $T$ (or constructing them on the fly in a low memory setting) in order to solve for  $\vect{b}_{T} := {\rm argmin}_{\vect{u} \in \mathbb{C}^{|T|}} \|\Phi_T\vect{u}-\vect{y^{\rm E}}\|_2$ should yield accurate estimates for the true coefficients of $\vect{c}$ indexed by the elements of $T$, $\vect{c}_T$.\footnote{In practice, it suffices to  approximate the least-squares solution $b_T$ by an iterative least-squares approach such as Richardson's iteration or conjugate gradient \cite{Dahlquist:2008:NMS:1383510} since computing the exact least squares solution can be expensive when $s$ is large.  The argument of \citep{needell2009cosamp} shows that it is enough  to take three iterations for Richardson's iteration or conjugate gradient if the initial condition is set to $\vect{a}^{t}$, and if $\Phi$ has an RIP constant $\delta_{2s}<0.025$. In fact, both of these methods have similar runtime performance.}  The vector $\left(\vect{b}_T\right)_s$ restricting $\vect{b}_T$ to its $s$ largest-magnitude elements then becomes the next approximation of $\vect{c}$, $\vect{a}^{t+1}$. 

As previously mentioned, the main difference between Algorithm~\ref{alg:main} and CoSaMP is in the support identification steps. In the proposed method support identification consists of two parts:  ``Entry Identification" and ``Pairing".  For each of these steps we use a different measurement matrix, $A^I$ or $A^P$, respectively, as well as a different set of samples (either $\vect{v}^{I}$ or $\vect{v}^{P}$) from the current residual vector.  Thus, we need to update a total of three estimates every iteration:  $\vect{v}^{I}$, $\vect{v}^{P}$ and $\vect{a}^t$.  In the next two sections we review each of the two newly proposed support identification steps in more detail.

\subsection{Support Identification Step \# 1:  Entry Identification}
\label{secEntryID}

For each $j\in[D]$ the entry identification algorithm (see Algorithm \ref{alg:Entry}) tries to find the $j$-th entry of each energetic index vector $\vect{n}$ corresponding to a nonzero entry $r_{\vect{n}}$ in the $2s$-sparse residual vector $\vect{r}=\vect{c}-\vect{a}^{t}$. Note that for each $j$ this gives rise to at most $2s$ index entries in $[M]$.\footnote{Note that we are generally assuming herein that $2s < M$.  In the event that $2s \geq M$ one can proceed in at least two different ways.  The first way is to not change anything, and to simply be at peace with the possibility of, e.g., occasionally returning $\mathcal{N}_j = [M]$.  This is our default approach.  The second way is regroup the first $g \in \mathbb{N}$ variables of $f$ together into a new collective ``first variable'', the second $g$ variables together into a new collective ``second variable", etc., for some $g$ satisfying $M^g > 2s$.  After such regrouping the algorithm can then again effectively be run as is with respect to these new collective variables.}  We therefore define $\mathcal{N}^t_j $ to be the resulting set $\{n_j ~|~ \vect{n} \in {\rm supp}(\vect{r}) \}$ of size at most $2s$ for each $j\in [D]$, and note that $\mathcal{N}^0_j = \mathcal{N}'_j$ in \eqref{equ:TrueNjsupport}.  Note further that $\widetilde{n} \in \mathcal{N}^t_j$ if and only if $\| \vect{r}_{j;\widetilde{n}} \|_2 > 0$.  As a result we can learn $\mathcal{N}^t_j$ by approximating $\| \vect{r}_{j;\widetilde{n}} \|_2$ using $\left \| \langle h, T_{j;\widetilde{n}} \rangle_{(\mathcal{D}_j, \nu_j)} \right \|^2_{L^2(\mathcal{D}'_j,\vect{\nu}'_j)}$ via \eqref{equ:EngEstEntryID} as long as we know 
\begin{equation}
h(\vect{x}):=\sum_{\vect{n}\in {\rm supp}(\vect{r})} r_{\vect{n}}T_{\vect{n}}(\vect{x}).
\label{equ:defhres}
\end{equation} 
Whenever $\left \| \langle h, T_{j;\widetilde{n}} \rangle_{(\mathcal{D}_j, \nu_j)} \right \|^2_{L^2(\mathcal{D}'_j,\vect{\nu}'_j)}$ is larger than a threshold value $\tau$ (e.g., zero) for a particular choice of $j \in [D]$ and $\widetilde{n} \in [M]$, we could simply add $\widetilde{n}$ to $\mathcal{N}_j$ (our estimate of $\mathcal{N}^t_j$) in this case.

\begin{algorithm}[H]
\caption{Entry identification}
\label{alg:Entry}
\begin{algorithmic}[1]
	 \Statex \{The method works because ${\rm median}_{k\in[L]} \left| \langle h_{j;k}, T_{j;\widetilde{n}} \rangle_{(\mathcal{D}_j, \nu_j)}\right| \approx \| \vect{r}_{j;\widetilde{n}} \|_2$.  See 
	\Statex \eqref{equ:defhres} and \eqref{def:fk} for the definition of $h_{j;k}$.  For exactly $s$-sparse $\vect{c}$ one can use 
	\Statex $\tau = 0$ below.  More generally, one can select the largest $s$ estimates for $\mathcal{N}_j$.\}
\Procedure{$\mathbf{Entry Identification}$}{}\\
{\textbf{Input: }}{$\vect{v}^I$, and a set of i.i.d. Gaussian weights $\mathcal{W}$ for use in the $h_{j;k}$ below (see \eqref{def:fk})}\\
{\textbf{Output: }}{$\mathcal{N}_j$ for $j\in[D]$}
\For {$j =0 \to D-1$}
   \State $\mathcal{N}_j \leftarrow \emptyset$
   \For {$\widetilde{n} = 0\to M-1$}
       
	\State  {\bf if} \small ${\rm median}_{k\in[L]} \left| \langle h_{j;k}, T_{j;\widetilde{n}} \rangle_{(\mathcal{D}_j, \nu_j)}\right| > \tau$,  
{\bf then}
\State \qquad $\mathcal{N}_j \leftarrow \left\{ \widetilde{n} \right\} \cup \mathcal{N}_j$. 
	\State {\bf end if}
	\EndFor
\EndFor
\EndProcedure
\end{algorithmic}
\end{algorithm}

Of course we don't actually know exactly what $h$ is.  However, we do have access to samples from $h$ in each iteration in the form of $\vect{v}^I$ and $\vect{v}^P$.  And, as a result, we are able to approximate $\left \| \langle h, T_{j;\widetilde{n}} \rangle_{(\mathcal{D}_j, \nu_j)} \right \|^2_{L^2(\mathcal{D}'_j,\vect{\nu}'_j)} = \| \vect{r}_{j;\widetilde{n}} \|_2$ for any $j \in [D]$ and $\widetilde{n} \in [M]$ with the estimator
$${\rm median}_k \left| \langle h_{j;k}, T_{j;\widetilde{n}}\rangle_{(\mathcal{D}_j, \nu_j)} \right|$$ 
defined using \eqref{def:fk}.  In Section \ref{sec:entryId} we show that this estimator can be used to accurately approximate $\| \vect{r}_{j;\widetilde{n}} \|_2$ {\it for all $2s$-sparse residual vectors} $\vect{r} \in \mathbb{C}^{[M]^D}$ using only $\mathcal{O}\left(s \cdot K^2 d \mathcal{L}' \cdot {\rm polylog}(D,s,M,K)\right)$ universal samples from any given $\vect{r}$'s associated $h$-function in \eqref{equ:defhres} (i.e., the samples in $\vect{v}^I$).\footnote{Recall that $\mathcal{L}'$ represents the maximum number of function evaluations one needs in order to compute $\langle g, T_{j;\widetilde{n}} \rangle$ for all $\widetilde{n} \in [M]$ in $\mathcal{O}(\mathcal{L})$-time for any given $j \in [D]$, and $s$-sparse $g: \mathcal{D}_j \rightarrow \mathbb{C}$ in ${\rm span} \left\{ T_{j;m}~\big|~m\in [M] \right\}$.}  Furthermore, the estimator can always be computed in just $\mathcal{O}\left(s^2 \cdot K^2 d^2 \mathcal{L} \cdot {\rm polylog}(D,s,M,K)\right)$-time.

At this point it is important to note that managing to find each $\mathcal{N}^t_j$ exactly for all $j \in [D]$ still does not provide enough information to allow us to learn ${\rm supp}(\vect{r})$ efficiently when $d > 1$.  In general the most we learn from this information is that ${\rm supp}(\vect{r}) \subset \left(\times_{j\in [D]} \mathcal{N}_j \right) \bigcap \mathcal{I} \subset [M]^D$.  In the next ``Pairing" step we address this problem by iteratively pruning the candidates in $\times_{j\in [1]} \mathcal{N}_j, \times_{j\in [2]} \mathcal{N}_j, \dots, \times_{j\in [D]} \mathcal{N}_j$ down at each stage to the best $2s$ candidates for being a prefix of some element in ${\rm supp}(\vect{r})$.  As we shall see, the pruning in each ``Pairing" stage involves energy estimates that are computed using only the samples from $h$ in $\vect{v}^P$.  These ideas are discussed in greater detail in the next section.

\subsection{Support Identification Step \# 2:  Pairing}
\label{secPairing}

\begin{algorithm}[H]
\caption{Pairing}
\label{alg:Pairing}
\begin{algorithmic}[1]
\Procedure{$\mathbf{Pairing}$}{}\\
{\textbf{Input: }}{$\vect{v}^P = \left\{ h(\vect{w}_{j,\ell},\vect{z}_{j,k})  ~\big|~ j \in [D] \setminus \{0\}, \ell\in[m_1], k\in[m_2] \right\} $, $\mathcal{N}_j$ for $j\in[D]$}\\
{\textbf{Output: }}{$\mathcal{P}$}
\State $\mathcal{P}_0 \gets \mathcal{N}_0$
\For {$j =1 \to D-1$}
	\State \{This method works because $E_{j;(\vect{\widetilde{n}}\dots)} \approx \| \vect{r}_{j;(\vect{\widetilde{n}},\cdots)} \|_2^2$ below.\}
	\State $E_{j;(\vect{\widetilde{n}}\dots)} \gets \frac{1}{m_2}\sum_{k\in[m_2]} \left| \frac{1}{m_1}\sum_{\ell\in[m_1]} h(\vect{w}_{j,\ell},\vect{z}_{j,k}) \overline{{T}_{\vect{\widetilde{n}}}(\vect{w}_{j,\ell})}\right|^2$ $\forall \vect{\widetilde{n}}\in \mathcal{P}_{j-1} \times \mathcal{N}_j$.
	\State Create $\mathcal{P}_{j}$ containing each $\vect{\widetilde{n}}$ whose energy estimate $E_{j;(\vect{\widetilde{n}}\dots)}$ is in the $2s$-largest.
\EndFor
\State $\mathcal{P} \gets \mathcal{P}_{D-1}$
\EndProcedure
\end{algorithmic}
\end{algorithm}

Once all the $\mathcal{N}_j \subset [M]$ have been identified for all $j \in [D]$ it remains to match them together in order learn as many of the true length-$D$ index vectors in ${\rm supp}(\vect{r}) \subset [M]^D$ as possible.  To achieve this we begin by attempting to identify all the prefixes of length two, $\vect{\widetilde{n}} = (\widetilde{n}_0,\widetilde{n}_1) \in \mathcal{N}_0 \times \mathcal{N}_1$, which begin at least one element in the support of $\vect{r}$.  Similar to the ideas utilized above, we now note that $(\vect{\widetilde{n}},\vect{n}') \in {\rm supp}(\vect{r})$ for some $\vect{n}' \in [M]^{D-2}$ if and only if $\| \vect{r}_{j;(\vect{\widetilde{n}},\cdots)} \|_2^2 > 0$.  As a result, it suffices for us to use the samples from $h$ (recall \eqref{equ:defhres}) in $\vect{v}^P$ in order to compute $E_{j;(\vect{\widetilde{n}}\dots)} \approx \| \vect{r}_{j;(\vect{\widetilde{n}},\cdots)} \|_2^2$ in Algorithm~\ref{alg:Pairing} above.  The $2s$-largest estimates $E_{j;(\vect{\widetilde{n}}\dots)}$ are then used to identify all the prefixes of length $2$ which begin at least one element of ${\rm supp}(\vect{r})$.  Of course, this same idea can then be used again to find all length-3 prefixes of elements in ${\rm supp}(\vect{r})$ by extending the previously identified length-2 prefixes in all possible $\mathcal{O}(s^2)$ combinations with the elements in $\mathcal{N}_2$, and then testing the resulting length-3 prefixes' energies in order to identify the $2s$ most energetic such combinations, etc..  See Algorithm~\ref{alg:Pairing} above for pseudocode, \S\ref{sec:pairing} for analysis of these $E_{j;(\vect{\widetilde{n}}\dots)}$ estimators, and \S\ref{sec:paringExample} just below for a concrete example of the pairing process.

\begin{figure}[t!]
	\begin{center}
	\resizebox{4in}{!}{
\begin{tikzpicture}[scale=0.50]

\node at (-12.5,3) {$(0,\cdot,\cdot)$};
\node at (-12.5,2) {$(1,\cdot,\cdot)$};
\node at (-12.5,1) {$\vdots$};
\node at (-12.5,-0.5) {$(M-1,\cdot,\cdot)$};

\node at (-4.5,-2) {$ \mathcal{N}_0 \times [M] \times [M]$};
\node at (-4.5,-23) {$[M] \times [M] \times \mathcal{N}_2$};

\node at (-8.5,1.5) {Alg 2.};
\node at (-8.5,.9) {$\longrightarrow$};

\node at (-12.5,-2) {$[M] \times [M] \times [M]$};

\node at (-8.5,-2) {$\supset$};
\node at (-4.5,2.5) {$(3,\cdot,\cdot)$};
\node at (-4.5,1.5) {$(4,\cdot,\cdot)$};
\node at (-4.5,0.5) {$(11,\cdot,\cdot)$};

\node at (-12.5,-7.5) {$(\cdot,0,\cdot)$};
\node at (-12.5,-8.5) {$(\cdot,1,\cdot)$};
\node at (-12.5,-9.5) {$\vdots$};
\node at (-12.5,-11) {$(\cdot,M-1,\cdot)$};

\node at (-4.5,-8.5) {$(\cdot,5,\cdot)$};
\node at (-4.5,-9.5) {$(\cdot,7,\cdot)$};

\node at (-8.5,-9) {Alg 2.};
\node at (-8.5,-9.6) {$\longrightarrow$};

\node at (-12.5,-12.5) {$[M] \times [M] \times [M]$};

\node at (-4.5,-12.5) {$[M] \times \mathcal{N}_1 \times [M]$};
\node at (-8.5,-12.5) {$\supset$};

\node at (-12.5,-23) {$[M] \times [M] \times [M]$};
\node at (-8.5,-19.5) {Alg 2.};
\node at (-8.5,-20.1) {$\longrightarrow$};

\node at (-8.5,-23) {$\supset$};

\node at (-12.5,-18) {$(\cdot,\cdot,0)$};
\node at (-12.5,-19) {$(\cdot,\cdot,1)$};
\node at (-12.5,-20) {$\vdots$};
\node at (-12.5,-21.5) {$(\cdot,\cdot,M-1)$};

\node at (-4.5,-18.5) {$(\cdot,\cdot,6)$};
\node at (-4.5,-19.5) {$(\cdot,\cdot,8)$};
\node at (-4.5,-20.5) {$(\cdot,\cdot,100)$};

\node at (4.5,-3.5) {$\mathcal{N}_0 \times \mathcal{N}_1 \times [M]$};
\node at (8.5,-3.5) {$\supset$};
\node at (12.5,-3.5) {$\mathcal{P}_1 \times [M]$};

\node at (4.5,-23) {$\mathcal{P}_1 \times \mathcal{N}_2$};

\node at (8.5,0.5) {Alg. 3};
\node at (8.5,-0.1) {$\longrightarrow$};
\node at (8.5,-17.5) {Alg. 3};
\node at (8.5,-18.1) {$\longrightarrow$};

\node at (12.5,-23) {${\rm supp} (\vect{r})$};

\node at (4.5,-13.5) {$(3,5,6)$};
\node at (4.5,-14.5) {$(3,5,8)$};
\node at (4.5,-15.5) {$(3,5,100)$};
\node at (4.5,-16.5) {$(4,7,6)$};
\node at (4.5,-17.5) {$(4,7,8)$};
\node at (4.5,-18.5) {$(4,7,100)$};
\node at (4.5,-19.5) {$(11,5,6)$};
\node at (4.5,-20.5) {$(11,5,8)$};
\node at (4.5,-21.5) {$(11,5,100)$};

\node at (-4.5,-3) {$\left| \mathcal{N}_0 \right| =3$};
\node at (-4.5,-13.5) {$\left| \mathcal{N}_1 \right| =2$};
\node at (-4.5,-24) {$\left| \mathcal{N}_2 \right|=3$};

\node at (4.5,-4.5) {$\left|\mathcal{N}_0  \times \mathcal{N}_1 \right|=6$};
\node at (12.5,-4.5) {$\left| \mathcal{P}_1 \right|= 3$};

\node at (4.5,-24) {$\left|\mathcal{P}_1 \times \mathcal{N}_2 \right|=9$};
\node at (12.5,-24) {$| {\rm supp} (\vect{r}) | = 3$};

\node at (-8.5,-4.5) {(a)};
\node at (-8.5,-15) {(b)};
\node at (-8.5,-25.5) {(c)};
\node at (8.5,-6) {(d)};
\node at (8.5,-25.5) {(e)};

\node at (8.5,-23) {$\supset$};

\node at (12.5,1.5) {$(3,5,\cdot)$};
\node at (12.5,0.5) {$(4,7,\cdot)$};
\node at (12.5,-0.5) {$(11,5,\cdot)$};

\node at (12.5,-16.5) {$(3,5,6)$};
\node at (12.5,-17.5) {$(4,7,8)$};
\node at (12.5,-18.5) {$(11,5,100)$};
\node at (4.5,3) {$(3,5,\cdot)$};

\node at (4.5,2) {$(3,7,\cdot)$};
\node at (4.5,1) {$(4,5,\cdot)$};
\node at (4.5,0) {$(4,7,\cdot)$};
\node at (4.5,-1) {$(11,5,\cdot)$};
\node at (4.5,-2) {$(11,7,\cdot)$};

\draw  (-14.5,-11.5) rectangle (-10.5,-7);
\draw  (-14.5,-1) rectangle (-10.5,3.5);
\draw  (-14.5,-22) rectangle (-10.5,-17.5);
\draw  (-6.5,3.5) rectangle (-2.5,-1);
\draw  (-6.5,-7) rectangle (-2.5,-11.5);
\draw  (-6.5,-17.5) rectangle (-2.5,-22);
\draw  (2.5,-22) rectangle (6.5,-13);

\draw  (2.5,-2.5) rectangle (6.5,3.5);
\draw  (10.5,-22) rectangle (14.5,-13);
\draw  (10.5,-2.5) rectangle (14.5,3.5);
\end{tikzpicture}}
	\end{center}
	\caption{Description of Entry Identification ((a), (b), (c)) and Pairing ((d), (e))}
	\label{fig:suppID}
	
\end{figure}
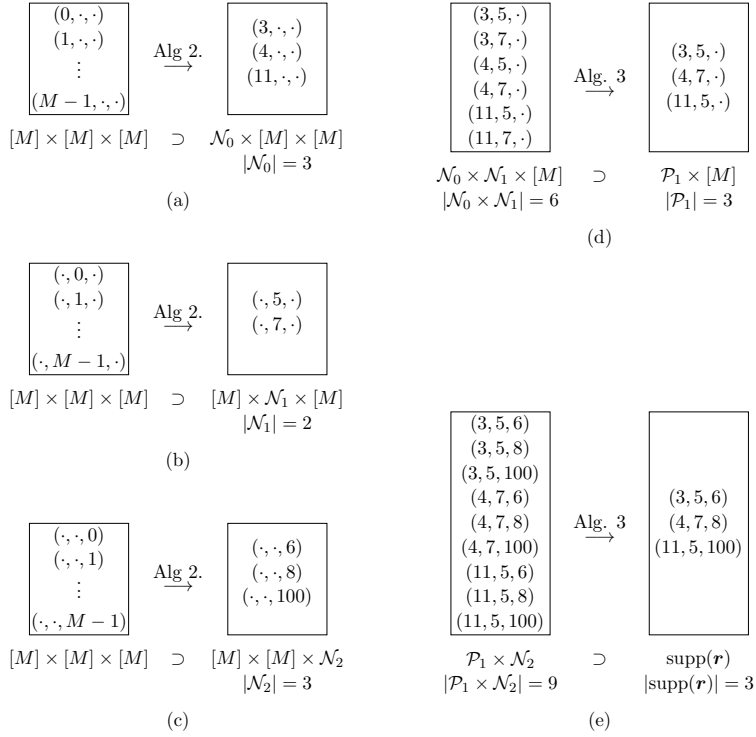

\subsubsection{An Example of Entry Identification and Pairing to Find Support}
\label{sec:paringExample}

Assume that $\vect{r} \in \mathbb{C}^{[M]^3}$ is three-sparse with a priori unknown energetic index vectors 
$${\rm supp}(\vect{r}) = \left\{ (3,5,6),~(4,7,8),~(11,5,100) \right\} \subset [M]^3$$ 
and corresponding nonzero coefficients $r_{(3,5,6)}$, $r_{(4,7,8)}$, and $r_{(11,5,100)}$.  We can further imagine that $M$ here is significantly larger than, e.g., $100$ so that computing all $M^3$ coefficients of $\vect{r}$ using standard numerical methods would be undesirable.  In this case, Algorithm~\ref{alg:Entry} aims to output the sets 
$$\mathcal{N}_0=\{3,4,11\},~\mathcal{N}_1=\{5,7\},~{\rm and }~\mathcal{N}_2=\{100,6,8\} \subset [M],$$  
i.e., the first, second, and third entries of each index vector in the support of $\vect{r}$, respectively. These sets are described in (a), (b) and (c) of Figure \ref{fig:suppID}.
Note that there are $18=3\times 2 \times 3$ possible index vectors which are consistent with the $\mathcal{N}_0,~\mathcal{N}_1,~{\rm and}~\mathcal{N}_2$ above.  Algorithm~\ref{alg:Pairing} is now tasked with finding out which of these $18$ possibilities are truly elements of ${\rm supp}(\vect{r})$ without having to test them all individually.\footnote{In this simple example we can of course simply estimate the energy for all 18 possible index vectors.  The three true index vectors in the support of $\vect{r}$ with nonzero energy would then be discovered and all would be well.  However, this naive approach becomes spectacularly inefficient for larger $D \gg 3$.}  

To identify ${\rm supp}(\vect{r})$ without having to test all $18$ index vectors in $\mathcal{N}_0 \times \mathcal{N}_1 \times \mathcal{N}_2$, the pairing process instead starts by estimating the energy of the $\left| \mathcal{N}_0 \right| \cdot \left| \mathcal{N}_1 \right|=6$ length-$2$ prefixes in $\mathcal{N}_0 \times \mathcal{N}_1$ which might begin an index vector in ${\rm supp}(\vect{r})$.  In the ideal case these energy estimates will reveal that only $3$ of these $6$ possible length-2 prefixes actually have any energy, 
$$\mathcal{P}_1 = \left\{ (3,5), (4,7), (11,5) \right\} \subset [M]^2$$ 
which is illustrated in (d) of Figure \ref{fig:suppID}.
In its next stage visualized in (e) of Figure \ref{fig:suppID}, the pairing process now continues by combining these three length-2 prefixes in $\mathcal{P}_1$ with $\mathcal{N}_2$ in order to produce $\left| \mathcal{P}_1 \right| \times \left| \mathcal{N}_2 \right| = 9$ final candidate elements potentially belonging to ${\rm supp}(\vect{r}) \subset [M]^3$.  Estimating the energy of these $9$ candidates then finally reveals the true identities of the index vectors in the support of $\vect{r}$.  

Note that instead of computing energy estimates for all $18$ possible support candidates in $\mathcal{N}_0 \times \mathcal{N}_1 \times \mathcal{N}_2$, the pairing process allows us to determine the correct support of $\vect{r}$ using only $15 = 6+9 < 18$ total energy estimates in this example.  Though somewhat underwhelming in this particular example, the improvement provided by Algorithm~\ref{alg:Pairing} becomes much more significant as the dimension $D$ of the index vectors grows larger.  When $|{\rm supp}(\vect{r})| = |\mathcal{N}_j| = s$ for all $j \in [D]$, for example, $\times_{j \in [D]} \mathcal{N}_j$ will have $s^D$ total elements.  Nonetheless, Algorithm~\ref{alg:Pairing} will be able to identify ${\rm supp}(\vect{r})$ using only $\mathcal{O}\left(s^2D\right)$ energy estimates in the ideal setting.\footnote{In less optimal settings one should keep in mind that Algorithm~\ref{alg:Pairing} only finds the most energetic entries in general, so that $\mathcal{P}\supset \left \{\vect{n} ~\big|~ |{r}_{\vect{n}}|^2 \geq \frac{\|\vect{r}\|_2^2}{\alpha^2 s} \right \}$ for a given $\alpha>1$.  This is why we need to apply it iteratively.}

\subsubsection{Comparison to Entry Identification and Pairing Methods in Prior Work}
\label{sec:EntryIDandPairingComp}

As previously mentioned, the support identification approach outlined above is similar in nature to the dimension incremental approaches utilized by both \cite{potts2016sparse, kammerer2017high} and \cite{choi2016multi,choi2019multiscale} in the Fourier setting.\footnote{In fact, this is unsurprising given that similar dimension incremental strategies have been proposed as far back as the 1970's in work related to recovering sparse algebraic polynomials \cite{zippel1979probabilistic}.}  For example, Algorithm 1 of \cite{potts2016sparse} also works by performing $D$ rounds of $(i)$ entry identification to find sets called $I^{(t)}$ in \cite{potts2016sparse} which are essentially identical the $\mathcal{N}_{t-1}$ sets found in our Algorithms~\ref{alg:main}~and~\ref{alg:Entry} above, followed by $(ii)$ a pairing step to build up sets called $I^{(1,\dots,t)}$ in \cite{potts2016sparse} which are essentially identical to the sets $\mathcal{P}_{t-1}$ in our Algorithm~\ref{alg:Pairing}.  As a result, the papers \cite{potts2016sparse, kammerer2017high} identify the support in Fourier of the Fourier sparse functions they seek to approximate via entry identification and pairing just as we do herein, at least superficially.  However, there are several important differences between our algorithm's support identification strategy and the one used by \cite{potts2016sparse, kammerer2017high}, the most crucial of which stems from the choice of the estimators used in \cite{potts2016sparse, kammerer2017high} versus those used herein in order to determine the $\mathcal{N}_{t-1}$/$I^{(t)}$ and $\mathcal{P}_{t-1}$/$I^{(1,\dots,t)}$ sets.  

As can be seen in line 7 of Algorithm~\ref{alg:Entry}, we utilize a median estimator to determine our $\mathcal{N}_{t-1}$ sets herein, whereas Algorithm 1 of \cite{potts2016sparse} effectively uses the magnitude of the Fourier coefficients of functions along the lines of $h_{j;k}$ from Algorithms~\ref{alg:Entry} (except with $h \leftarrow f$) as an estimator to identify their $I^{(t)}$ sets.  Similarly, the energy estimator we use in line 7 of Algorithm~\ref{alg:Pairing} to determine each $\mathcal{P}_{t-1}$ is instead replaced in Algorithm 1 of \cite{potts2016sparse} by the magnitude of the Fourier coefficients of functions along the lines of $h(\vect{w}_{j,\ell},\vect{z}_{j,k})$ from Algorithm~\ref{alg:Pairing} above, except, again, with $h$ equal to the original function $f$.  
It turns out that these differences in the choices of the estimators used to determine the entry and pairing sets is pivotal for many other features of each method.  In \cite{potts2016sparse, kammerer2017high} the use of Fourier coefficients of modified functions for estimators allows the authors to employ (multiple) rank-1 lattice techniques (see e.g. \cite{kuo2019function,gross2020deterministic}) in order to efficiently and accurately compute their estimators in the Fourier basis setting.  Herein, our choice of different and more general estimators is essential to us being able to achieve uniform recovery guarantees for a significantly more general class of BOPB's.  

Similarly, in \cite{choi2016multi,choi2019multiscale} the support of Fourier sparse functions is also recovered using a modified dimension incremental strategy.  In particular, therein $D$-dimensional frequencies with nonzero Fourier coefficients are again extended element-wise in order to recover sets analogous to $\mathcal{P}_{t-1}$ above using a multidimensional phase encoding approach (see, e.g., section 3.1.1 of \cite{gilbert2014recent} for an example of phase encoding in the one-dimensional setting).
Though fast, the resulting support identification approach in \cite{choi2016multi,choi2019multiscale} is potentially lossy in the sense that some elements of the $\mathcal{P}_{t-1}$ sets that the methods above might recover can be lost by the algorithms in \cite{choi2016multi,choi2019multiscale}.  As a result, the algorithms therein may not identify all of the Fourier support of some Fourier sparse functions.  To compensate for those worst case scenarios  \cite{choi2016multi,choi2019multiscale} suggest a ``tilting method'' where the frequencies are projected onto alternate coordinates rotated with a certain angle which can help to recover such difficult functions, and also employs and iterative reconstruction approach in order to try to compensate for Fourier coefficients which might cancel one another out in some of the utilized projections.

A second crucial difference between the algorithms in \cite{potts2016sparse, kammerer2017high,choi2016multi,choi2019multiscale} and the approach proposed herein is the fact that we apply our support identification methods to the residual functions $h$ above \eqref{equ:defhres} so that the overall approximation error reduces at a fixed rate each iteration following the iterative strategy of the CoSaMP algorithm instead of, e.g., directly applying them to the original function $f$ itself, or iterating in some other fashion.  This allows us to use the compressive sensing error analysis techniques developed in \cite{needell2009cosamp} to our benefit, whereas those techniques are not applicable to the algorithms as developed in \cite{potts2016sparse, kammerer2017high,choi2016multi,choi2019multiscale}.  We are now prepared to begin proving the promised recovery results for our method.

\subsection{A Theoretical Guarantee for Support Identification}
\label{sec:theorySuppId}

The following lemma and theorem show that our support identification procedure (i.e., Algorithm~\ref{alg:Entry}, followed by Algorithm~\ref{alg:Pairing}) always identifies the indexes of the majority of the energetic entries in $\vect{r}$. Consequently, the energy of the residual is guaranteed to decrease from iteration to iteration of Algorithm~\ref{alg:main}.  We want to remind readers that $\vect{r}$ is always $2s$-sparse since $\vect{c}$ and each $\vect{a}^{t-1}$ are $s$-sparse in the present analysis (i.e., $\epsilon=0$).\\

\begin{lemma}
Suppose that  $\left\{ T_{\vect{n}}~\big|~\vect{n} \in \mathcal{I} \subseteq [M]^D \right\}$ is a BOS where each basis function $T_{\vect{n}}$ is defined as per \eqref{def:T_n} , and $T_{j;0}\equiv 1$ for all $j \in [D]$.
Let $\mathcal{H}_{2s}$ be the set of all functions, $h:  \mathcal{D} \rightarrow \mathbb{C}$, in ${\rm span} \left\{ T_{\vect{n}}~\big|~\vect{n} \in \mathcal{I} \right\}$ whose coefficient vectors are $2s$-sparse, and let $\vect{r}_h \in \mathbb{C}^{\mathcal{I}}$ denote the $2s$-sparse coefficient vector for each $h \in \mathcal{H}_{2s}$.  Fix $p\in \left(0,1/2 \right)$, $1 \leq d \leq D$, $N={D \choose d} M^d$, and $\displaystyle K=\sup_{\vect{n} \text{ s.t.} \|\vect{n}\|_{0}\leq d } \|T_{\vect{n}}\|_{\infty}$.  Then, one can randomly select a set of i.i.d. Gaussian weights $\mathcal{W} \subset \mathbb{R}$ for use in \eqref{def:fk}, and also randomly construct entry identification and pairing grids, $\mathcal{G}^I \subset \mathcal{D}$ and $\mathcal{G}^P \subset \mathcal{D}$  (recall \S\textcolor{cyan}{\ref{sec:SublinCompSense}}), whose total cardinality $\left| \mathcal{G}^I \right| + \left| \mathcal{G}^P \right|$ is\\
\small{$\mathcal{O}\left(sD\mathcal{L}'K^2 \max\{\log^2(s)\log^2(DN),\log(\frac{D}{p})\} \text{\footnotesize$+$}s^3DK^4 \max \left\{\log^4(s)\log^2(N) \log^2(DN), \log^2(\frac{D}{p}) \right\} \right)$},\normalsize \\
such that the following property holds $\forall h \in \mathcal{H}_{2s}$ with probability greater than $1-2p$:
\begin{adjustwidth}{0.25in}{0.25in}
Let $\vect{v}^I_h \in \mathbb{C}^{m'_2}$ and $\vect{v}^P_h \in \mathbb{C}^{m'_3}$ be samples from $h \in \mathcal{H}_{2s}$ on $\mathcal{G}^I$ and $\mathcal{G}^P$, respectively.  If Algorithms \ref{alg:Entry} and \ref{alg:Pairing} are granted access to $\vect{v}^I_h$, $\vect{v}^P_h$, $\mathcal{G}^I$, $\mathcal{G}^P$, and $\mathcal{W}$ then they will find a set $\Omega \subset [M]^D$ of cardinality $2s$ in line $11$ of Algorithm~\ref{alg:main} such that 
\begin{equation*}
\|(\vect{r}_{h})_{\Omega^c}\|_2\leq 0.202 \|\vect{r}_h\|_2.
\end{equation*}
\end{adjustwidth}
Furthermore, the total runtime complexity of Algorithms \ref{alg:Entry} and \ref{alg:Pairing} is always\\ 
\small{ $\mathcal{O}\left( s^2 D \mathcal{L} K^2 \max\{ \log^2(s) \log^2(DN),\log(\frac{D}{p})\} \text{\footnotesize$+$}s^5 D^2 K^4 \max \left\{\log^4(s)\log^2(N) \log^2(DN), \log^2(\frac{D}{p}) \right\} \right)$}. \normalsize \\
\label{lem:replCosamp}
\end{lemma}

In Lemma~\ref{lem:replCosamp} above $\mathcal{L}'$ denotes the maximum number of points in $\mathcal{D}_j$ required in order to determine the value of $\langle g, T_{j;\widetilde{n}} \rangle_{(\mathcal{D}_j, \nu_j)}$ for all $\widetilde{n}\in[M]$ in $\mathcal{O}(\mathcal{L})$-time for any given $s$-sparse $g: \mathcal{D}_j \rightarrow \mathbb{C}$ in ${\rm span} \left\{ T_{j;m}~\big|~m\in [M] \right\}$, maximized over all $j \in [D]$.  See \S\ref{sec:MainResults} for a more in depth discussion of these quantities.  The reader is also referred back to \S\ref{sec:UniversalGrids} and \S\ref{sec:SublinCompSense} for a discussion of the entry identification and pairing grids, $\mathcal{G}^I$ and $\mathcal{G}^P$, mentioned in Lemma~\ref{lem:replCosamp}.  The proof of Lemma \ref{lem:replCosamp}, which is quite long and technical, is given in Section \ref{EntryId+pairing}.  Once Lemma \ref{lem:replCosamp} has been established, however, it is fairly straightforward to prove that Algorithm~\ref{alg:main} will always rapidly recover any function of $D$-variables which exhibits sparsity in a tensor product basis by building on the results in \cite{needell2009cosamp}.  We have the following theorem.\\

\begin{mytheorem}
Suppose that  $\left\{ T_{\vect{n}}~\big|~\vect{n} \in \mathcal{I} \subseteq [M]^D \right\}$ is a BOS where each basis function $T_{\vect{n}}$ is defined as per \eqref{def:T_n}, and $T_{j;0}\equiv 1$ for all $j \in [D]$.
Let $\mathcal{F}_{s}$ be the subset of all functions $f \in {\rm span} \left\{ T_{\vect{n}}~\big|~\vect{n} \in \mathcal{I} \right\}$ whose coefficient vectors are $s$-sparse, and let $\vect{c}_f \in \mathbb{C}^{\mathcal{I}}$ denote the $s$-sparse coefficient vector for each $f \in \mathcal{F}_{s}$.  Fix $p\in \left(0,1/3 \right)$, $1 \leq d \leq D$, $N={D \choose d} M^d$, $\displaystyle K=\sup_{\vect{n} \text{ s.t.} \|\vect{n}\|_{0}\leq d } \|T_{\vect{n}}\|_{\infty}$, and a precision parameter $\eta > 0$.  Then, one can randomly select a set of i.i.d. Gaussian weights $\mathcal{W} \subset \mathbb{R}$ for use in \eqref{def:fk}, and also randomly construct a compressive sensing grid, $\mathcal{G} \subset \mathcal{D}$, whose total cardinality $\left| \mathcal{G} \right|$ is\\
 \small $\mathcal{O}\left(sD\mathcal{L}'K^2 \max\{ \log^2(s)\log^2(DN),\log(\frac{D}{p})\} +s^3DK^4\max \left\{\log^4(s) \log^2(N) \log^2(DN), \log^2(\frac{D}{p}) \right\} \right)$, \normalsize\\ 
such that the following property holds $\forall f \in \mathcal{F}_{s}$ with probability greater than $1-3p$:
\begin{adjustwidth}{0.25in}{0.25in}
Let $\vect{y}$ consist of samples from $f \in \mathcal{F}_{s}$ on $\mathcal{G}$.  If Algorithm \ref{alg:main} is granted access to $\vect{y}$, $\mathcal{G}$, and $\mathcal{W}$, then it will produce an $s$-sparse approximation $\vect{a}$ such that 
\begin{equation*}
\|\vect{c}_f-\vect{a}\|_2 \leq C\eta.
\end{equation*}
Here $C > 0$ is an absolute constant.
\end{adjustwidth}
Furthermore, the total runtime complexity of Algorithm \ref{alg:main} is always\\
\small $\mathcal{O}\Big(\Big( s^2 D^2 \mathcal{L} K^2 \max\{ \log^2(s) \log^2(DN) ,\log(\frac{D}{p})\}+s^5 D^2 K^4   \max \big\{\log^4(s)\log^2(N) \log^2(DN), \log^2(\frac{D}{p}) \big\} \Big) \allowbreak \times \log \frac{\|\vect{c}_f\|_2}{\eta}\Big)$\normalsize when $\left| \mathcal{G} \right|$ is bounded as above.  
\label{mainThm}
\end{mytheorem}
As above, we refer the reader back to \S\ref{sec:MainResults} for a discussion of the quantities $\mathcal{L}'$ and $\mathcal{L}$ appearing in Theorem~\ref{mainThm}, as well as to \S\ref{sec:UniversalGrids} for more information on the compressive sensing grid $\mathcal{G} \subset \mathcal{D}$ mentioned therein.

\begin{proof}
In order to analyze the support identification step, we replace Lemma 4.2 in \citep{needell2009cosamp} by Lemma \ref{lem:replCosamp}, and then we obtain
\begin{equation*}
\|\vect{c}_f-\vect{a}^{t+1}\|_2 \leq 0.5 \|\vect{c}_f-\vect{a}^{t}\|_2
\end{equation*}
for each iteration $t\geq 0$, which is the same as in  Theorem 2.1 in \cite{needell2009cosamp} provided that $f$ is exactly $s$-sparse and samples are not noisy. Except for the support identification step(s), Algorithm \ref{alg:main} agrees with CoSaMP, so that Lemmas 4.3 -- 4.5 in \cite{needell2009cosamp} still directly apply to Algorithm~\ref{alg:main}.  After $\mathcal{O}\left(\log \frac{\|\vect{c}_f\|_2}{\eta}\right)$ iterations, we see that the $s$-sparse approximation $\vect{a}$ therefore satisfies
\begin{equation*}
\|\vect{c}_f-\vect{a}\|_2 \leq C\eta.
\end{equation*}
Since the runtime complexity of the support identification steps and the sample update process in each iteration, the total running time arises from multiplying it with the number of iterations. The number of sample points, $m'_1$, $m$, $m_1$ and $m_2$ used to define the matrices $\Phi$, $\widetilde{A}_{j}$, $B_j$ and $C_j$ discussed in \S\ref{sec:SublinCompSense} are all chosen to ensure that these resulting measurement matrices have the RIP.  Thus, the samples of $f$ in $\vect{y}$ can be reused over as many iterations as needed. Updating the samples of each residual function for the entry identification or pairing causes an extra $\mathcal{O}({sDm'_2})$ and $\mathcal{O}({sDm'_3})$ computations respectively which gives rise to a $D^2$ factor instead of $D$ in the first term of runtime complexity. The probability of successful recovery for all $f \in \mathcal{F}_{s}$ is obtained by taking the union bound over the failure probability $p$ of $\Phi$ having $\delta_{2s}<0.025$ via Theorem~\ref{thm:BOS_RIP} together with the failure probability $2p$ of Lemma \ref{lem:replCosamp}.
\end{proof}

We are now prepared to begin the process of proving Lemma~\ref{lem:replCosamp}.


\section{Analysis of the Support Identification} \label{sec:Analysis}
\setcounter{equation}{0}

In this section, we analyze the performance of the sublinear-time support identification technique proposed in Algorithms \ref{alg:Entry} and \ref{alg:Pairing}.  First, we show in Section \ref{sec:entryId} the success of the entry identification step. Indeed, Theorems \ref{Thm:FiniteSetMedtoUn} and \ref{Thm:FiniteSetMedtoUn2}  show under the RIP assumption that certain one-dimensional proxy functions allow us to identify the entry with large corresponding coefficients. Lemma \ref{Thm:FiniteSetMedtoUn3} then estimates the necessary sample complexity. In Section \ref{sec:pairing}, we analyze the pairing step showing that it works uniformly for any $2s$-sparse functions in Theorem \ref{thm:pairing}. Finally, in Section \ref{EntryId+pairing}, we complete the proof of the Lemma \ref{lem:replCosamp} providing the complete result for the proposed support identification method by combining/using the results of Sections \ref{sec:entryId} and \ref{sec:pairing}.

\subsection{Entry Identification}
\label{sec:entryId}

In this section we aim to find $\mathcal{N}^t_j$ containing the $j$-th entries of the index vectors of the nonzero transform coefficients of $h$ for all $j\in[D]$, which is done by Algorithm \ref{alg:Entry}. Define $[D]':=[D]\setminus \{j\}$.

Assume without loss of generality that the number $L \in \mathbb{N}$ of proxy functions is odd.  Choose $\mathcal{X}_j := \{\vect{x}_{j,\ell}\}_{\ell \in [m]}$ where each $\vect{x}_{j,\ell}$ is chosen independently at random from $\mathcal{D}'_j$ according to $d\vect{\nu}'_j$.  
Also, choose $\{{g}_1^{k}, \cdots, {g}_m^{k}\}_{k\in[L]}$ where each $g^k_{\ell}$ is an i.i.d. standard Gaussian variable $\sim \mathcal{N}(0,1)$, which forms $\mathcal{W}$ introduced in Section \ref{sec:UniversalGrids}. We define a function $h_{j;k}: \mathcal{D}_j \rightarrow \mathbb{C}$ of one variable in $\mathcal{D}_j$ as follows,
\begin{equation}
h_{j;k}(x):=\frac{1}{\sqrt{m}}\sum_{\ell\in[m]} g^k_{\ell} h([x,\vect{x}_{j,\ell}]) =\frac{1}{\sqrt{m}}\sum_{\ell\in[m]} g^k_{\ell} \sum_{\vect{n}\in {\rm  supp}(\vect{r})}  r_{\vect{n}}T_{j;n_j}(x)\prod _{i\in [D]'} T_{i;n_{i}}(x_{j,\ell})_i,
\label{def:fk}
\end{equation}
and $[x,\vect{x}_{j,\ell}]$ is the vector obtained by inserting the variable $x$ between the entries of $\vect{x}_{j,\ell}$ indexed by $[j]$ and $\{j,j+1\cdots, D-2\}$, i.e.,
\begin{equation*}
[x,\vect{x}_{j,\ell}]:=\left( (x_{j,\ell})_0, (x_{j,\ell})_1, \cdots, (x_{j,\ell})_{j-1},~ x~, (x_{j,\ell})_{j}, \cdots, (x_{j,\ell})_{D-2} \right).
\end{equation*}
For the sake of simplicity, we let $T_{\check{\vect{n}}}(\check{\vect{x}}):=\prod _{i\in [D]'} T_{i;\check{n}_{i}} (\check{x}_i)$ with 
$\check{\vect{n}}\in [M]^{D-1}$
and
$\check{\vect{x}} \in \mathcal{D}'_j$. 

We choose $m$ large enough above to form the normalized random sampling matrix $\widetilde{A}_{j} \in \mathbb{C}^{ m \times [M]^{D-1} }$ for each $j \in [D]$, defined as
\begin{equation}
\left( \widetilde{A}_{j} \right)_{\ell, \check{\vect{n}}}:= 
\frac{1}{\sqrt{m}} {T}_{{\check{\vect{n}}}}
({{\vect{x}}_{j,\ell}}), ~ \ell \in [m] ,~ \check{\vect{n}} \in [M]^{D-1},
\label{eqn:defA}
\end{equation}
so that each one has a restricted isometry constant $\delta_{2s}$ of at most $\delta$ with high probability in its restricted form. Here, the restricted form is introduced by 
eliminating the columns of the full $\widetilde{A}_j$ indexed by vectors ${\vect{n}} \notin \mathcal{I}_{j;\widetilde{n}}$. 
 To explain further, we denote $\widetilde{A}_j P_{j;\widetilde{n}} \vect{r}=\widetilde{A}_j \vect{r}_{j;\widetilde{n}}$ where $P_{j;\widetilde{n}}$ is a restriction matrix defined in  (\ref{equ:ProjMatrix}). This comes from the inner product in line $10$ of Algorithm \ref{alg:Entry} which is calculated by using the evaluations of $h_{j,k}$ at $\mathcal{U}_j$ defined in Section  \ref{sec:UniversalGrids}. Then, the inner product can be written as $\left( G \widetilde{A}_j P_{j;\widetilde{n}} \vect{r} \right)_k$ where $G$ is defined as in (\ref{defG}). In other words, the evaluations of $h$ at $[u_{j,k},\vect{x}_{j,\ell}]$ from $\mathcal{G}^{I}_j$ in Section \ref{sec:UniversalGrids} are utilized to compute the inner product.
 Then, the matrix-vector multiplication $\widetilde{A}_j \vect{r}_{j;\widetilde{n}}$ can be considered in its restricted form by eliminating the columns of $\widetilde{A}_j$ and elements of $\vect{r}_{j;\widetilde{n}}$ which are zero due to their corresponding index vectors not belonging to $\mathcal{I}_{j;\widetilde{n}}$. The resulting restricted matrix $\widetilde{A}_j$ has the size $m \times N'$ where $N'$ is bounded above in (\ref{equ:Dimvec1}) so that $\widetilde{A}_j$ has the restricted isometry constant $\delta_{2s}$ mentioned.
The advantage of forming RIP matrices in this fashion is that it allows us to analyze the different iterations of Algorithm~\ref{alg:main} repeatedly with the same RIP matrices. For a discussion about the number of measurements needed to ensure that  $\widetilde{A}_j$ satisfies the RIP, see the following lemma (which is a simple consequence of Theorem \ref{thm:BOS_RIP}).

	\begin{lemma} 
		Let $\widetilde{A}_j\in \mathbb{C}^{m\times N'}$ be the random sampling matrix as in (\ref{eqn:defA}) in its restricted form. 
		If, for $\delta, p \in (0,1)$, 
		\begin{equation*}
		m \geq a K^2 \delta^{-2} s \max \left\{ d  \log^2(4s) \log(9m) \log\left( \frac{8e(D-1)DM}{d}\right), \log\left(\frac{D}{p} \right) \right\},
		\end{equation*}
		then with probability at least $1-p$, the restricted isometry constant $\delta_s$ of  $\widetilde{A}_j$ satisfies $\delta_s \leq \delta$ for all $j\in[D]$. The constant $a>0$ is universal.
		\label{lem:RIPforD}
	\end{lemma} 

As we will show, more than half of the proxy functions, $\{h_{j;k}\}_{k\in [L]}$ are guaranteed with high probability to have $\|h_{j;k}\|_{L_2(\mathcal{D}_j,\nu_j)}$ bounded above by $\|\vect{r}\|_2$ up to some constant, and also $\left| \langle h_{j;k}, T_{j;\widetilde{n}}\rangle_{(\mathcal{D}_j,\nu_j)}\right|$ bounded above and below by $\|\vect{r}_{j;\widetilde{n}}\|_2$ up to some constants for  all $ \widetilde{n}\in [M]$ and $j\in[D]$. 

To show this we consider the indicator variable
 $E_{h,j,\widetilde{n},k}$ which is $1$ if and only if all three of
\begin{enumerate}
	\item $\| h_{j;k} \|^2_{L^2(\mathcal{D}_j,\nu_j)} \leq \alpha' \| \vect{r} \|_2^2$ for the absolute constant $\alpha'$ defined in Lemma~\ref{lem:fkPresTotalEng}, 
	\item ${\frac{9}{4}}\| \vect{r}_{j;\widetilde{n}} \|_2 \geq \left| \langle h_{j;k}, T_{j;\widetilde{n}} \rangle_{(\mathcal{D}_j, \nu_j)} \right| \geq \frac{\sqrt{23}}{12} \| \vect{r}_{j;\widetilde{n}} \|_2 $, and
	\item the vector of Gaussian weights $\vect{g}^k \in \mathbb{R}^m$ satisfying $ \frac{1}{2} m \leq \| \vect{g}^k  \|^2_2 \leq \frac{3}{2}  m$,
\end{enumerate}
are simultaneously true, and 0 otherwise.

The proof will proceed as follows. Lemmas \ref{lem:fkPresTotalEng} and \ref{lem:fkInnerProfBig} together with the bound on $\|\vect{g}^k\|_2$ through Bernstein's inequality imply that the probability of each $E_{h,j,\widetilde{n}, k}$ being $1$ is greater than $0.5$. Combining this with Chernoff bound, the deviation of $\sum_{k\in[L]} E_{h,j,\widetilde{n},k}$ below its expectation shows exponential decay in its distribution. As a result, with sufficiently many proxy functions, i.e., sufficiently large $L$, the probability that $\sum_{k\in [L]} E_{h,j,\widetilde{n},k}< L/2$ for all $(h,j,\widetilde{n}) \in \mathcal{H}\times [D] \times [M]$ becomes very small for any finite set of functions $\mathcal{H} \subset \mathcal{H}_{2s}$ whose BOS coefficient vectors are $2s$-sparse (see Theorem \ref{Thm:FiniteSetMedtoUn}).  The number $L$ logarithmically depends on $DM|\mathcal{H}|$. In order to get the desired properties for all $2s$-sparse functions satisfying our support assumption, the finite function set $\mathcal{H}$ is taken as $\mathcal{H}^{\epsilon}$ with corresponding coefficient vector set $\mathcal{R}^{\epsilon}$ which is an $\epsilon$-cover over all $\ell_2$-normalized $2s$-sparse vectors in $\mathbb{C}^N$ in Theorem \ref{Thm:FiniteSetMedtoUn2}. Thus, with high probability, the desired properties hold uniformly, i.e., for all functions of interest, and for all $ j\in[D]$ and $\widetilde{n}\in[M]$.

The following lemma bounds the energy of the proxy functions.


\begin{lemma} Suppose that $\vect{r} \in \mathbb{C}^{N}$ is $2s$-sparse, and the restricted isometry constant $\delta_{2s}$ of $\widetilde{A}_j$ satisfies $\delta_{2s}\leq\delta$ for all $j\in[D]$ where $\delta\in(0,7/16)$.  Then, for each $k\in[L]$, there exists an absolute constant $\alpha' \in \mathbb{R}^+$ such that 
\begin{equation}
\mathbb{P} \left[ \| h_{j;k} \|^2_{L^2(\mathcal{D}_j,\nu_j)} \geq \alpha' ||\vect{r}||_2^2 \right] \leq .025
\label{lem:equ:fkPresTotalEng}
\end{equation}
for all $j\in[D]$.
\label{lem:fkPresTotalEng}
\end{lemma}

\begin{proof}
Consider the random sampling point set $\mathcal{X}_j=\left\{ \vect{x}_{j,\ell}~|~ \ell\in[m]\right\}$ to be fixed for the moment. We begin by noting that

\begin{align}
\| h_{j;k} \|^2_{L^2(\mathcal{D}_j,\nu_j)} 
=& \int_{\mathcal{D}_j} |h_{j;k}(x)|^2 d\nu_j(x) \nonumber \\
=&\int_{\mathcal{D}_j} \left| \frac{1}{\sqrt{m}}\sum_{\ell\in[m]} g^k_{\ell} h(\left[x,\vect{x}_{j,\ell}\right]) \right|^2 d\nu_j(x) \nonumber \\
=&\frac{1}{m} \sum_{\ell, \ell' \in [m]} g^k_{\ell} g^k_{\ell'} \int_{\mathcal{D}_j} h([x,\vect{x}_{j,\ell}]) \overline{h([x,\vect{x}_{j,\ell'}])} ~d\nu_j(x) \nonumber \\
=&\frac{1}{{m}}\sum_{\ell,\ell' \in [m]}  g^k_{\ell} g^k_{\ell'} \sum_{\vect{n}, \vect{n}' \in {\rm supp}(\vect{r})}  ~r_{\vect{n}} \overline{r_{\vect{n}'}}  \prod _{i\in [D]'} T_{i;n_{i}}(x_{j,\ell})_i 
\prod _{i'\in [D]'} \overline{T_{i';n'_{i'}}(x_{j,\ell'})_{i'} }\\
& \times \int_{\mathcal{D}_j} T_{j;n_j}(x) \overline{T_{j;n'_j}(x)} ~d\nu_j(x) \nonumber\\
=&\frac{1}{{m}}\sum_{\ell,\ell' \in [m]} g^k_{\ell} g^k_{\ell'} \left( \sum_{\widetilde{n}\in[M]} \sum_{\substack{{\vect{n},\vect{n}'\text{ s.t.}}\\{n_j=n'_j=\widetilde{n}}}} r_{\vect{n}} \overline{r_{\vect{n}'}} \prod _{i\in [D]'} T_{i;n_{i}}(x_{j,\ell})_i \prod _{i'\in [D]'} \overline{T_{i';n'_{i'}}(x_{j,\ell'})_{i'} } \right) \nonumber\\
=&\sum_{\widetilde{n}\in[M]}   \sum_{\ell,\ell' \in [m]} g^k_{\ell} g^k_{\ell'} \left({\widetilde{A}_{j}\vect{r}_{j;\widetilde{n}}} \right)_{\ell} \overline{\left({\widetilde{A}_{j}\vect{r}_{j;\widetilde{n}}} \right)_{\ell'}}, \label{equ:forfkExpect}
\end{align}

where $\widetilde{A}_{j} \in \mathbb{C}^{m \times N'}$ is the restricted random sampling matrix from (\ref{eqn:defA}) and $\vect{r}_{j;\widetilde{n}} \in \mathbb{C}^{N'}$ is the restricted vector from (\ref{def:vecentryfixed}).  Thus, we can see that
 \begin{equation*} 
 \| h_{j;k} \|^2_{L^2(\mathcal{D}_j,\nu_j)} ~=~ \sum_{\widetilde{n}\in[M]}  \big| \langle {\widetilde{A}_{j}\vect{r}_{j;\widetilde{n}}}, \vect{g}^k \rangle \big|^2 ~=~ \left\| (\vect{g}^k)^* \widetilde{A}_{j} R \right\|^2_2 ~=~ \left\| \left(\widetilde{A}_{j} R \right)^* \vect{g}^k  \right\|^2_2
 \end{equation*}
 where $R \in \mathbb{C}^{N' \times M}$ is the matrix whose $\widetilde{n}^{\rm th}$ column is $\vect{r}_{j;\widetilde{n}}$.  This yields results that $\mathbbm{E}_{\vect{g}^k} \left[ \| h_{j;k} \|^2_{L^2(\mathcal{D}_j,\nu_j)} \right] = \| \widetilde{A}_{j} R\|^2_{\rm F}$.

Now observe that $\left(\widetilde{A}_{j} R \right)^* \vect{g}^k \sim \mathcal{N}(\vect{0},U\Sigma^2U^{\ast})$, where $U\Sigma V^{\ast}$ is the SVD of $\left(\widetilde{A}_{j} R\right)^* $ and $\Sigma \in \mathbb{R}^{\min \{M,m\} \times \min \{M,m\}}$ is the diagonal matrix containing at most $\min \{M,m\}$ nonzero singular values, $\sigma_1 \geq \sigma_2 \geq \dots \geq \sigma_{\min \{M,m\}} \geq 0$, of $\widetilde{A}_{j} R$.  Let $\vect{\widetilde{g}}^k:=\Sigma V^{\ast}\vect{g}^k \sim \mathcal{N}(\vect{0},\Sigma^2)$ and note that $\|U\vect{\widetilde{g}}^k\|_2^2=\|\vect{\widetilde{g}}^k\|_2^2$. As a consequence, one can see that 
\begin{equation*}
\mathbb{P} \left[ \| h_{j;k} \|^2_{L^2(\mathcal{D}_j,\nu_j)} \geq t \right] =\mathbb{P} \left[ \| \vect{\widetilde{g}}^k \|_2^2 \geq t \right] = \mathbb{P} \left[ \sum^{\min \{M,m\}}_{\ell = 1} \sigma^2_{\ell} X_{\ell} \geq t \right]
\end{equation*}
holds for all $t \in \mathbb{R}$, where each $X_{\ell}$ is an i.i.d $\chi^2$ random variable.  Applying the Bernstein type inequality given in Proposition 5.16 in \cite{vershynin2011introduction} we deduce that
\begin{align}
\mathbb{P} \left[ \left| \| h_{j;k} \|^2_{L^2(\mathcal{D}_j,\nu_j)}  - \| \widetilde{A}_{j} R\|^2_{\rm F} \right| \geq t \right] &\leq 2 \exp\left(-a' \min\left\{ \frac{t^2}{\|\vect{\sigma}\|_4^4}, \frac{t}{\|\vect{\sigma}\|_{\infty}^2} \right\}\right) \nonumber \\
& \leq 2 \exp\left(-a' \min\left\{ \frac{t^2}{\|\widetilde{A}_{j} R\|_F^4}, \frac{t}{\|\widetilde{A}_{j} R\|_{F}^2} \right\}\right)
 \label{equ:ProbNormfk}
\end{align}
where $a' \in \mathbb{R}^+$ is an absolute constant, and $\vect{\sigma}$ is the vector containing the diagonal elements of $\Sigma$.  An application of \eqref{equ:ProbNormfk} with $t = \max \left\{ \log 80 / a' ,\sqrt{\log 80 / a'} \right\}\| \widetilde{A}_{j} R\|^2_{\rm F}$ finally tells us that 
$$\|h_{j;k} \|^2_{L^2(\mathcal{D}_j,\nu_j)} \geq \left( 1 + \max\left\{ \frac{\log 80}{a'},\sqrt{\frac{\log 80}{a'}}  \right\} \right) \| \widetilde{A}_{j} R\|^2_{\rm F}$$
will hold with probability at most $1/40$.

Turning our attention to $\| \widetilde{A}_{j} R\|^2_{\rm F} = \sum_{\widetilde{n}\in[M]} \| \widetilde{A}_{j} \vect{r}_{j;\widetilde{n}} \|^2_2$, we assert that
$$\| \widetilde{A}_{j} R\|^2_{\rm F} = \sum_{\widetilde{n}\in[M]} \| \widetilde{A}_{j} \vect{r}_{j;\widetilde{n}} \|^2_2 \geq \frac{1}{2} \sum_{\widetilde{n}\in[M]} \| \vect{r}_{j;\widetilde{n}} \|^2_2 = \frac{1}{2} \| \vect{r} \|^2_2$$  
holds since the restricted isometry constant $\delta_{2s}$ of $\widetilde{A}_{j}$ is assumed to be bounded above by $\frac{7}{16}$. Therefore, we finally get the desired probability estimate with $\alpha':= \frac{1}{2} \big( 1 + \max \big\{ \frac{\log 80}{a'}, \allowbreak \sqrt{\frac{\log 80}{a'}}  \big\} \big)$.
\end{proof}

The following lemma bounds the estimated inner products.

\begin{lemma} Suppose that $\vect{r} \in \mathbb{C}^{N}$ is $2s$-sparse, and the restricted isometry constant $\delta_{2s}$ of $\widetilde{A}_j$ satisfies $\delta_{2s}\leq\delta$ for all $j\in[D]$ where $\delta\in(0,7/16)$.  Let $k \in [L]$, $j\in [D],$ and $\widetilde{n} \in [M]$.  Then, there exists an absolute constant $\beta' \in \mathbb{R}^+$ such that 
\begin{equation}
\mathbb{P} \left[ \left| \langle h_{j;k}, T_{j;\widetilde{n}} \rangle_{(\mathcal{D}_j, \nu_j)} \right| \leq \frac{\sqrt{23}}{12} \| \vect{r}_{j;\widetilde{n}} \|_2 \text{ or } \left| \langle h_{j;k}, T_{j;\widetilde{n}} \rangle_{(\mathcal{D}_j, \nu_j)} \right| \geq {\frac{9}{4}} \| \vect{r}_{j;\widetilde{n}} \|_2  \right] \leq 0.273.
\label{lem:equ:fkInnerProfBig}
\end{equation}
\label{lem:fkInnerProfBig}
\end{lemma}
\begin{proof}
Consider the random sampling point set $\mathcal{X}_j$ to be fixed for the time being.  Recalling the definitions of $h_{j;k}$ and of $\vect{r}_{j;\widetilde{n}}$, one can see that
\begin{equation}
\langle h_{j;k}, T_{j;\widetilde{n}} \rangle_{(\mathcal{D}_j, \nu_j)} = \frac{1}{\sqrt{m}}\sum_{\ell\in[m]} g^k_{\ell} \sum_{\substack{ {\vect{n}\text{ s.t.}}\\ {n_j=\widetilde{n}}}} r_{\vect{n}}\prod _{i\in [D]'} T_{i;n_{i}}(x_{j,\ell})_i = \sum_{\ell\in[m]} g^k_{\ell} \left({\widetilde{A}_{j}\vect{r}_{j;\widetilde{n}}} \right)_{\ell}.
\label{equ:DistInnerProdfk}
\end{equation}
Looking at \eqref{equ:DistInnerProdfk} one can see that $\langle h_{j;k}, T_{j;\widetilde{n}} \rangle_{(\mathcal{D}_j, \nu_j)} \sim \mathcal{N}(0,\| \widetilde{A}_{j} \vect{r}_{j;\widetilde{n}} \|^2_2)$ and hence,
\begin{equation}
\mathbb{P} \left[ \left|\langle h_{j;k}, T_{j;\widetilde{n}} \rangle_{(\mathcal{D}_j, \nu_j)} \right| \leq \frac{\| \widetilde{A}_{j} \vect{r}_{j;\widetilde{n}} \|_2}{3} \text{ or } \left|\langle h_{j;k}, T_{j;\widetilde{n}} \rangle_{(\mathcal{D}_j, \nu_j)} \right| \geq 3\| \widetilde{A}_{j} \vect{r}_{j;\widetilde{n}} \|_2 \right] \leq 0.273
\label{equ:ProbBound}
\end{equation}
holds. Combining (\ref{equ:ProbBound}) and the assumption on $\delta_{2s}$, which yields $\frac{9}{16}\| \vect{r}_{j;\widetilde{n}} \|^2_2 \leq \| \widetilde{A}_{j} \vect{r}_{j;\widetilde{n}} \|^2_2 \leq \frac{23}{16} \| \vect{r}_{j;\widetilde{n}} \|^2_2$, establishes the desired result.
\end{proof}



\begin{mytheorem}
Let $\mathcal{H}$ be a finite set of functions $h$ whose BOS coefficient vectors are $2s$-sparse, and let $\vect{r}_h \in \mathbb{C}^N$ denote the coefficient vector for each $h \in \mathcal{H}$.  Suppose that the restricted isometry constant $\delta_{2s}$ of $\widetilde{A}_j$ satisfies $\delta_{2s}\leq\delta$ for all $j\in[D]$ where $\delta\in(0,7/16)$. Furthermore, let $p \in (0,1)$, $L \in \mathbb{N}$ be odd, and $L \geq \widetilde{\gamma} \log (DM |\mathcal{H}| / p)$ hold for a sufficiently large absolute constant $\widetilde{\gamma} \in \mathbbm{R}^+$.  Then, $\sum_{k \in [L]} E_{h,j,\widetilde{n},k} > L / 2$ simultaneously for all  $(h,j,\widetilde{n}) \in \mathcal{H} \times[D]\times [M]$ with probability at least $1 - p$.  That is, with probability at least $1 - p$, the following will hold simultaneously for each $(h,j,\widetilde{n}) \in \mathcal{H}\times [D] \times [M]$:  All three of
\begin{enumerate}
\item $\| h_{j;k} \|^2_{L^2(\mathcal{D}_j,\nu_j)} \leq \alpha' \|\vect{r}_h \|_2^2$ for the absolute constant $\alpha'$ defined in Lemma~\ref{lem:fkPresTotalEng},
\item ${\frac{9}{4}}\| (\vect{r}_h)_{j;\widetilde{n}} \|_2 \geq \left| \langle h_{j;k}, T_{j;\widetilde{n}} \rangle_{(\mathcal{D}_j, \nu_j)} \right| \geq \frac{\sqrt{23}}{12} \| (\vect{r}_h)_{j;\widetilde{n}} \|_2 $, and
\item the vector of Gaussian weights $\vect{g}^k \in \mathbbm{R}^m$ satisfying $ \frac{1}{2} m \leq \| \vect{g}^k  \|^2_2 \leq \frac{3}{2}  m$,
\end{enumerate}
will be simultaneously true for more than half of the $k \in [L]$.
\label{Thm:FiniteSetMedtoUn}
\end{mytheorem}

\begin{proof}
Let $h \in \mathcal{H}$, $k \in[L]$, and $\widetilde{n} \in [M]$.  The probabilities that the first and second properties fail are given in \eqref{lem:equ:fkPresTotalEng} and \eqref{lem:equ:fkInnerProfBig}, respectively. For the third property, applying the Bernstein type inequality given in Proposition 5.16 in \cite{vershynin2011introduction}, one obtains 
\begin{equation}
\mathbb{P} \left[ \left| \| \vect{g}^k  \|^2_2  - m \right| \geq \frac{m}{2} \right] \leq 2 \mathbbm{e}^{-a'' m} \leq 0.03, 
\label{equ:FiniteUnivSetupBern}
\end{equation}
where $a'' \in \mathbb{R}^+$ is an absolute constant. 

Combining \eqref{lem:equ:fkPresTotalEng}, \eqref{lem:equ:fkInnerProfBig}, \eqref{equ:FiniteUnivSetupBern} via a union bound now tell us that $\mathbb{P} \left[ E_{h,j,\widetilde{n},k} = 0 \right] \leq 328/1000$.  Utilizing the Chernoff bound (see, e.g., \cite{motwani1995randomized, arratia1989tutorial}) one now sees that
$$\mathbb{P} \left[ \sum_{k \in [L]} E_{h,j,\widetilde{n},k} < L / 2 \right]= \mathbb{P} \left[ \sum_{k \in [L]} (1 - E_{h,j,\widetilde{n},k}) > L/2\right] < \mathbbm{e}^{-L / \bar{\gamma}} \leq \frac{p}{DM |\mathcal{H}|}$$
for an absolute constant $\bar{\gamma}\in \mathbb{R}^+$, where the last inequality follows by choosing $\widetilde{\gamma}=\bar{\gamma}$ in the assumption.  Applying the union bound over all choices of $(h,j,\widetilde{n}) \in \mathcal{H} \times [D]\times [M]$ now establishes the desired result.
\end{proof}


\begin{mytheorem}
Let $\mathcal{H}_{2s}$ be the set of all functions $h$ whose coefficient vectors are $2s$-sparse, and let $\vect{r}_h \in \mathbb{C}^N$ denote the coefficient vector for each $h \in \mathcal{H}_{2s}$. Suppose that the restricted isometry constant $\delta_{2s}$ of $\widetilde{A}_j$ satisfies $\delta_{2s}\leq\delta$ for all $j\in[D]$ where $\delta\in(0,7/16)$. Furthermore, let $p \in \left( 0,1 \right)$, $L \in \mathbb{N}$ be odd, and assume that $L \geq \gamma' s \cdot d \cdot \log \left( \frac{DM}{\sqrt[d]{sd~p^{1/s}}} \right)$ for sufficiently large absolute constant $ \gamma' \in \mathbbm{R}^+$.  Then, with probability greater than $1 - p$, one has
$\sum_{k \in [L]} E_{h,j,\widetilde{n},k} > L / 2$ simultaneously for all  $(h,j,\widetilde{n}) \in \mathcal{H} \times[D] \times [M]$.
Consequently, with probability greater than $1 - p$,  it will hold that for all choices of $(h,j,\widetilde{n}) \in \mathcal{H}_{2s}\times [D] \times [M]$ both 
\begin{enumerate}
\item $\| h_{j;k} \|^2_{L^2(\mathcal{D}_j,\nu_j)} \leq (\alpha' + 1) \|\vect{r}_h \|_2^2$ for the absolute constant $\alpha'$ defined in Lemma~\ref{lem:fkPresTotalEng}, and
\item $ \frac{9}{2} \| (\vect{r}_h)_{j;\widetilde{n}} \|_2\geq \left| \langle h_{j;k}, T_{j;\widetilde{n}} \rangle_{(\mathcal{D}_j, \nu_j)} \right| \geq \frac{1}{3} \| (\vect{r}_h)_{j;\widetilde{n}} \|_2 $
\end{enumerate}
are true simultaneously for more than half of the $k \in [L]$.
\label{Thm:FiniteSetMedtoUn2}
\end{mytheorem}

\begin{proof}
Define $\mathcal{R}^\epsilon \subset \mathbb{C}^{N}$ as a finite $\epsilon$-cover of all $2s$-sparse coefficient vectors $\vect{r} \in \mathbb{C}^{N}$ with $\| \vect{r} \|_2 = 1$, together with $\vect{0}$, where $N = {D \choose d}M^d$ and $\epsilon \in (0,1)$.  Such covers exist of cardinality $|\mathcal{R}^\epsilon|  \leq  \left( \frac{\mathbbm{e}N}{2s} \right)^{2s} \left(1+ \frac{2}{\epsilon} \right)^{2s}$ (see, e.g., Appendix C of \cite{foucart2013mathematical}).  Define $\mathcal{H}^\epsilon$ as the set of functions  corresponding to the $2s$-sparse coefficient vectors in $\mathcal{R}^\epsilon$.  Assume that for $\mathcal{H}$ = $\mathcal{H}^\epsilon$ and all choices of $(h,j,\widetilde{n}) \in \mathcal{H}^{\epsilon} \times [D]\times [M]$, Properties 1 -- 3 of Theorem~\ref{Thm:FiniteSetMedtoUn} will hold for more than half of the $k \in [L]$.  By the theorem, this event will happen with probability at least $1-p$. We will now prove that under this assumption both Properties 1 and 2 above will hold as desired.

Let $\widetilde{n} \in [M]$, $j\in [D]$, consider $h \in \mathcal{H}_{2s}$ with coefficient vector $\vect{r}=\vect{r}_h$, and let $\tau := \| \vect{r} \|_2^2$.  Then, there exists an $h' \in \mathcal{H}^\epsilon$ with coefficient vector $\vect{r}' \in \mathcal{R}^\epsilon$ such that both $\| \vect{r}' \|_2 = 1$ and $\left\| \vect{r} - \tau \vect{r}' \right\|_2 \leq \epsilon \tau$ hold.  Finally, let $k \in [L]$ be one of the values for which Properties 1 -- 3 of Theorem~\ref{Thm:FiniteSetMedtoUn} are simultaneously true for $(h',j,\widetilde{n})$.  We will begin by establishing Property 1 above for $h$, $j$ and $k$.  Using \eqref{def:fk} we have that
\begin{align}
\left \| h_{j;k} \right \|_{L^2(\mathcal{D}_j,\nu_j)} &= \left \| \frac{1}{\sqrt{m}}\sum_{\ell\in[m]} g^k_{\ell} \sum_{\vect{n}\in {\rm supp}(\vect{r})}   r_{\vect{n}}T_{j;n_j}(x)\prod _{i\in [D]'} T_{i;n_{i}}(x_{j,\ell})_i \right \|_{L^2(\mathcal{D}_j,\nu_j)} \nonumber \\
&\leq \tau \left \| h'_{j;k} \right \|_{L^2(\mathcal{D}_j,\nu_j)} \\
&+ \left \| \frac{1}{\sqrt{m}}\sum_{\ell\in [m]} g^k_{\ell} \sum_{\vect{n}\in {\rm supp}(\vect{r})}   \left( r_{\vect{n}} - \tau r'_{\vect{n}} \right)T_{j;n_j}(x)\prod _{i\in [D]'} T_{i;n_{i}}(x_{j,\ell})_i \right \|_{L^2(\mathcal{D}_j,\nu_j)} \nonumber \\
& \leq \tau \sqrt{\alpha'} \| \vect{r}' \|_2 +  \left \| \left( h-  \tau h' \right)_{j;k} \right \|_{L^2(\mathcal{D}_j,\nu_j)} \nonumber \\
& = \tau \sqrt{\alpha'} +  \left \| \left( h-  \tau h' \right)_{j;k} \right \|_{L^2(\mathcal{D}_j,\nu_j)},
\label{equ:UniProofEq1}
\end{align}
where the last inequality follows from the first property of Theorem~\ref{Thm:FiniteSetMedtoUn} holding for $h'$.

Repeating the expansion from the proof of Lemma~\ref{lem:fkPresTotalEng} for $\left \| \left( h-  \tau h' \right)_{j;k} \right \|^2_{L^2(\mathcal{D}_j,\nu_j)}$, one obtains
\begin{align*}
\left \| \left( h-  \tau h' \right)_{j;k} \right \|^2_{L^2(\mathcal{D}_j,\nu_j)} &= \sum_{\widetilde{n}\in[M]}  \bigg| \left\langle {\widetilde{A}_{j}\left( \vect{r} - \tau \vect{r}' \right)_{j;\widetilde{n}}}, \vect{g}^k \right\rangle \bigg|^2\\
&\leq \sum_{\widetilde{n}\in[M]} \left \| \widetilde{A}_{j}\left( \vect{r} - \tau \vect{r}' \right)_{j;\widetilde{n}} \right\|^2_2 \left \| \vect{g}^k  \right\|^2_2 \\
&\leq \frac{9}{4}m \sum_{\widetilde{n}\in[M]}  \left \| \left( \vect{r} - \tau \vect{r}' \right)_{j;\widetilde{n}} \right\|^2_2
\end{align*}
where the last inequality follows from the third property of Theorem~\ref{Thm:FiniteSetMedtoUn}, and $\widetilde{A}_{j}$ having $\delta_{2s} \leq \frac{7}{16}$.  Continuing, we can see that
$$ \left \| \left( h-  \tau h' \right)_{j;k} \right \|^2_{L^2(\mathcal{D}_j,\nu_j)} \leq \frac{9}{4} m \left \| \left( \vect{r} - \tau \vect{r}' \right) \right\|^2_2 \leq \frac{9}{4} m \tau^2 \epsilon^2.$$
Combining this expression with \eqref{equ:UniProofEq1} we now learn that
$$\left \| h_{j;k} \right \|_{L^2(\mathcal{D}_j,\nu_j)} \leq \tau \left( \sqrt{\alpha'} + \frac{3}{2}\epsilon \sqrt{ m}\right) = \| \vect{r} \|_2 \left( \sqrt{\alpha'} + \frac{3}{2} \epsilon \sqrt{ m}\right).$$
Making sure to use, e.g., an $\epsilon \leq \left (6\sqrt{\alpha' m} \right)^{-1}$ ensures property one.

Turning our attention to establishing Property 2 above for $h$, $\widetilde{n}$, and $k$, we now choose an $h' \in \mathcal{H}^\epsilon$ whose coefficient vector $\vect{r}' \in \mathcal{R}^\epsilon$ has $\| \vect{r}' \|_2 
= 1$, and also satisfies 

\begin{equation}
\label{eq:rrprime}
\left\| \vect{r}_{j;\widetilde{n}} - \tau' \vect{r}'
\right\|_2 \leq \epsilon \tau'
\end{equation}
 for $\tau' := \| \vect{r}_{j;\widetilde{n}} \|_2$.  
 Note that all nonzero entries of $\vect{r}'_{j;\widetilde{n}}$ agree with those of $\vect{r}'$, the latter vector only has certain additional nonzero entries in locations where $\vect{r}'_{j;\widetilde{n}}$ and also $\vect{r}_{j;\widetilde{n}}$ vanish. Consequently, replacing $\vect{r}'$ by $\vect{r}'_{j;\widetilde{n}}$ makes the left hand side of \eqref{eq:rrprime} smaller, and one obtains that
 \begin{equation}
 \label{eq:rrprime2}
 \left\| \vect{r}_{j;\widetilde{n}} - \tau' \vect{r}'_{j;\widetilde{n}}.
 \right\|_2 \leq \epsilon \tau'
 \end{equation}
 
 From \eqref{def:fk} one can see that
\begin{align*}
\left| \langle h_{j;k}, T_{j;\widetilde{n}} \rangle_{(\mathcal{D}_j, \nu_j)} \right| &\geq \left| \langle \tau' h'_k, T_{j;\widetilde{n}} \rangle_{(\mathcal{D}_j, \nu_j)} \right| - \left| \left \langle \left(h_{j;k}- \tau' h'_k \right), T_{j;\widetilde{n}} \right \rangle_{(\mathcal{D}_j, \nu_j)} \right|\\
&\geq \frac{\sqrt{23}}{12}{\tau'} \| \vect{r}'_{j;\widetilde{n}} \|_2 - \left| \sum_{\ell\in[m]} g^k_{\ell} \left({\widetilde{A}_{j} \left(\vect{r} - \tau' \vect{r}' \right)_{j;\widetilde{n}}} \right)_{\ell} \right|.
\end{align*}
where the last inequality follows from the second property of Theorem~\ref{Thm:FiniteSetMedtoUn} holding for $h'$.  Continuing using \eqref{eq:rrprime2} we have that
\begin{align*}
\left| \langle h_{j;k}, T_{j;\widetilde{n}} \rangle_{(\mathcal{D}_j, \nu_j)} \right| &\geq \frac{\sqrt{23}}{12} \left( \| \vect{r}_{j;\widetilde{n}} \|_2 - \| (\tau' \vect{r}' - \vect{r})_{j;\widetilde{n}} \|_2 \right) - \bigg| \left\langle {\widetilde{A}_{j}\left( \vect{r} - \tau' \vect{r}' \right)_{j;\widetilde{n}}}, \vect{g}^k \right\rangle \bigg| \\
&\geq \frac{\sqrt{23}}{12} \| \vect{r}_{j;\widetilde{n}} \|_2 - \frac{\sqrt{23}}{12}{\epsilon \tau'}  - \bigg| \left\langle {\widetilde{A}_{j}\left( \vect{r} - \tau' \vect{r}' \right)_{j;\widetilde{n}}}, \vect{g}^k \right\rangle \bigg| \\
&\geq \frac{\sqrt{23}}{12} \| \vect{r}_{j;\widetilde{n}} \|_2 - \frac{\sqrt{23}}{12}{\epsilon \tau'}  - \left \| \widetilde{A}_{j}\left( \vect{r} - \tau' \vect{r}' \right)_{j;\widetilde{n}} \right\|_2 \left \| \vect{g}^k  \right\|_2 \\
&\geq \frac{\sqrt{23}}{12} \| \vect{r}_{j;\widetilde{n}} \|_2 - \frac{\sqrt{23}}{12}{\epsilon \tau'}  - \frac{3}{2}\sqrt{ m} \left \|\left( \vect{r} - \tau' \vect{r}' \right)_{j;\widetilde{n}} \right\|_2\\
&= \frac{1}{3} \| \vect{r}_{j;\widetilde{n}} \|_2 \left( \frac{\sqrt{23}}{4} - \frac{\sqrt{23}\epsilon}{4} - \frac{9}{2}\epsilon \sqrt{ m} \right).
\end{align*}
On the other hand,
\begin{align*}
\left| \langle h_{j;k}, T_{j;\widetilde{n}} \rangle_{(\mathcal{D}_j, \nu_j)} \right| &\leq \left| \langle \tau' h'_{j;k}, T_{j;\widetilde{n}} \rangle_{(\mathcal{D}_j, \nu_j)} \right| + \left| \left \langle \left(h_{j;k}- \tau' h'_{j;k} \right), T_{j;\widetilde{n}} \right \rangle_{(\mathcal{D}_j, \nu_j)} \right|\\
&\leq \frac{9}{4}\tau'  \| \vect{r}'_{j;\widetilde{n}} \|_2 + \left| \sum_{\ell\in[m]} g^k_{\ell} \left({\widetilde{A}_{j} \left(\vect{r} - \tau' \vect{r}' \right)_{j;\widetilde{n}}} \right)_{\ell} \right|.
\end{align*}
As above, we obtain using \eqref{eq:rrprime2} that
\begin{align*}
\left| \langle h_{j;k}, T_{j;\widetilde{n}} \rangle_{(\mathcal{D}_j, \nu_j)} \right| &\leq \frac{9}{4} \left( \| \vect{r}_{j;\widetilde{n}} \|_2 + \| (\tau' \vect{r}' - \vect{r})_{j;\widetilde{n}} \|_2 \right) + \bigg| \left\langle {\widetilde{A}_{j}\left( \vect{r} - \tau' \vect{r}' \right)_{j;\widetilde{n}}}, \vect{g}^k \right\rangle \bigg| \\
&\leq  \frac{9}{4} \| \vect{r}_{j;\widetilde{n}} \|_2 +  \frac{9}{4}\epsilon \tau'  + \bigg| \left\langle {\widetilde{A}_{j}\left( \vect{r} - \tau' \vect{r}' \right)_{j;\widetilde{n}}}, \vect{g}^k \right\rangle \bigg| \\
&\leq  \frac{9}{4} \| \vect{r}_{j;\widetilde{n}} \|_2 +  \frac{9}{4}\epsilon \tau'  + \left \| \widetilde{A}_{j}\left( \vect{r} - \tau' \vect{r}' \right)_{j;\widetilde{n}} \right\|_2 \left \| \vect{g}^k  \right\|_2 \\
&\leq \frac{9}{4} \| \vect{r}_{j;\widetilde{n}} \|_2+  \frac{9}{4}\epsilon \tau' + \frac{3}{2}\sqrt{ m} \left \|\left( \vect{r} - \tau' \vect{r}' \right)_{j;\widetilde{n}} \right\|_2\\
&= \frac{9}{4} \| \vect{r}_{j;\widetilde{n}} \|_2 \left(1+ \epsilon + \frac{2}{3} \epsilon  \sqrt{ m } \right).
\end{align*}
Once again, making sure to use, e.g., an $\epsilon \leq \left (6\sqrt{\alpha' m} \right)^{-1}$ will now ensure property two for $h$ as well.
\end{proof}


\begin{lemma}
Let $\mathcal{H}_{2s}$ be the set of all functions $h$ whose coefficient vectors are $2s$-sparse, and let $\vect{r}_h \in \mathbb{C}^N$ denote the coefficient vector for each $h \in \mathcal{H}_{2s}$.  Furthermore, let $\delta\in(0,7/16)$, $p \in \left( 0,1 \right)$, $L \in \mathbb{N}$ be odd, and assume that $m \geq \widetilde{\beta}' K^2 \delta^{-2}s \max\big\{ d \log^2(4s) \log(9m) \allowbreak \log \big( \frac{8\mathbbm{e}(D-1)DM}{d} \big),   \log\big(\frac{2D}{p}\big) \big\}$ and $L \geq \gamma' s \cdot d \cdot \log \left( \frac{DM}{\sqrt[d]{sd~(p/2)^{1/s}}} \right)$ for sufficiently large absolute constants $\widetilde{\beta}', \gamma' \in \mathbbm{R}^+$.  Then, with probability greater than $1 - p$, all of the following will hold for all $(h,j,\widetilde{n}) \in \mathcal{H}_{2s}\times [D] \times [M]$:  Both 
\begin{enumerate}
\item $\| h_{j;k} \|^2_{L^2(\mathcal{D}_j,\nu_j)} \leq (\alpha' + 1) \|\vect{r}_h \|_2^2$ for the absolute constant $\alpha'$ defined in Lemma~\ref{lem:fkPresTotalEng}, and
\item $ \frac{9}{2} \| (\vect{r}_h)_{j;\widetilde{n}} \|_2\geq \left| \langle h_{j;k}, T_{j;\widetilde{n}} \rangle_{(\mathcal{D}_j, \nu_j)} \right| \geq \frac{1}{3} \| (\vect{r}_h)_{j;\widetilde{n}} \|_2 $
\end{enumerate}
will be simultaneously true for more than half of the $k \in [L]$.
\label{Thm:FiniteSetMedtoUn3}
\end{lemma}

\begin{proof}
Let $A$ be the event that for all $(h,j,\widetilde{n}) \in \mathcal{H}_{2s}\times [D] \times [M]$, the properties 1 and 2 in Theorem  \ref{Thm:FiniteSetMedtoUn2} simultaneously hold for more than half of the $k \in [L]$, and let $B$ be the event that the restricted isometry constant $\delta_{2s}$ of $\widetilde{A}_j$ satisfies $\delta_{2s}\leq\delta$ for all $j\in[D]$.
By the Theorem \ref{Thm:FiniteSetMedtoUn2} and Lemma \ref{lem:RIPforD} with properly chosen parameters including $L$ and $m$, both $\mathbb{P}\left[A~ \big|~B\right]$ and $\mathbb{P}[B]$ are greater than $1 - p/2$.  
We obtain, by Bayes' theorem,
\begin{equation*}
1-p \leq \left(1-\frac{p}{2}\right) \left(1-\frac{p}{2}\right)\leq \mathbb{P} \left[ A ~ \big|~ B \right] \mathbb{P} \left[ B \right] = \mathbb{P} [A \cap B]\leq P[A],
\end{equation*}
which establishes the desired result.
\end{proof}

The results from Lemma \ref{Thm:FiniteSetMedtoUn3} can be exploited in our algorithm as follows. In the entry identification, by taking the median over $k\in[L]$ of $\left| \langle h_{j;k}, T_{j;\widetilde{n}} \rangle_{(\mathcal{D}_j, \nu_j)} \right|$, we get a nonzero value if $ \left\| (\vect{r}_h)_{j;\widetilde{n}} \right\|_2$ is nonzero and a zero value if  $ \left\| (\vect{r}_h)_{j;\widetilde{n}} \right\|_2$ is zero, especially due to the second property in Lemma \ref{Thm:FiniteSetMedtoUn3} being satisfied for more the half of $k \in [L]$. Thus, we store all those $\widetilde{n}$ with nonzero median value in $\mathcal{N}_j$. On the other hand, the summation over $\widetilde{n}\in[M]$ of ${\rm median}_{k} \left| \langle h_{j;k}, T_{j;\widetilde{n}} \rangle_{(\mathcal{D}_j, \nu_j)} \right|$ can be also used for the halting criterion in our algorithm. Although $\mathcal{O}(\|\vect{c}\|_2/\eta)$ iterations guarantee the desired precision, it is not necessary to repeat the iteration if the residual $\vect{r}$ already has small energy. For any $j\in[D]$, if 
\begin{equation*}
\sum_{\widetilde{n} \in [M]} \left |{\rm median}_{k}| \langle h_{j;k}, T_{j;\widetilde{n}} \rangle_{(\mathcal{D}_j, \nu_j)}| \right|^2 \leq \frac{1}{9} \eta^2,
\end{equation*}
then $\|\vect{r}\|_2 \leq \eta$ by using the lower bound in the Property 2 of Lemma \ref{Thm:FiniteSetMedtoUn3}.



\subsection{Pairing}
\label{sec:pairing}

In the entry identification, we can find at most $2s$ entries belonging to $\mathcal{N}^t_j := \{ {n}_j ~|~ \vect{n} \in {\rm supp}(\vect{r}) \} \subset [M]$ for each $j\in[D]$. 
Now, the question is how to combine the entries of each $\mathcal{N}_j  \supseteq \mathcal{N}_j \cap \mathcal{N}^t_j$ in order to correctly identify the corresponding energetic elements of ${\rm supp}({\vect{r} })$ efficiently.
In order to do this, in the pairing process briefly introduced in Section \ref{secPairing}, we successively build up the prefix set $\mathcal{P}_{j}$ of the energetic pairs for all $j\in [D]\setminus \{0\}$ such that $\mathcal{P}_{j} \supset \left\{\widetilde{\vect{n}}\in \mathcal{P}_{j-1} \times \mathcal{N}_j ~\big|~ \|\vect{r}_{j;(\widetilde{\vect{n}},\cdots)}\|_2^2 \geq \frac{\|\vect{r}\|_2^2}{\alpha^2 s} \right\}$ with the initialization of $\mathcal{P}_0=\mathcal{N}_0$. The prefix set $\mathcal{P}_j$ contains only $2s$ pairs throwing out the other pairs with smaller energy for each $j\in[D]\setminus \{0\}$ so that $\mathcal{P}_{D-1} \supset \left\{\widetilde{\vect{n}}\in {\rm supp}(\vect{r}) ~\big|~ |{r}_{\widetilde{\vect{n}}}|^2 \geq \frac{\|\vect{r}\|_2^2}{\alpha^2 s} \right\}$. From (\ref{equ:EngEstPairing}), the energy $\left\| \vect{r}_{j;(\widetilde{\vect{n}},\cdots)} \right\|_2^2$ corresponding to each $\widetilde{\vect{n}} \in[M]^{j+1}$ has the following equality,
\begin{equation}
\left\| \vect{r}_{j;(\widetilde{\vect{n}},\cdots)} \right\|_2^2=\left\|  \langle h, {T}_{j;\widetilde{\vect{n}} }\rangle_{L^2 \left( \times_{i\in[j+1]} \mathcal{D}_i,  \otimes_{i\in[j+1]} \nu_i  \right)} \right\|^2_{L^2 \left( \mathcal{D}''_j,  \vect{\nu}''_j \right)}.
\label{eqnEest}
\end{equation}
The energy is estimated in Algorithm \ref{alg:Pairing} by using the following estimator $E_{j;(\widetilde{\vect{n}},\cdots)}$ defined as
\begin{equation}
E_{j;(\widetilde{\vect{n}},\cdots)}:=\frac{1}{m_2}\sum_{k\in[m_2]} \left| \frac{1}{m_1}\sum_{\ell\in[m_1]} h(\vect{w}_{j,\ell},\vect{z}_{j,k}) \overline{{T}_{j;\vect{\widetilde{n}}}(\vect{w}_{j,\ell})}\right|^2\nonumber\\
\label{eqnEest2}
\end{equation}
which approximates the right hand side of (\ref{eqnEest}) by using only a finite evaluations of $h$. Those sampling point sets $\mathcal{W}_{j}\times \mathcal{Z}_{j}$ for all $j\in [D]\setminus\{0\}$ are constructed from  $\mathcal{W}_{j}:=\{\vect{w}_{j,\ell}\}_{\ell\in[m_1]}$ and $\mathcal{Z}_{j}:=\{\vect{z}_{j,k}\}_{k\in[m_2]}$ where $\vect{w}_{j,\ell}$ and $\vect{z}_{j,k}$ are chosen independently at random from $ \times_{i\in[j+1]} \mathcal{D}_i $ and $\mathcal{D}''_j$  for $j\in[D-1]\setminus\{0\}$, respectively. If $j=D-1$, $\mathcal{W}_{D-1}:=\{\vect{w}_{D-1,\ell}\}_{\ell\in[m_1]}$ is chosen from $\times_{i\in[D]} \mathcal{D}_i$ and $\mathcal{Z}_{D-1}=\emptyset$. Note that $\mathcal{W}_{j}\times \mathcal{Z}_{j}=\mathcal{G}^P_{j}$ from Section \ref{sec:UniversalGrids}. Furthermore, the sets $\mathcal{W}_{j}\times \mathcal{Z}_{j}$ for all $j\in [D]\setminus\{0\}$ build a random sampling matrix $A^P$ in (\ref{defSamMat}) which explicitly expresses the samples(evaluations) of the $2s$-sparse $h$ used in the (\ref{eqnEest2}) as $A^P \vect{r}$. The matrix $A^P$ is broken into smaller matrices $B_j$ and $C_j$ for $j\in[D]\setminus \{0\}$ defined and explained in the next paragraph for the complete analysis of the pairing process in the upcoming lemmas and theorems in this section.

For all $j\in [D-1]\setminus\{0\}$, the measurement matrix ${C}_j \in \mathbb{C}^{m_2 \times [M]^{D-j-1}}$ is defined as
\begin{equation}
\left( C_j \right)_{k, \vect{n}_2}:= \frac{1}{\sqrt{m_2}} T''_{\vect{n}_2}(\vect{z}_{j,k}), \quad k \in [m_2], ~ \vect{n}_2 \in [M]^{D-j-1},
\label{eqn:defC}
\end{equation}
where $T''_{\vect{n}_2}(\vect{y})$ is a partial product of the last $D-j-1$ terms of $T_{\vect{n}}(\vect{x})$ defined in (\ref{def:T_n}), i.e.,
\begin{equation*}
T''_{\vect{n}_2}(\vect{y})=\prod_{i\in[D-j-1]}T_{{i+j+1};n_{i+j+1}}(y_i), 
\quad \vect{y} \in \mathcal{D}''_j. 
\end{equation*}
The matrix $C_j $ can be restricted to the matrix of size $m_2 \times \widetilde{N}_j$ when $C_j $ is applied to $\vect{v}_{j,(\widetilde{n}, \cdots)}$ where $\widetilde{N}_j$ estimated in (\ref{lem:estNj}) is the cardinality of the superset of any possible ${\rm supp}(\vect{v}_{j;\widetilde{\vect{n}}})$ with fixed $j, \widetilde{\vect{n}}, D$ and $d$ so that it satisfies RIP with sufficiently large $m_2$. When the $j=D-1$, the matrix $C_j$ is defined to be $1$ since $\mathcal{Z}_{D-1}=\emptyset$.
For all $j\in [D]\setminus\{0\}$, on the other hand, the matrices ${B}_j \in \mathbb{C}^{m_1 \times [M]^{j+1}}$
is defined as 
\begin{equation}
\left( B_j \right)_{\ell, \vect{n}_1}:= \frac{1}{\sqrt{m_1}} T'_{\vect{n}_1}(\vect{w}_{j,\ell}), \quad \ell \in [m_1], ~ \vect{n}_1 \in [M]^{j+1},
\label{eqn:defB}
\end{equation}
where $T'_{\vect{n}_1}(\vect{z})$ is a partial product of the first $j+1$ terms of $T_{\vect{n}}(\vect{x})$ in (\ref{def:T_n}), i.e.,
\begin{equation*}
T'_{\vect{n}_1}(\vect{z})=\prod_{i\in[j+1]}T_{{i};n_{i}}(z_i), 
\quad \vect{z} \in \times_{i\in[j+1]}\mathcal{D}_i. 
\end{equation*}
The matrix $B_j$ can be restricted to the matrix of size $m_1 \times \bar{N}_j$ where 
$\bar{N}_j={j+1 \choose d} M^{d}$ if $j+1\geq d$, or $M^d$ otherwise. The number $\bar{N}_j$ is the cardinality of the set of any possible prefix $\widetilde{n}\in[M]^{j+1}$ with fixed $j, \widetilde{n}$ and $d$.
The sampling numbers $m_1$ and $m_2$ are chosen for all ${B}_j$ and ${C}_j$ with any $j\in[D]\setminus \{0\}$ to satisfy RIP with the restricted isometry constants $\widetilde{\delta}_{2s+1}\leq \widetilde{\delta}$ and $\delta'_{2s} \leq \delta'$, respectively. We again mention that $C_{D-1}=1$.
\\
\begin{lemma}
Let $\mathcal{H}_{2s}$ be the set of all functions $h$ whose coefficient vectors are $2s$-sparse and let $\vect{r}_h \in \mathbb{C}^N$ denote the coefficient vector for each $h \in \mathcal{H}_{2s}$.
Let $j\in [D]\setminus \{0\}$, and $\widetilde{\delta}$ and $\delta'$ be chosen from $(0,1)$. Assume that ${B}_j$ and ${C}_j$ satisfy {RIP} with $\widetilde{\delta}_{2s+1}\leq \widetilde{\delta}$ and $\delta'_{2s} \leq \delta'$, respectively. Then, denoting $\vect{r}:=\vect{r}_h$ for simplicity,  
\begin{align}
(1 - \delta'_{2s} ) \left( \max\{0,\left\| {\vect{r}}_{j;(\widetilde{\vect{n}},\cdots)} \right\|_2-\widetilde{\delta}_{2s+1} \|\vect{r}\|_2 \}\right)^2  & \leq E_{j;(\widetilde{\vect{n}},\cdots)} \nonumber \\
&\leq (1 + \delta'_{2s} ) \left( \left\|{\vect{r}}_{j;(\widetilde{\vect{n}},\cdots)} \right\|_2+\widetilde{\delta}_{2s+1}\|\vect{r}\|_2 \right)^2.
\label{equ:PairUnif3}
\end{align}
for any $\widetilde{\vect{n}}\in [M]^{j+1}$.
\label{lem:induction1}
\end{lemma}
\begin{proof}
As $j\in[D]\setminus \{0\}$ is fixed, for simplicity, we use notation  $\vect{w}_{\ell}$ and $\vect{z}_{k}$ for sampling points instead of $\vect{w}_{j,\ell}$ and $\vect{z}_{j,k}$ constructing the sampling matrices $B_j$ and $C_j$ as in (\ref{eqn:defB}) and (\ref{eqn:defC}), respectively. Fix $\widetilde{\vect{n}}\in [M]^{j+1}$. Letting $\vect{n}=(\vect{n}_1,\vect{n}_2)$, $\vect{n}_1 \in[M]^{j+1}$ and $\vect{n}_2 \in[M]^{D-j-1}$, we can rewrite the energy estimate $E_{j;(\widetilde{\vect{n}},\cdots)}$ as follows,
\begin{align}
E_{j;(\widetilde{\vect{n}},\cdots)}&=\frac{1}{m_2}\sum_{k\in[m_2]} \left| \frac{1}{m_1}\sum_{\ell\in[m_1]} h(\vect{w}_{\ell},\vect{z}_{k}) \overline{{T}_{j;\vect{\widetilde{n}}}(\vect{w}_{\ell})}\right|^2\nonumber\\
&=\frac{1}{m_2}\sum_{k\in[m_2]} \left| \frac{1}{m_1}\sum_{\ell\in[m_1]} \sum_{\vect{n}=(\vect{n}_1,\vect{n}_2)\in {\rm supp}(\vect{r})} r_{\vect{n}} T_{\vect{n}}(\vect{w}_{\ell},\vect{z}_{k}) \overline{{T}_{j;\vect{\widetilde{n}}}(\vect{w}_{\ell})}\right|^2\nonumber\\
&=\frac{1}{m_2}\sum_{k\in[m_2]} \frac{1}{m_1^2} \left| \sum_{\vect{n}\in {\rm supp}(\vect{r})} \sum_{\ell\in[m_1]} r_{\vect{n}} T'_{\vect{n}_1}(\vect{w}_{\ell}) \overline{{T}_{j;\vect{\widetilde{n}}}(\vect{w}_{\ell})} T''_{\vect{n}_2}(\vect{z}_{k}) \right|^2\nonumber\\
&=:\frac{1}{m_2}\sum_{k\in[m_2]}  \frac{1}{m_1^2} \left| \sum_{\substack{{\vect{n}_2 \text{ s.t. }} \exists \vect{n}_1 \\ \text{  with }(\vect{n}_1, \vect{n}_2) \in {\rm supp}(\vect{r})} } \left( \widetilde{r}_{\vect{\widetilde{n}}}\right)_{\vect{n}_2} T''_{\vect{n}_2}(\vect{z}_{k}) \right|^2, 
\label{eqn:energyEst}
\end{align}
where 
\begin{equation}
\left( {\widetilde{r}_{\vect{\widetilde{n}}}}\right)_{\vect{n}_2}:=\sum_{\substack{{\vect{n}_1 \text{ s.t.}}\\{\vect{n}=(\vect{n}_1, \vect{n}_2) \in {\rm supp}(\vect{r})} }} r_{\vect{n}} \sum_{\ell\in[m_1]}  T'_{\vect{n}_1}(\vect{w}_{\ell}) \overline{{T}_{j;\vect{\widetilde{n}}}(\vect{w}_{\ell})}.
\label{eqn:innerProd}
\end{equation}
We can construct a vector $\vect{\widetilde{r}}:=\left( \left(\widetilde{r}_{\vect{\widetilde{n}}} \right)_{\vect{n}_2}\right) \in \mathbb{C}^{\widetilde{N}_j}$ with entries $\left( {\widetilde{r}_{\vect{\widetilde{n}}}}\right)_{\vect{n}_2}$ at $\vect{n}_2$ so that $\widetilde{\vect{r}}$ has a support whose cardinality is at most $2s$ since $h$ is $2s$-sparse. Thus, the energy estimate $E_{j;(\widetilde{\vect{n}},\cdots)}$ in \eqref{eqn:energyEst} can be expressed as $\left\|C_j \left(\frac{\widetilde{\vect{r}}}{m_1}\right) \right\|_2^2$. Since the restricted measurement matrix ${C}_j \in \mathbb{C}^{m_2\times \widetilde{N}_j}$ satisfies the {RIP}, 
\begin{align}
(1 - \delta_{2s}' ) \left \| \frac{\vect{\widetilde{r}}}{m_1} \right \|_2^2 \leq \left\|C_j \left(\frac{\widetilde{\vect{r}}}{m_1}\right) \right\|_2^2 \leq (1 + \delta_{2s}' ) \left \| \frac{\vect{\widetilde{r}}}{m_1} \right \|_2^2.
\label{equ:PairUnif1}
\end{align}
In order to get an upper bound and lower bound of $\left \| \frac{\vect{\widetilde{r}}}{m_1} \right \|_2$, we define $\vect{e}_{\widetilde{\vect{n}}}\in \mathbb{C}^{\bar{N}_j}$ as a standard basis vector  with all $0$ entries except for a $1$ at $\widetilde{\vect{n}}$, and  $R_j\in \mathbb{C}^{\widetilde{N}_j\times \bar{N}_j}$ as 
\begin{equation*}
(R_j)_{\vect{n}_2,\vect{n}_1}:=\begin{cases}
r_{(\vect{n}_1,\vect{n}_2)} \qquad \text{ if } 
(\vect{n}_1, \vect{n}_2) \in {\rm supp}(\vect{r})\\
0 \qquad \text{ otherwise }.
\end{cases}
\end{equation*}
Note that $R_j$ contains at most $2s$ nonzero elements since $h$ is $2s$-sparse, and $\widetilde{\vect{n}}$ is any element in $[M]^{j+1}$. Set $\mathcal{Q}:=\{\widetilde{\vect{n}}\} \cup \{\vect{n}_1 \in [M]^{j+1}~|~\exists \vect{n}_2 \in [M]^{D-j-1} \text{ such that } (\vect{n}_1,\vect{n}_2)\in {\rm supp}(\vect{r}) \}$ with a fixed $\widetilde{\vect{n}}\in  [M]^{j+1}$. Thus, the cardinality of $\mathcal{Q}$ is at most $2s+1$. Orderings of indices $\vect{n}_1$ and $\vect{n}_2$ depend on the column orderings of $B_j$ and $C_j$, respectively. We note that the $\widetilde{\vect{n}}$-th column of $R_j$ is $\vect{r}_{j;(\widetilde{\vect{n}},\cdots)}$. Both bounds of $\left \| \frac{\vect{\widetilde{r}}}{m_1} \right \|_2$ are found as follows
\begin{align*}
\left\| \frac{\vect{\widetilde{r}}}{m_1} -{\vect{r}}_{j;(\widetilde{\vect{n}},\cdots)} \right\|_2 &=\left\| R_j(B_j)_{\mathcal{Q}}^* (B_j)_{\mathcal{Q}} \vect{e}_{\widetilde{\vect{n}}} -  R_j\vect{e}_{\widetilde{\vect{n}}} \right\|_2 \\
&\leq \| R_j\|_{2\rightarrow 2} \|(B_j)_{\mathcal{Q}}^* (B_j)_{\mathcal{Q}}- I\|_{2 \rightarrow 2}  \|\vect{e}_{\widetilde{\vect{n}}}\|_2 \\
&\leq  \widetilde{\delta}_{2s+1} \| R_j\|_F  \\
&=  \widetilde{\delta}_{2s+1} \| \vect{r}\|_2 
\end{align*}
and therefore, 
\begin{align}
&\left| \left\| \frac{\vect{\widetilde{r}}}{m_1} \right\|_2-\left\|{\vect{r}}_{j;(\widetilde{\vect{n}},\cdots)} \right\|_2 \right| 
\leq \widetilde{\delta}_{2s+1} \|\vect{r}\|_2,\nonumber\\
&\left\| {\vect{r}}_{j;(\widetilde{\vect{n}},\cdots)} \right\|_2-\widetilde{\delta}_{2s+1} \|\vect{r}\|_2 \leq 
\left\| \frac{\vect{\widetilde{r}}}{m_1} \right\|_2 \leq 
\left\|{\vect{r}}_{j;(\widetilde{\vect{n}},\cdots)} \right\|_2+\widetilde{\delta}_{2s+1}\|\vect{r}\|_2.
\label{equ:PairUnif0} 
\end{align}
Combining \eqref{equ:PairUnif1} and \eqref{equ:PairUnif0}, we reach the conclusion in \eqref{equ:PairUnif3}.
\end{proof}
\begin{lemma}
Let $\mathcal{H}_{2s}$ be the set of all functions $h$ whose coefficient vectors are $2s$-sparse and let $\vect{r}_h \in \mathbb{C}^N$ denote the coefficient vector for each $h \in \mathcal{H}_{2s}$.
Let $j$ be any integer such that $j\in [D]\setminus \{0\}$, and $\widetilde{\delta}$ and $\delta'$ be chosen from $(0,1)$. Given $\alpha>1$, assume that the restricted ${B}_j$ and ${C}_j$ satisfy {RIP} with $\widetilde{\delta}_{2s+1} \leq \widetilde{\delta} \leq \frac{1}{\check{c}\sqrt{s}}$ for some $\check{c}> \sqrt{2}$ and $\delta'_{2s} \leq \delta' \in (0,1)$, respectively, with $\sqrt{\frac{1+\delta'}{1-\delta'}}<\frac{\check{c}}{\alpha}-1$. Then, denoting $\vect{r}:=\vect{r}_h$ for simplicity, one has the set $\mathcal{P}_{j} \supset \left\{\widetilde{\vect{n}}\in [M]^{j+1} ~\big|~ \|\vect{r}_{j;(\widetilde{\vect{n}},\cdots)}\|_2^2 \geq \frac{\|\vect{r}\|_2^2}{\alpha^2 s} \right\}$  of cardinality $2s$ resulting from Algorithm \ref{alg:Pairing} if $\mathcal{P}_{j-1} \supset \left\{\hat{\vect{n}}\in [M]^{j} ~\big|~ \|\vect{r}_{j-1;(\hat{\vect{n}},\cdots)}\|_2^2 \geq \frac{\|\vect{r}\|_2^2}{\alpha^2 s}\right\}$.
\label{lem:induction2}
\end{lemma}

\begin{proof}
By assumption that $\mathcal{P}_{j-1} \subset \times_{i\in[j]} \mathcal{N}_i$ contains all prefixes in $\big\{\hat{\vect{n}}\in [M]^{j}~\big|~ \|\vect{r}_{j-1;(\hat{\vect{n}},\cdots)}\|_2^2 \allowbreak \geq \frac{\|\vect{r}\|_2^2}{\alpha^2 s}\big\}$, ~$\mathcal{P}_{j-1} \times \mathcal{N}_{j}$  contains all possible prefixes $\widetilde{\vect{n}} \in [M]^{j+1}$ with $\|\vect{r}_{j;(\widetilde{\vect{n}},\cdots)}\|_2^2 \geq \frac{\|\vect{r}\|_2^2}{\alpha^2 s}$ by the definition of $\mathcal{P}_j$ and $\mathcal{N}_j$ for all $j\in[D]$.
If $\left\| {\vect{r}}_{j;(\widetilde{\vect{n}},\cdots)} \right\|_2=0$, i.e., there is no $\vect{n}_2$ such that $(\widetilde{\vect{n}},\vect{n}_2) \in {\rm supp}(\vect{r})$, then from Lemma \ref{lem:induction1},
\begin{equation*}
0 \leq E_{j;(\widetilde{\vect{n}},\cdots)}
\leq (1 + \delta' ) \left( \widetilde{\delta}\|\vect{r}\|_2 \right)^2 
\leq (1 + \delta' ) \left( \frac{\|\vect{r}\|_2}{\check{c}\sqrt{s}} \right)^2.
\end{equation*}
On the other hand, if $\left\| {\vect{r}}_{j;(\widetilde{\vect{n}},\cdots)} \right\|_2^2 \geq \frac{\|\vect{r}\|_2^2}{\alpha^2 s}$, i.e., there is $\vect{n}_2$ such that $(\widetilde{\vect{n}},\vect{n}_2)\in {\rm supp}(\vect{r})$, then 
\begin{align}
(1 - \delta' ) \left( \left\| {\vect{r}}_{j;(\widetilde{\vect{n}},\cdots)} \right\|_2- \frac{\|\vect{r}\|_2}{\check{c}\sqrt{s}}  \right)^2 \nonumber
& \leq (1 - \delta' ) \left( \left\| {\vect{r}}_{j;(\widetilde{\vect{n}},\cdots)} \right\|_2-\widetilde{\delta} \|\vect{r}\|_2 \right)^2 \nonumber \\
& \leq E_{j;(\widetilde{\vect{n}},\cdots)} \leq (1 + \delta' ) \left( \left\|{\vect{r}}_{j;(\widetilde{\vect{n}},\cdots)} \right\|_2+\widetilde{\delta}\|\vect{r}\|_2 \right)^2.\nonumber
\end{align}
In order to distinguish nonzero $\left\| {\vect{r}}_{j;(\widetilde{\vect{n}},\cdots)} \right\|_2 \geq \frac{\|\vect{r}\|_2}{\alpha \sqrt{s}}$ from zero $\left\| {\vect{r}}_{j;(\widetilde{\vect{n}},\cdots)} \right\|_2$, we should have
\begin{align}
(1 + \delta' ) \left( \frac{\|\vect{r}\|_2}{\check{c}\sqrt{s}}  \right)^2
&< (1 - \delta' ) \left( \frac{\|\vect{r}\|_2}{\alpha \sqrt{s}}-\frac{\|\vect{r}\|_2}{\check{c}\sqrt{s}}  \right)^2 \nonumber
\end{align}
which is implied by 
\begin{equation}
\sqrt{\frac{1+\delta'}{1-\delta'}}<\frac{\check{c}}{\alpha}-1, \nonumber
\end{equation}
as in the assumption.
Since estimates of zero energy and nonzero energy greater than $\frac{\|\vect{r}\|_2^2}{\alpha^2 s}$ are separated, choosing $2s$ prefixes with largest estimates $E_{\vect{\widetilde{n}}}$ guarantees that it contains all prefixes with energy greater than $\frac{\|\vect{r}\|_2^2}{\alpha^2 s}$.
\end{proof}
\begin{mytheorem}
Let $\mathcal{H}_{2s}$ be the set of all functions $h$ whose coefficient vectors are $2s$-sparse and let $\vect{r}_h \in \mathbb{C}^N$ denote the coefficient vector for each $h \in \mathcal{H}_{2s}$. We assume that we have $\mathcal{N}_j$ for all $j\in [D]$. 
Let $\alpha > 1$, $\widetilde{\delta} \leq \frac{1}{\check{c}\sqrt{s}}$ for some $\check{c}> \sqrt{2}$ and $\delta'\in (0,1)$, satisfying 
$\sqrt{\frac{1-\delta'}{1+\delta'}} < \frac{\check{c}}{\alpha }-1$,
and $p\in(0,1)$. If 
$$m_1\geq \bar{\alpha} K^2 \left(\widetilde{\delta}\right)^{-2}s \max\left\{ \log^2(4s)\cdot d \cdot \log \left( \frac{\mathbbm{e} M D^{1+\frac{1}{d\log^3(4s)}}}{d} \right) \log(9m_1), \log(2D/p)\right\} \text{ and}$$
$$m_2 \geq \bar{\beta} K^2 \left(\delta'\right)^{-2} s \max \left\{ \log^2(4s)\cdot d  \cdot \log \left( \frac{\mathbbm{e}MD^{1+\frac{1}{d\log^3(4s)}}}{d}\right) \log (9m_2), \log(2D/p)\right\} , $$
for  absolute constants $\bar{\alpha}$ and $\bar{\beta}$, then denoting $\vect{r}:=\vect{r}_h$ for simplicity, Algorithm  \ref{alg:Pairing} finds $\mathcal{P} \supset \big\{ \vect{n} \in [M]^D \allowbreak ~|~ |r_{\vect{n}}|^2 \geq \frac{\|\vect{r}\|_2^2}{\alpha^2 s}\big \}$ of $|\mathcal{P}|=2s$ with probability at least $1-p$.
\label{thm:pairing}
\end{mytheorem}

\begin{proof}
Given $m_1$ and $m_2$, by Theorem \ref{lem:RIPforD}, the probability of either $B_j$ or $C_j$ not satisfying $\widetilde{\delta}_{2s+1}<\widetilde{\delta}$ or $\delta'_{2s}<\delta'$ respectively is at most $\frac{p}{2D}$  for each $j\in [D]\setminus \{0\}$, and thus the union bound over all $j$ yields the failure probability at most $\frac{p(D-1)}{D}<p$. That is, Theorem \ref{lem:RIPforD} ensures that $B_j$ and $C_j$ have RIP uniformly for all $j\in [D]\setminus \{0\}$ with probability at least $1-p$. Repeatedly applying Lemma \ref{lem:induction2} yields the final $\mathcal{P}(=\mathcal{P}_{D-1}) \supset \left\{{\vect{n}}\in [M]^D~\big|~ |{r}_{\vect{n}}|^2 \geq \frac{\|\vect{r}\|_2^2}{\alpha^2 s}\right\}$ of cardinality $2s$ by combining the fact that $\mathcal{P}_0(=\mathcal{N}_0) \supset \left\{\widetilde{\vect{n}}\in [M] ~\big|~ \|\vect{r}_{(\widetilde{\vect{n}},\cdots)}\|_2^2 \geq \frac{\|\vect{r}\|_2^2}{\alpha^2 s} \right\}$. 
\end{proof}

\subsection{Support Identification}
\label{EntryId+pairing}

In this section, it remains to combine the results of entry identification and pairing processes in order to give the complete support identification algorithm and to prove Lemma \ref{lem:replCosamp} which is the main ingredient of Theorem \ref{mainThm} in Section \ref{sec:theorySuppId}.
The support identification starts with the entry identification providing $\mathcal{N}_j$, $j\in[D]$ as outputs, and in turn, the pairing takes $\mathcal{N}_j$, $j\in[D]$ as inputs and outputs $\mathcal{P}$ of cardinality $2s$ containing $\left\{ \vect{n} \in [M]^D ~|~ |r_{\vect{n}}|^2 \geq \frac{\|\vect{r}\|_2^2}{\alpha^2 s}\right\}$. Accordingly, we can get the following result.

\begin{lemma}
The set $\mathcal{P}$ contains at most $2s$ index vectors satisfying 
\begin{equation*}
\|\vect{r}_{\mathcal{P}^c}\|_2\leq \sqrt{2s\frac{\|\vect{r}\|_2^2}{\alpha^2s}}=\frac{\sqrt{2}\|\vect{r}\|_2}{\alpha},
\end{equation*}
where $\vect{r}_{\mathcal{P}^c}$ is the vector of $\vect{r}$ restricted to the complement of $\mathcal{P}$.
\label{lem:suppEnergy}
\end{lemma}

\begin{proof}
Note that $\vect{r}$ is $2s$-sparse and by Theorem \ref{thm:pairing}  the squared magnitude of each $r_{\vect{n}}$ at $\mathcal{P}^c$ is less than $\frac{\|\vect{r}\|_2^2}{\alpha^2 s}$ so that we obtain the desired result. 
\end{proof}
Finally, we are ready to prove Lemma \ref{lem:replCosamp} in order to complete the analysis of support identification.

\begin{proof}[Proof of Lemma \ref{lem:replCosamp}]
By choosing $\alpha=7$ and $\mathcal{P}=\Omega$ in Lemma \ref{lem:suppEnergy}, we obtain the desired upper bound of $\|(\vect{r}_h)_{\Omega^c}\|_2$ with probability at least $1-2p$ given the grids $\mathcal{G}^I$ and $\mathcal{G}^P$. The union bound of the failure probabilities $p$ of Lemma \ref{Thm:FiniteSetMedtoUn3} and Theorem \ref{thm:pairing} gives the desired probability. It remains to demonstrate the sampling complexity combining $\left|\mathcal{G}^I\right|$ and $\left|\mathcal{G}^P\right|$, and the runtime complexity combining Algorithms \ref{alg:Entry} and \ref{alg:Pairing}. The first term $\left|\mathcal{G}^I\right|$ is $m \mathcal{L}' D$ where $m$ comes from Lemma \ref{Thm:FiniteSetMedtoUn3}, $\mathcal{L}'$ is defined as in  Section \ref{sec:theorySuppId}
, and $D$, the number of changes in $j\in[D]$, therein. We emphasize that $m$ samples are utilized repeatedly in order to implicitly construct $L$ proxy functions combined with different Gaussian weights so that $L$ does not affect the sampling complexity but affects the runtime complexity. The second term $\left|\mathcal{G}^P \right|$ is $m_1m_2(D-2)+m_1$ where $m_1$ and $m_2$ are from Theorem \ref{thm:pairing}, and $D-2$ is the number of changes in $j \in [D-1]\setminus \{0\}$. We remind readers that $C_{D-1}=1$ implied by $\mathcal{Z}_{D-1}=\emptyset$ so that $m_1$ samples are utilized instead of $m_1m_2$ when $j=D-1$. Now, we consider the runtime complexity. The first term in runtime complexity is $\mathcal{O}(mL\mathcal{L}D)$ from Algorithm \ref{alg:Entry} where $mL$ computations are taken to implicitly construct the $L$ proxy functions, $\mathcal{O}(\mathcal{L})$ is defined as in Section \ref{sec:theorySuppId}
, and $D$ comes from the for loop from Algorithm \ref{alg:Entry}. The second term in runtime complexity is $4s^2 \left( m_1m_2(D-2)+m_1 \right)+ 4s^2 m_1\sum_{j=1}^{D-1}(j+1)$ from Algorithm \ref{alg:Pairing} since $4s^2$ energy estimates are calculated using the $m_1m_2$ samples for each $j\in[D-1]\setminus\{0\}$ and $m_1$ samples for $j=D-1$, and the evaluations of $T_{j;\widetilde{\vect{n}}}(\vect{w}_{j,\ell})$ are calculated for all $\widetilde{\vect{n}} \in \mathcal{P}_{j}$, $\forall j\in[D]\setminus\{0\}$. 
Here, it is assumed that it takes $\mathcal{O}(1)$ runtime to evaluate each $i^{\rm th}$ component $T_{i;\widetilde{{n}}_i}(({w}_{j,\ell})_i)$ of $T_{j;\widetilde{\vect{n}}}(\vect{w}_{j,\ell})$.
\end{proof}

\section{Empirical Evaluation} \label{Empirical Evaluation}
\setcounter{equation}{0}

In this section Algorithm~\ref{alg:main} is evaluated numerically and compared to CoSaMP \cite{needell2009cosamp}, its superlinear-time progenitor.  Both Algorithm~\ref{alg:main} and CoSaMP were implemented in MATLAB for this purpose.  All code used to produce the plots below is publicly available at \cite{MarksCodePage}.  

\subsection{Experimental Setup}

We consider two kinds of tensor product basis functions below:  Fourier and Chebyshev.  In both cases each parameter, $M$, $D$, and $s$, is changed while the others remain fixed so that we can see how each parameter affects the runtime, sampling number, memory usage, and error of both Algorithm~\ref{alg:main} and CoSaMP.  For all experiments below $d = D$ so that $\mathcal{I} = [M]^D$. 
Every data point in every plot below was created using $100$ different randomly generated trial signals, $f$, of the form
\begin{equation}
f(\vect{x})=\sum_{\vect{n}\in\mathcal{S}} c_{\vect{n}} T_{\vect{n}}(\vect{x}),
\label{equ:TestfuncEE}
\end{equation}
where each function's support set, $\mathcal{S}$, contained $s$ index vectors $\vect{n} \in [M]^D$ each of which was independently chosen uniformly at random from $[M]^D$, 
and where each function's coefficients  $c_{\vect{n}}$ were each independently chosen uniformly at random from the unit circle in the complex plane (i.e., each $c_{\vect{n}} = \mathbbm{e}^{\mathbbm{i}\theta}$ where $\theta$ is chosen uniformly at random from $[0,2\pi]$).  In the Fourier setting the basis functions $T_{\vect{n}}(\vect{x})$ in \eqref{equ:TestfuncEE} were chosen as per \eqref{equ:FourierProd}, and in the Chebyshev setting as per \eqref{equ:ChebProd}.  

Below a {\it trial} will always refer to the execution of Algorithm~\ref{alg:main} and/or CoSaMP on a particular randomly generated trial function $f$ in \eqref{equ:TestfuncEE}.  A {\it failed trail} will refer to any trial where either CoSaMP or Algorithm~\ref{alg:main} failed to recover the correct support set $\mathcal{S}$ for $f$.  Herein the parameters of both Algorithm~\ref{alg:main} and CoSaMP were tuned to keep the number of failed trials down to less than 10 out of the total 100 used to create every data point in every plot.  Finally, in all of our plots Algorithm~\ref{alg:main} is graphed with red, and CoSaMP with blue. 

\subsection{Experiments with the Fourier Basis for $\mathcal{D} = [0,1]^D$}
\label{sec:5.2}

\begin{figure}[t!] 
\subfloat[\label{fig:1.1}]
  {\includegraphics[width=.5\linewidth]{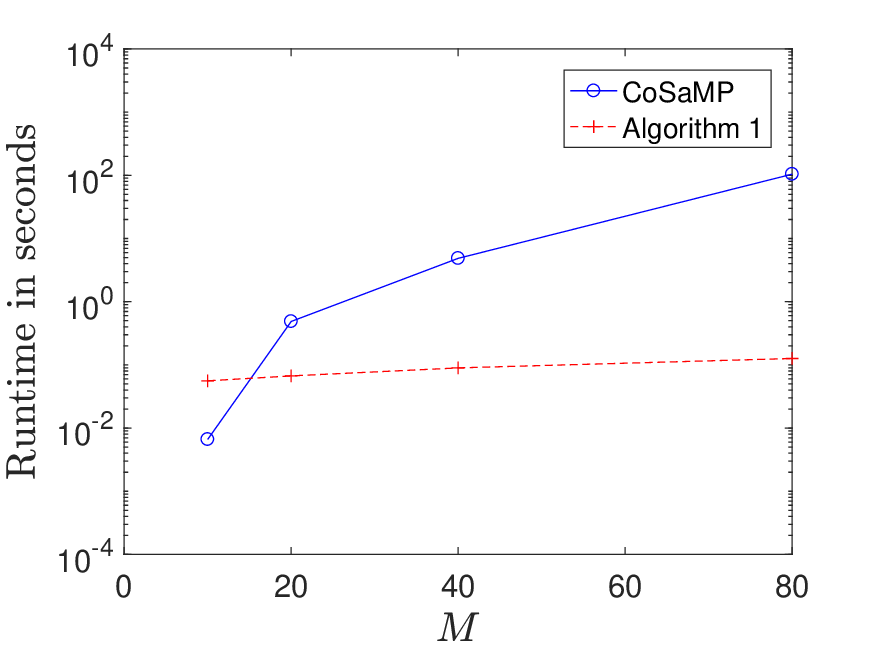}}
  \hfill
\subfloat[\label{fig:1.2}]
  {\includegraphics[width=.5\linewidth]{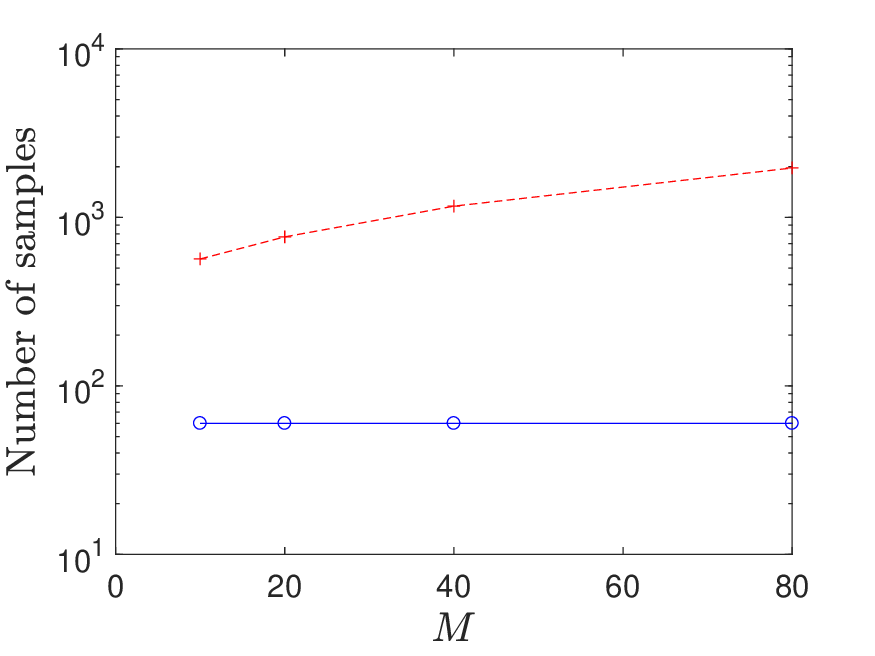}}
\caption{Fourier basis, $M\in\{10,20,40,80\},\ D=4,\ s=5$}
\label{fig:1}
\end{figure}

\begin{figure}[t!]
\subfloat[\label{fig:2.1}]
  {\includegraphics[width=.5\linewidth]{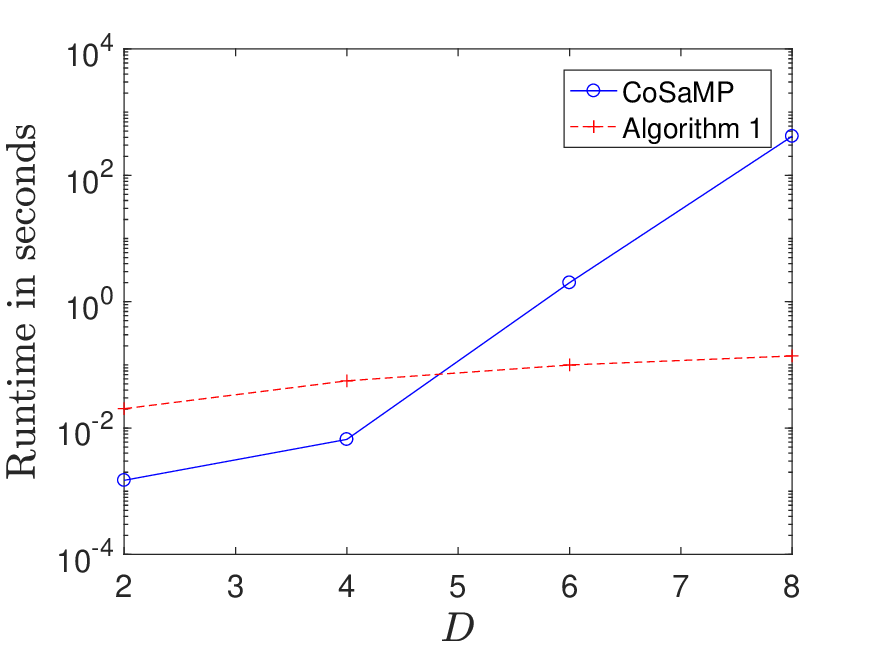}}
  \hfill
\subfloat[\label{fig:2.2}]
  {\includegraphics[width=.5\linewidth]{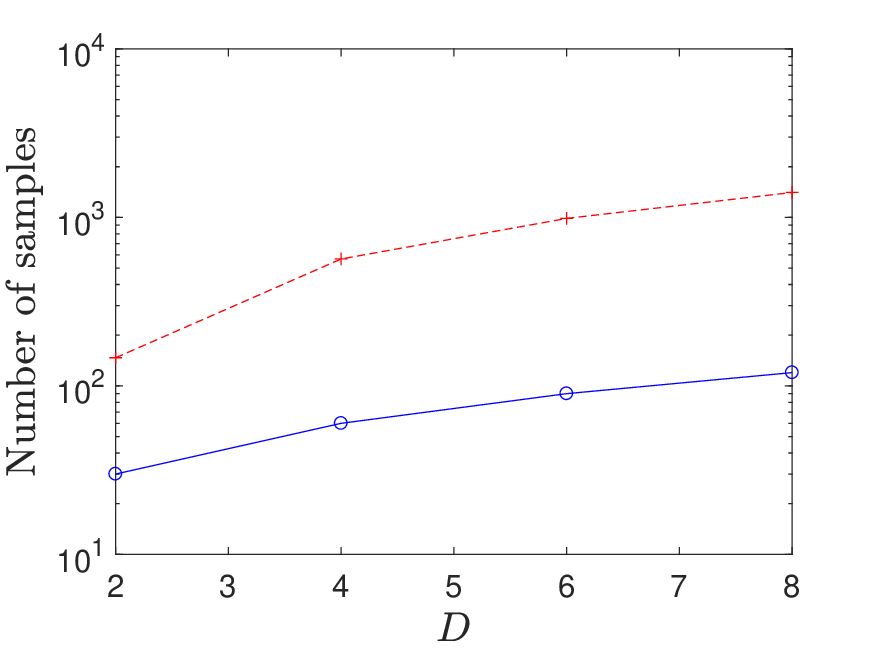}}
\caption{Fourier basis, $M=10,\ D=\{2,4,6,8\},\ s=5$}
\label{fig:2}
\end{figure}

\begin{figure}[t!]
\subfloat[\label{fig:3.1}]
  {\includegraphics[width=.5\linewidth]{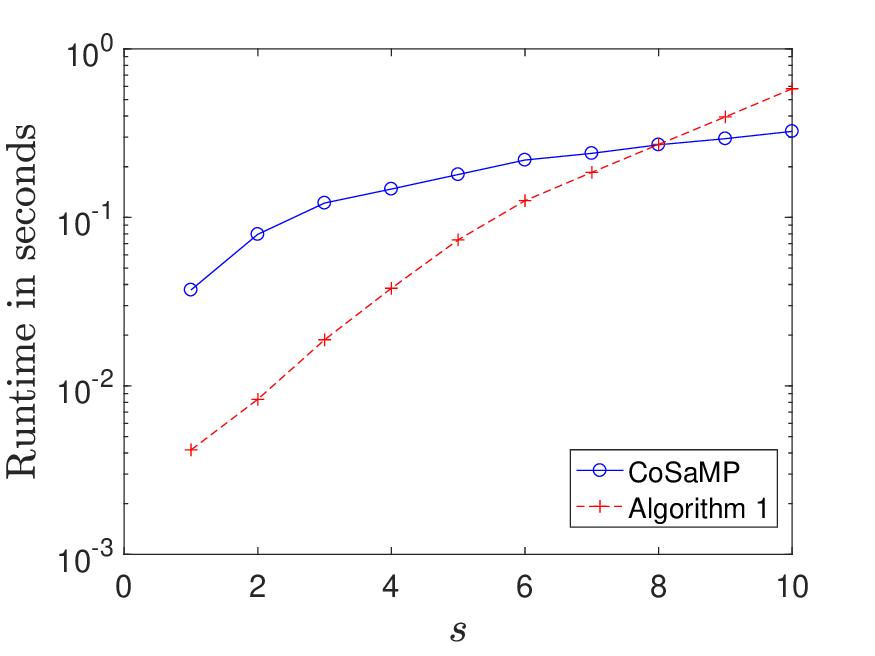}}
  \hfill
\subfloat[\label{fig:3.2}]
  {\includegraphics[width=.5\linewidth]{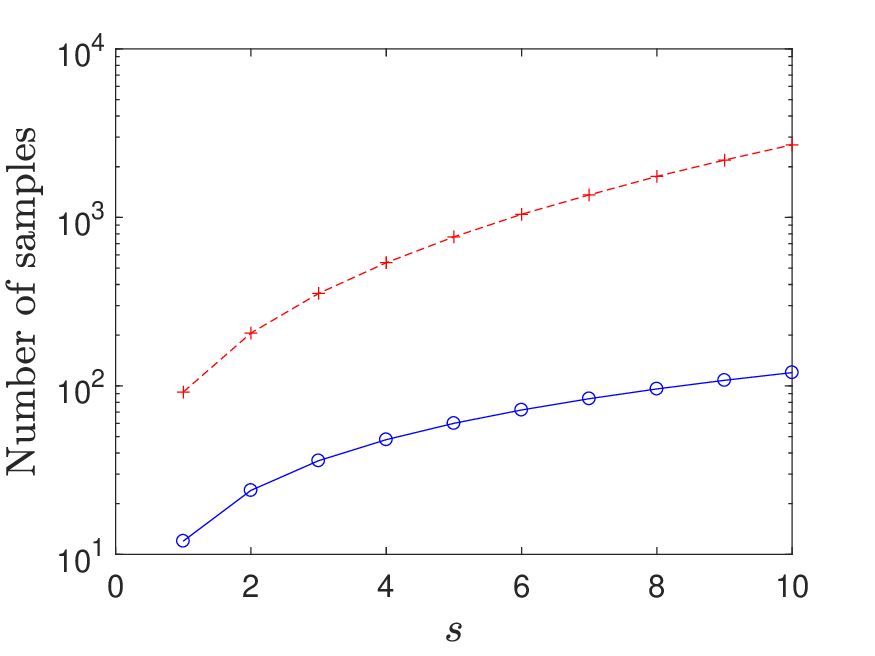}}
\caption{Fourier basis, $M=20,\ D=4,\ s=\{1,2,3,\cdots,10\}$}
\label{fig:3}
\label{fig: red line: our method, blue line: CoSaMP}
\end{figure}

\begin{figure}[t!]
	\includegraphics[width=1\linewidth]{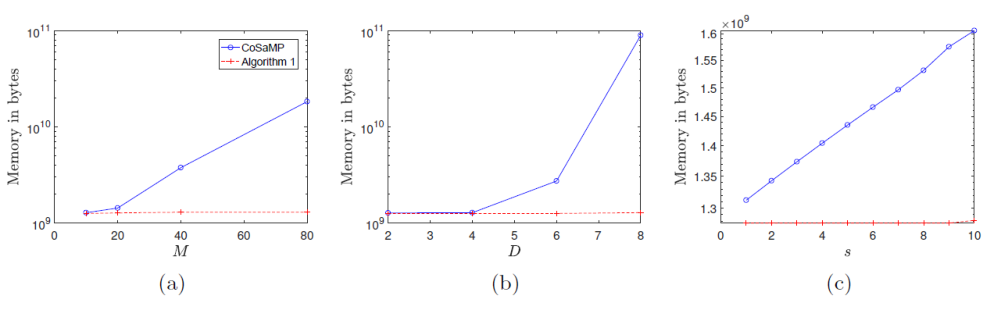}
	\caption{Fourier basis, memory usage when varying $M$, $D$ and $s$}
	\label{fig:15}
	\label{fig: red line: our method, blue line: CoSaMP}
\end{figure}

In this section we consider the Fourier tensor product basis 
\begin{equation}
T_{\vect{n}}(\vect{x}) := \prod_{j=0}^{D-1} \mathbbm{e}^{2\pi \mathbbm{i} n_j x_j}
\label{equ:FourierProd}
\end{equation}
whose orthogonality measure is the Lebesgue measure on $\mathcal{D}=[0,1]^D$.  In Figures \ref{fig:1}, \ref{fig:2}, \ref{fig:3} and \ref{fig:15}, results are shown for approximating Fourier-sparse trial functions \eqref{equ:TestfuncEE} using noiseless samples $\vect{y}$.  In Figures \ref{fig:1} and 6a, the parameter $M$ changes over the set $\{10,20,40,80\}$ while $D=4$ and $s=5$ are held constant.  In Figure \ref{fig:1.1}, the average runtime (in seconds) is shown as $M$ changes. The average here is calculated over all $100$ trials at each data point excluding any failed trials.  As we can see, the runtime of Algorithm~\ref{alg:main} grows very slowly as $M$ grows, whereas the runtime grows fairly quickly for CoSaMP since its measurement matrix's size increases significantly as $M$ grows. Figure \ref{fig:1.2} shows the number of samples used by both CoSaMP and Algorithm~\ref{alg:main}. We can see that Algorithm~\ref{alg:main} requires more samples due mainly to its support identification's pairing step.  
On the other hand, we can see in Figure 6a that the memory usage of CoSaMP grows very rapidly compared to the slow growth of Algorithm~\ref{alg:main}'s memory usage.  This exemplifies the tradeoff between Algorithm~\ref{alg:main} and CoSaMP -- Algorithm~\ref{alg:main} uses more samples than CoSaMP in order to reduce its runtime complexity and memory usage for large $D$ and $M$.
Finally, both methods produce outputs whose average errors (over the trials where they don't fail) are on the order of $10^{-15}$ to $10^{-14}$, which is also observed in all other experiments in Sections \ref{sec:5.2}, \ref{sec:5.3} and \ref{sec:5.4}.

In Figures \ref{fig:2} and 6b, the number of dimensions, $D$, changes while both $M=10$ and $s=5$ are held fixed. Here, we can clearly see the advantage of Algorithm~\ref{alg:main} for functions of many variables. The runtime and memory usage of CoSaMP blow up quickly as $D$ increases due to the gigantic matrix-vector multiplies it requires to identify support.  Algorithm~\ref{alg:main}, on the other hand, shows much slower growth in runtime and memory usage. When $D=10$, for example, CoSaMP requires terabytes of memory whereas Algorithm~\ref{alg:main} requires only a few gigabytes. In Figures \ref{fig:3} and 6c, $s$ varies in $\{1,2,3,\cdots,10\}$ while $M=20$ and $D=4$ are fixed. Since Algorithm~\ref{alg:main} has $\mathcal{\tilde{O}}(s^5)$ scaling\footnote{The $\mathcal{\tilde{O}}$ complexity notation here neglects all logarithmic factors while simultaneously holding $D$, $K$, $d$, $\eta$, $\| {\bf c}_f \|_2$ and $\mathcal{L}$ constant.} in runtime due to its pairing step, it suffers as sparsity increases more quickly than CoSaMP does. Note that the crossover point is around $s=8$, so that  Algorithm~\ref{alg:main} appears to be slower than CoSaMP for all $s>8$ when $M = 20$ and $D = 4$. Though ``only polynomial in $s$'', it is clear from these experiments that the runtime scaling of Algorithm~\ref{alg:Pairing} in $s$ needs to be improved before the methods proposed herein can become truly useful in practice.

\subsection{Experiments with the Chebyshev Basis for $\mathcal{D} = [-1,1]^D$}
\label{sec:5.3}

\begin{figure}[t!] 
\subfloat[\label{fig:4.1}]
  {\includegraphics[width=.5\linewidth]{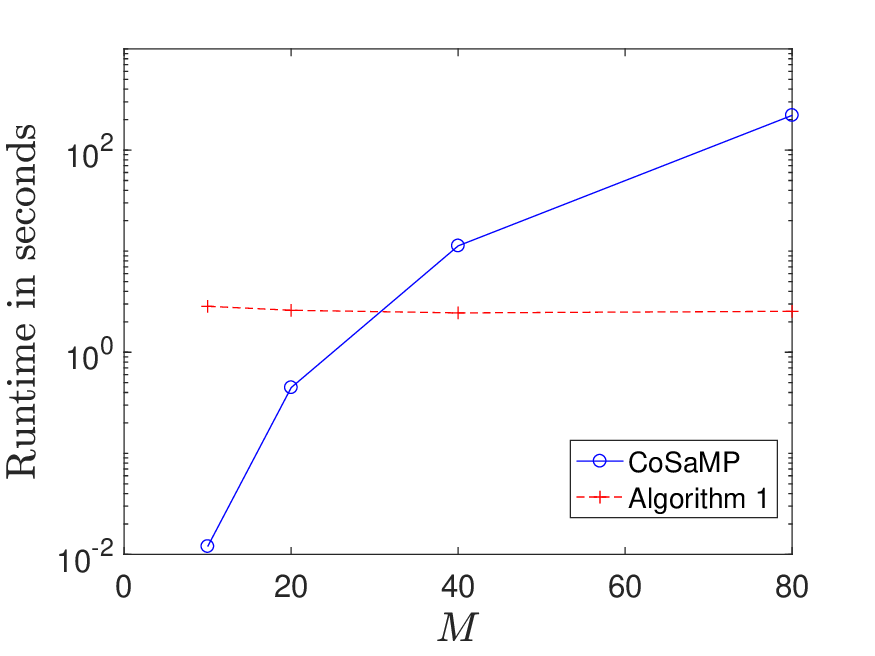}}
  \hfill
\subfloat[\label{fig:4.2}]
  {\includegraphics[width=.5\linewidth]{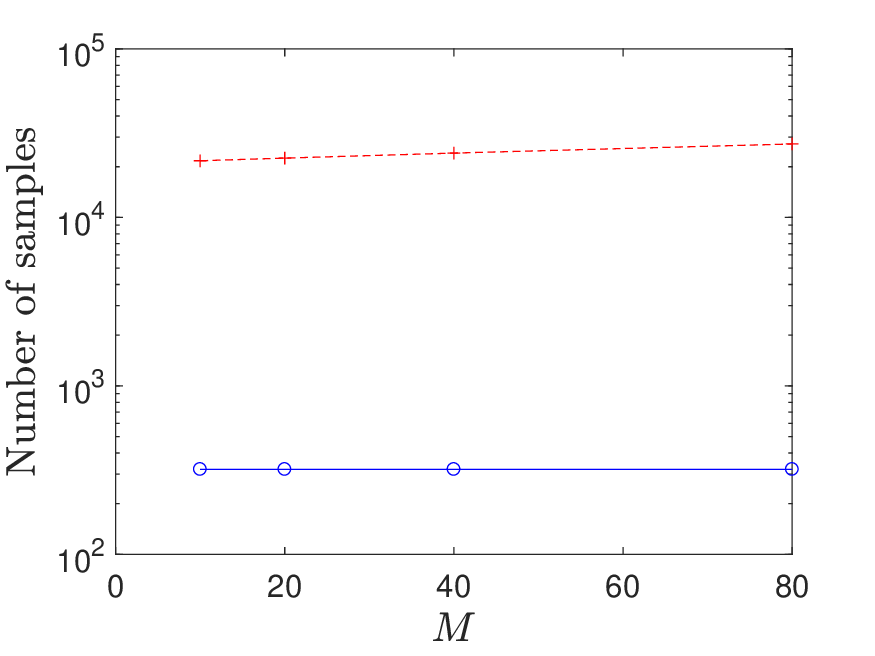}}
\caption{Chebyshev basis, $M\in\{10,20,40,80\},\ D=4,\ s=5$}
\label{fig:4}
\end{figure}

\begin{figure}[t!]
\subfloat[\label{fig:5.1}]
  {\includegraphics[width=.5\linewidth]{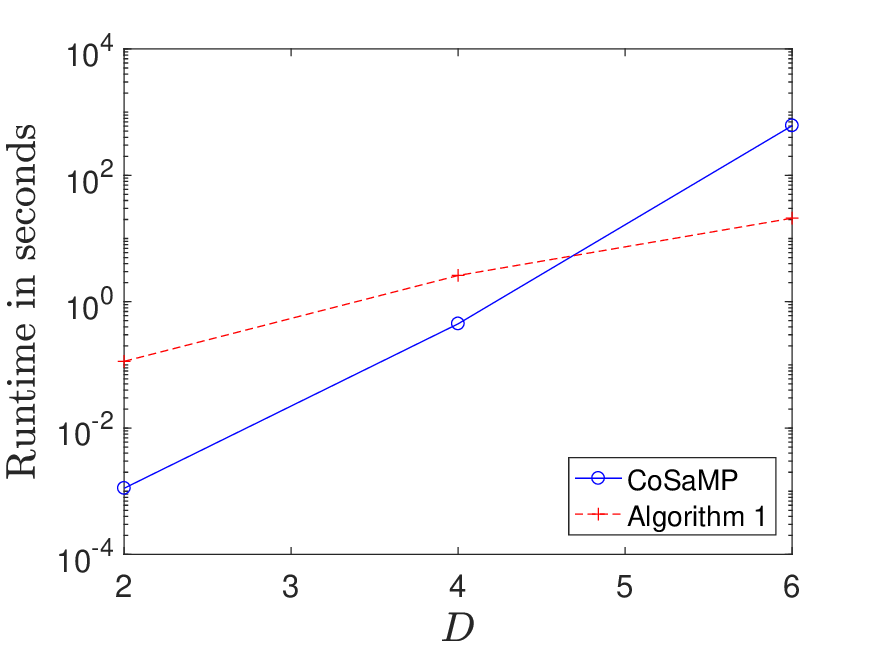}}
  \hfill
\subfloat[\label{fig:5.2}]
  {\includegraphics[width=.5\linewidth]{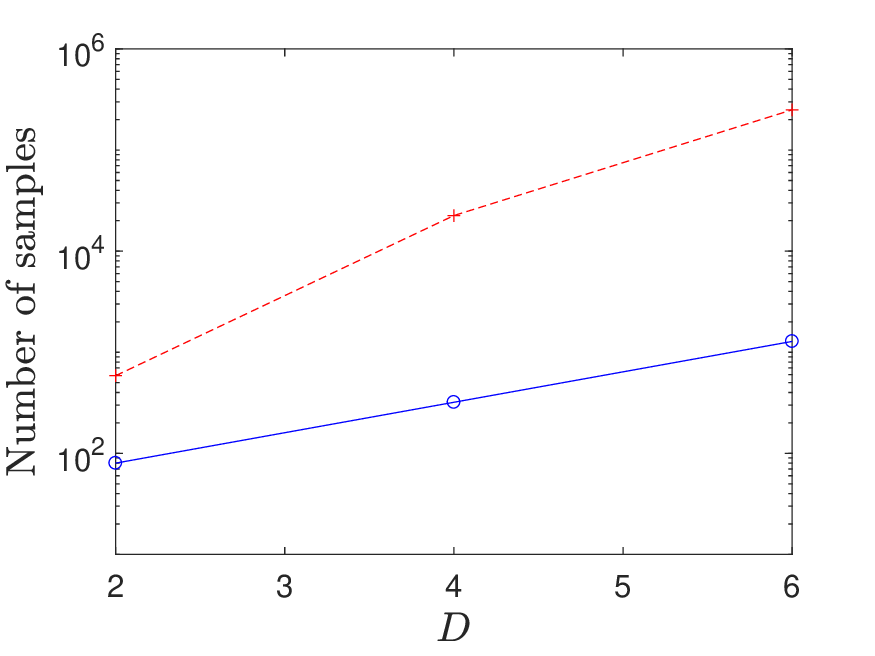}}
\caption{Chebyshev basis, $M=20,\ D=\{2,4,6\},\ s=5$}
\label{fig:5}
\end{figure}

\begin{figure}[t!]
\subfloat[\label{fig:6.1}]
  {\includegraphics[width=.5\linewidth]{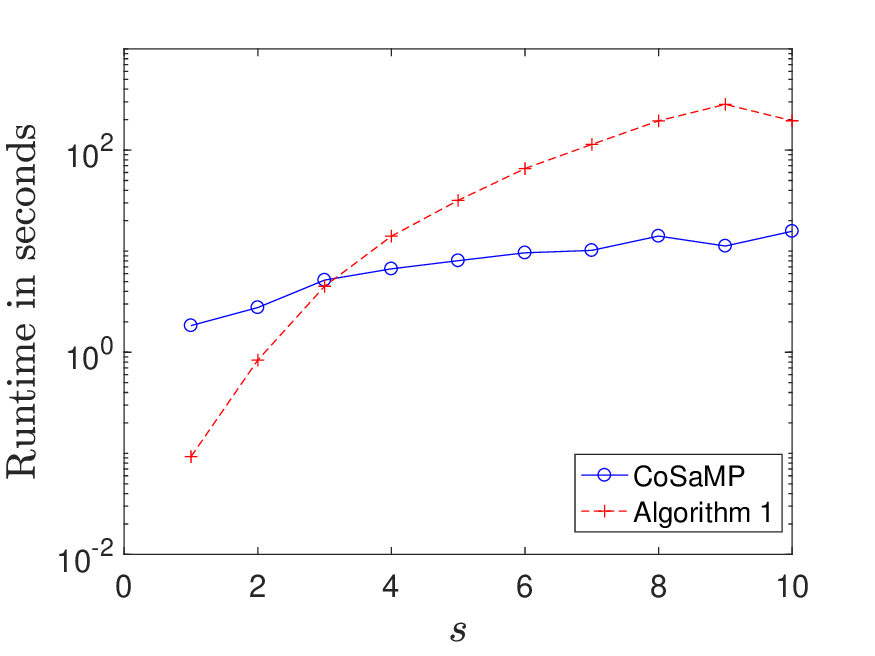}}
  \hfill
\subfloat[\label{fig:6.2}]
  {\includegraphics[width=.5\linewidth]{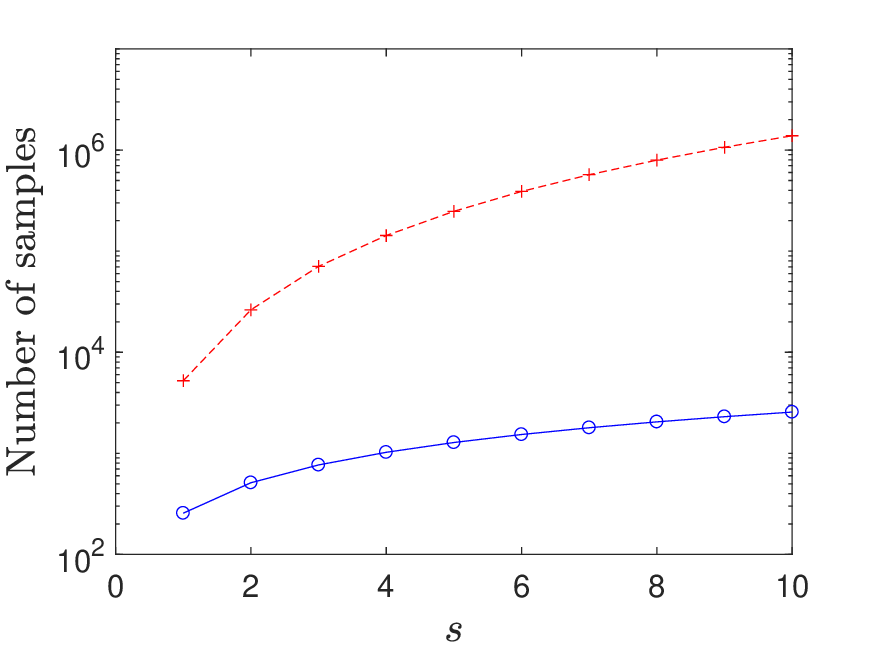}}
\caption{Chebyshev basis, $M=10,\ D=6,\ s=\{1,2,3,\cdots,10\}$}
\label{fig:6}
\label{fig: red line: our method, blue line: CoSaMP}

\end{figure}

\begin{figure}[t!]
	\includegraphics[width=1\linewidth]{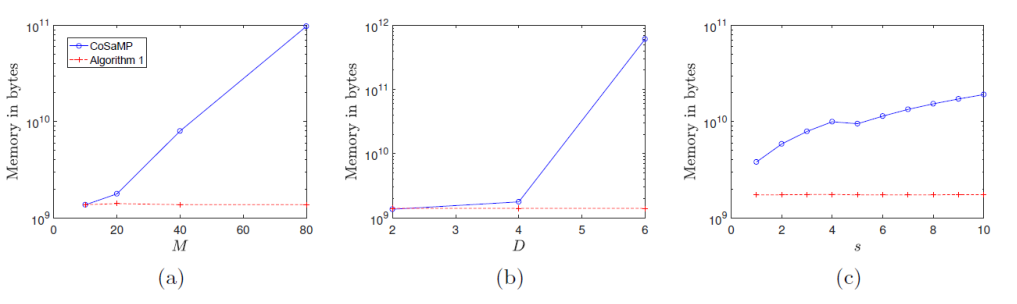}
	\caption{Chebyshev basis, memory usage when varying $M$, $D$, and $s$}
	\label{fig:16}
	\label{fig: red line: our method, blue line: CoSaMP}
	
\end{figure}

In this section we consider the Chebyshev tensor product basis 
\begin{equation}
T_{\vect{n}}(\vect{x}) := {2^{\frac{1}{2}\| \vect{n} \|_0}} \prod_{j=0}^{D-1}\cos \left(n_j \arccos (x_j) \right) 
\label{equ:ChebProd}
\end{equation}
whose orthogonality measure is $d\vect{\nu}=\otimes_{j\in[D]}\frac{dx_j}{\pi \sqrt{1-x_j^2}}$ on $\mathcal{D} = [-1,1]^D$.  Runtime and sampling complexity graphs are provided in Figures \ref{fig:4}, \ref{fig:5} and \ref{fig:6} as $M$, $D$, and $s$ vary, respectively. In Figure \ref{fig:16}, memory usage is also graphed for each $M$, $D$ and $s$ variation. Since this Chebyshev product basis has a BOS constant of $K = 2^{D/2}$, both CoSaMP and Algorithm~\ref{alg:main} suffer from a mild exponential growth in sampling, runtime, and memory complexity as $D$ increases (recall that $d = D$ for these experiments).  This leads to markedly different overall performance for the Chebyshev basis than what is observed for the Fourier basis where $K=1$.  A reduction in performance from the Fourier case for both methods is clearly visible, e.g., in Figure~\ref{fig:5}.  Nonetheless, Algorithm~~\ref{alg:main} demonstrates the expected reduced runtime and sampling complexity dependence on $M$ and $D$ over CoSaMP in Figures \ref{fig:4} and \ref{fig:5}, as well as a striking reduction in its required memory usage over CoSaMP  in Figure \ref{fig:16} even when its runtime complexity is worse in Figure \ref{fig:6.1}.  Unfortunately, the $\mathcal{\tilde{O}}(s^5)$ runtime dependance of Algorithm~\ref{alg:Pairing} on sparsity is again clear in Figure \ref{fig:6.1} leading to a crossover point of Algorithm~\ref{alg:main} with CoSaMP at only $s=3$ when $M = 10$ and $D = 6$.  This again clearly marks the pairing process of Algorithm~\ref{alg:Pairing} as being in need of improvement.

\subsection{Experiments for Larger Ranges of Sparsity $s$ and Dimension $D$}
\label{sec:5.4}

\begin{figure}[t!] 

\subfloat[\label{fig:7.1}]
  {\includegraphics[width=.5\linewidth]{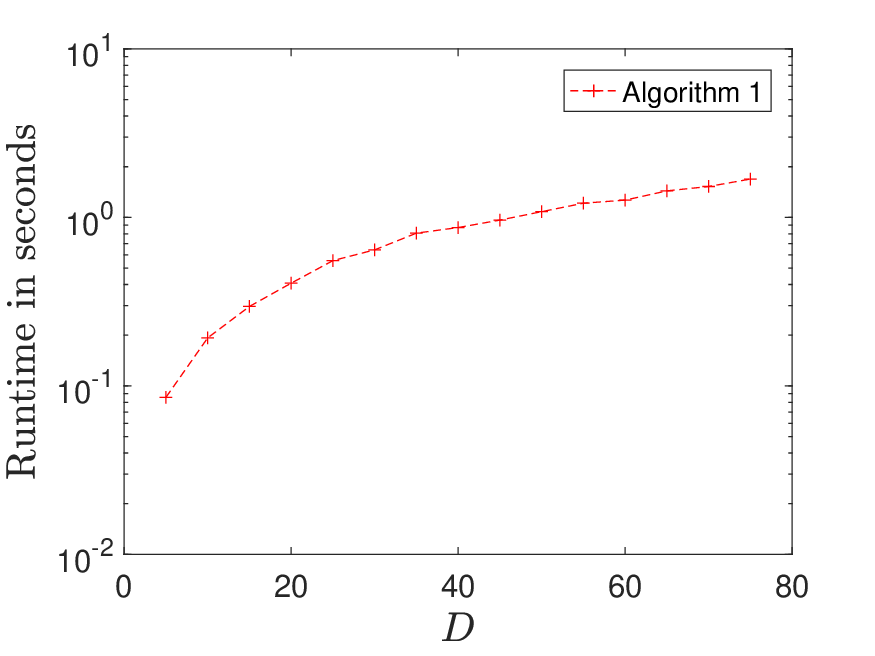}}\hfill
\subfloat[\label{fig:7.2}]
  {\includegraphics[width=.5\linewidth]{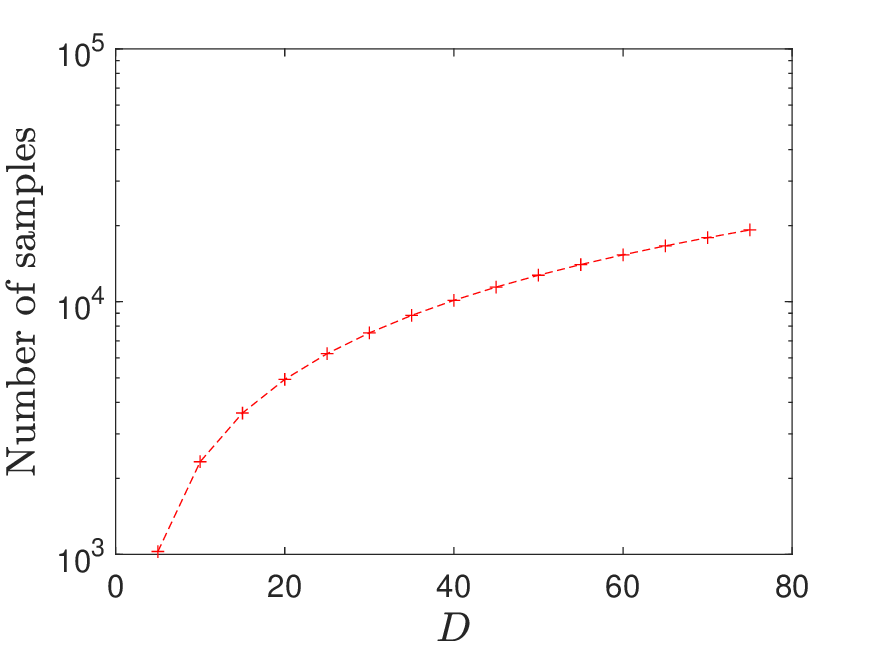}}
\caption{Fourier basis, $M=20,\ D\in \{5,10,15,20,\cdots,75\}, \ s=5$}
\label{fig:7}

\medskip

\subfloat[\label{fig:8.1}]
  {\includegraphics[width=.5\linewidth]{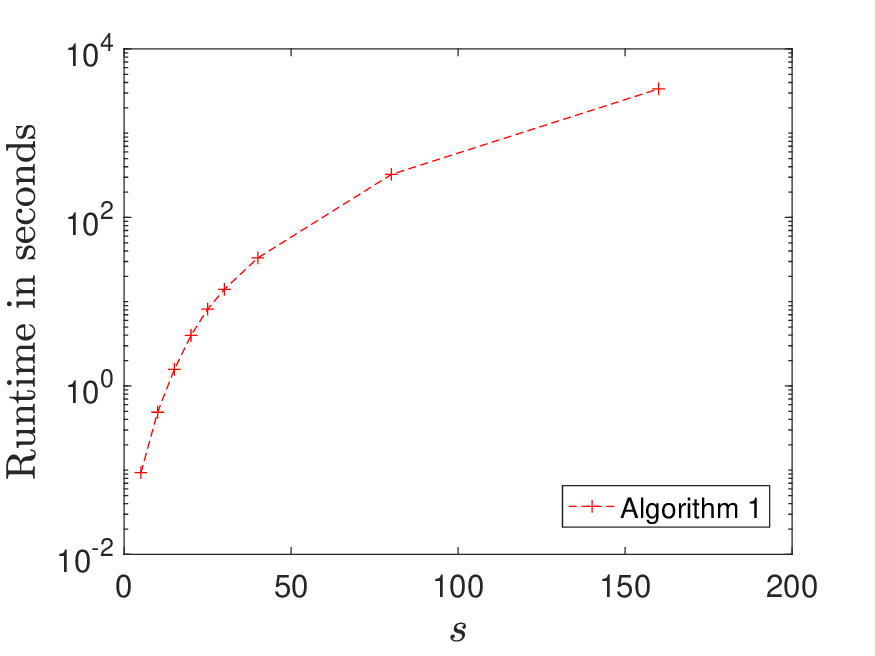}}\hfill
\subfloat[\label{fig:8.2}]
  {\includegraphics[width=.5\linewidth]{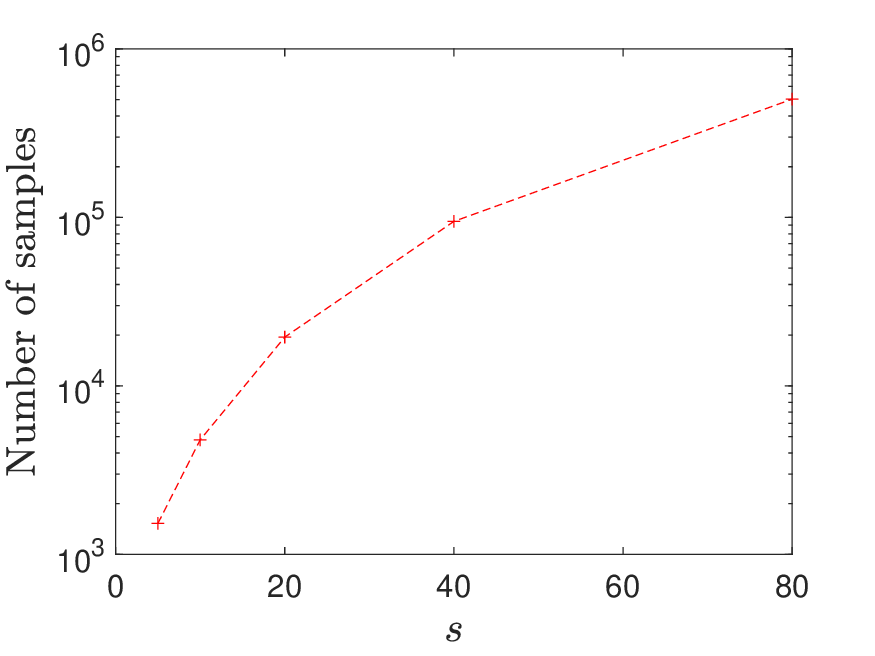}}
\caption{Fourier basis, $M=40,\ D=5, \ s\in\{5,10,20,40,80\} $}
\label{fig:8}

\label{fig: red line: our method, blue line: CoSaMP}
\end{figure}

\begin{figure}[t!] 

\subfloat[\label{fig:9.1}]
  {\includegraphics[width=.5\linewidth]{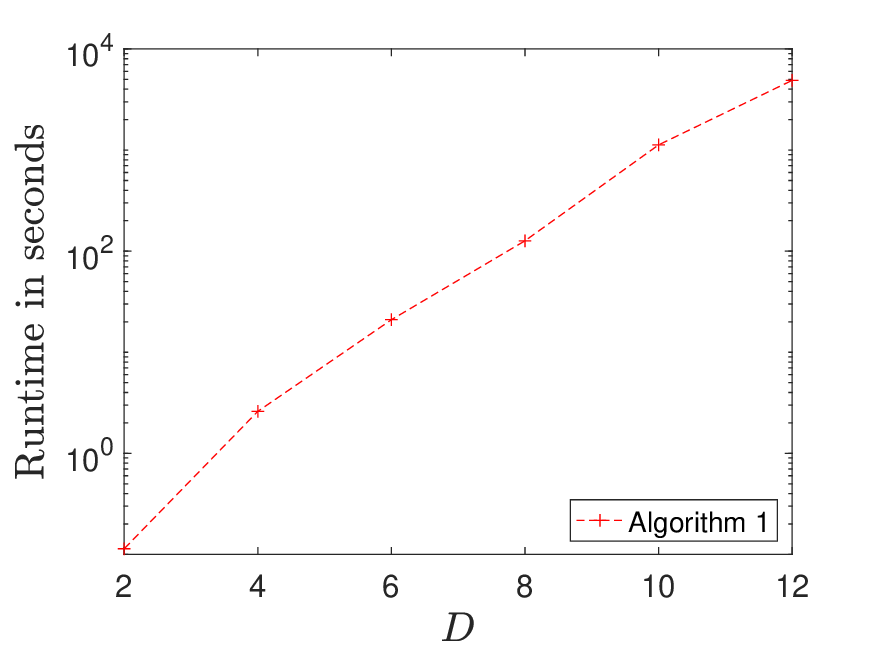}}\hfill
\subfloat[\label{fig:9.2}]
  {\includegraphics[width=.5\linewidth]{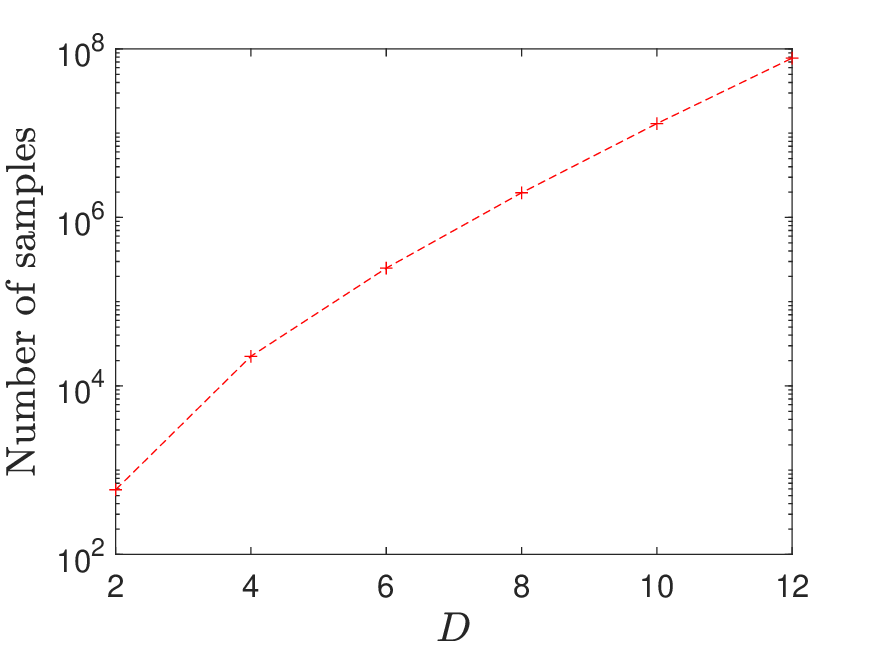}}
\caption{Chebyshev basis, $M=20,\ D\in \{2,4,6,\cdots,12\}, \ s=5$}
\label{fig:9}

\medskip

\subfloat[\label{fig:10.1}]
  {\includegraphics[width=.5\linewidth]{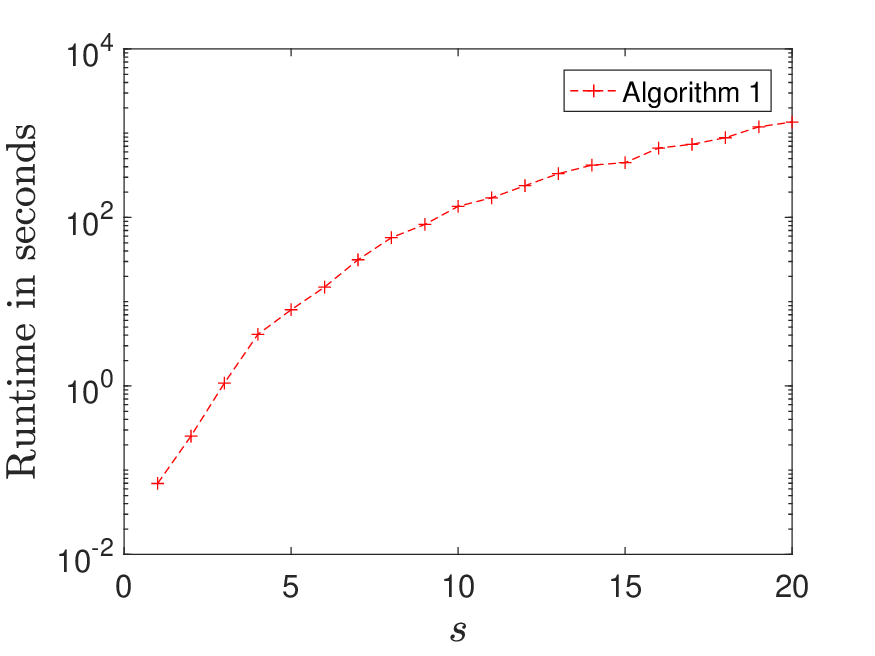}}\hfill
\subfloat[\label{fig:10.2}]
  {\includegraphics[width=.5\linewidth]{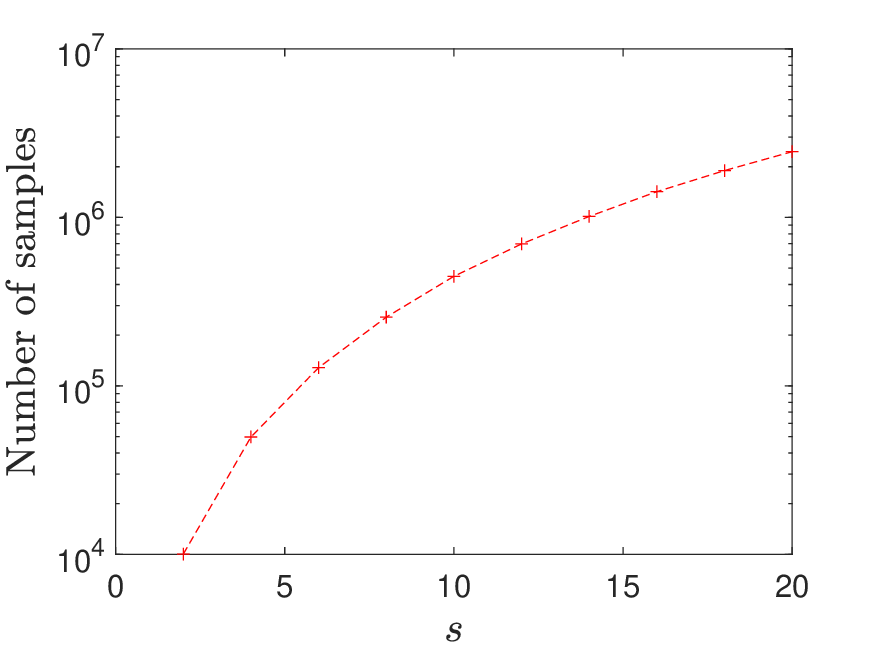}}
\caption{Chebyshev basis, $M=40,\ D=5, \ s\in\{2,4,6,\cdots,20\} $}
\label{fig:10}

\label{fig: red line: our method, blue line: CoSaMP}
\end{figure}

Figures \ref{fig:7} and \ref{fig:8} explore the performance of Algorithm~\ref{alg:main} on Fourier sparse functions for larger ranges of $D$ and $s$, respectively.  In Figure \ref{fig:7}, a function of $D=75$ variables can be recovered in just a few seconds when it is sufficiently sparse in the Fourier basis.  It is worth pointing out here that when $D=75$ the BOS in question contains $20^{75} \sim 10^{97}$ basis functions, significantly more than the approximately $10^{82}$ atoms estimated to be in the observable universe.  We would like to emphasize that Algorithm~\ref{alg:main} is solving problems in this setting that are simply too large to be solved efficiently, if at all, using standard superlinear-time compressive sensing approaches due to their memory requirements when dealing with such extremely large bases.  
Figure \ref{fig:8} also shows that functions with larger Fourier sparsities, $s$, than previously considered (up to $s = 160$) can be be recovered in about an hour or less from a BOS of size $40^5 = 102,400,000$.  

In Figures \ref{fig:9} and \ref{fig:10} we consider the functions which are sparse in the Chebyshev product basis. Again, due to the larger BOS constant of the Chebyshev basis, the $D$ and $s$ ranges that our method can deal efficiently are smaller than in the Fourier case.  When $D$ is $12$ or $s$ is $20$ in Figures \ref{fig:9} and \ref{fig:10}, respectively, for example, it takes a few hours for Algorithm~\ref{alg:main} to finish running. We again remind the readers that standard superlinear-time compressive sensing methods cannot solve with such high dimensional problems at all, however, on anything less than a world class supercomputer due to their memory requirements.  In the Figure \ref{fig:9} experiments the Chebyshev BOS contains $20^{12} \sim 10^{15}$ basis functions when $D = 12$.  In the Figure \ref{fig:10} experiments the BOS contains just over $100$ million basis functions.

\subsection{Recovery of Functions from Noisy Measurements}

\begin{figure}[t!] 

\subfloat[\label{fig:11.1}]
  {\includegraphics[width=.5\linewidth]{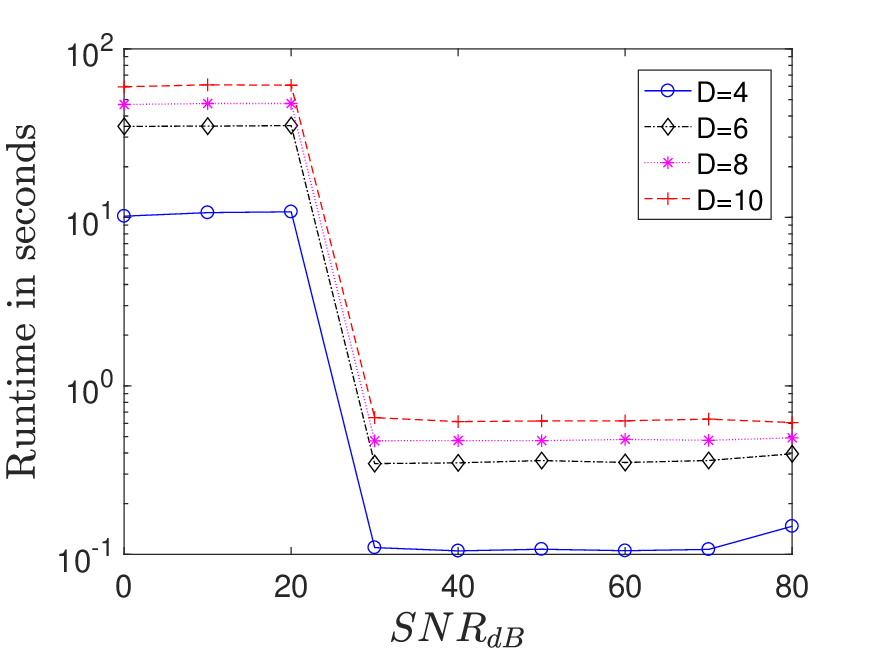}}\hfill
\subfloat[\label{fig:11.2}]
  {\includegraphics[width=.5\linewidth]{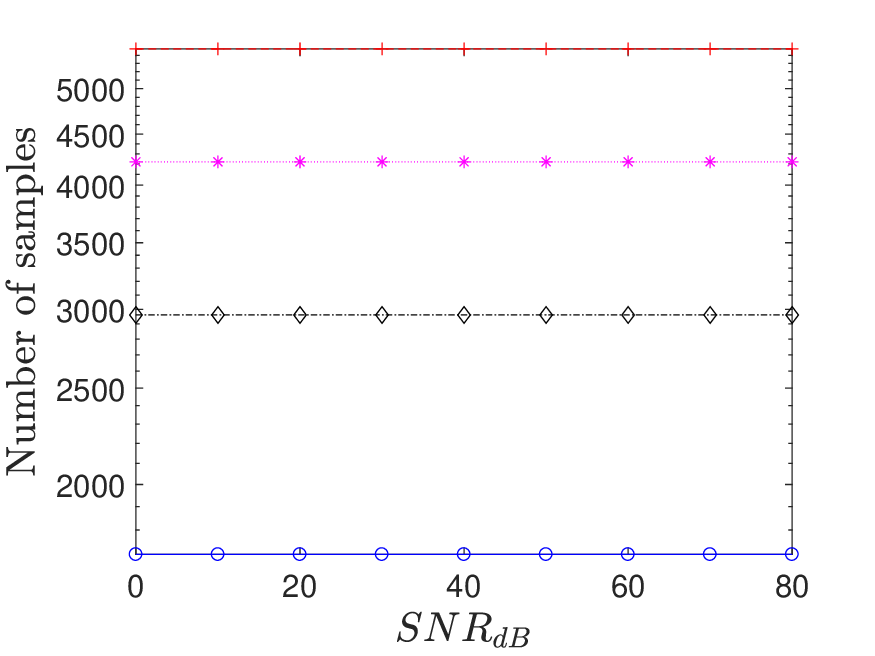}}\hfill
\subfloat[\label{fig:11.3}]
  {\includegraphics[width=.5\linewidth]{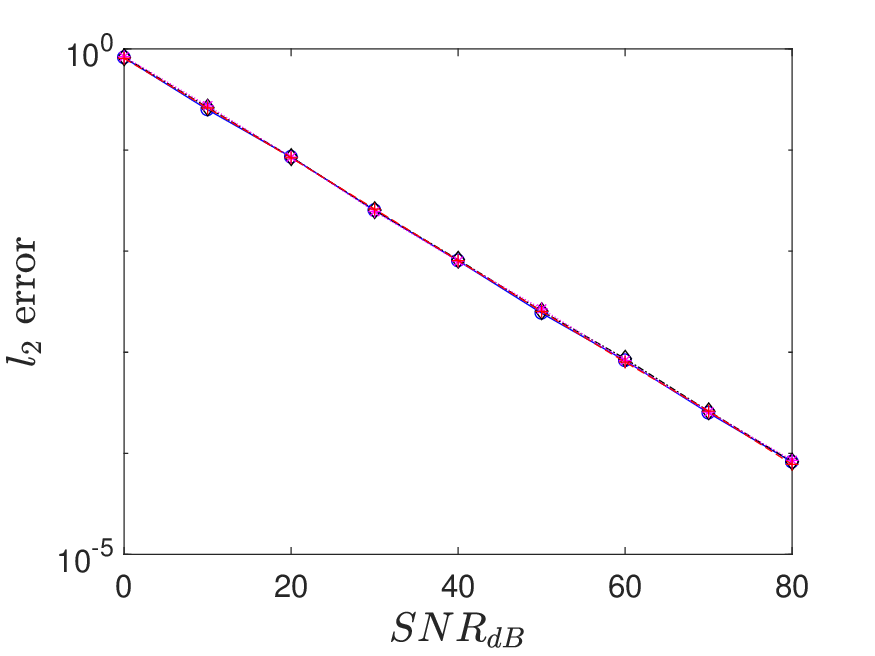}}
\subfloat[\label{fig:11.4}]
  {\includegraphics[width=.5\linewidth]{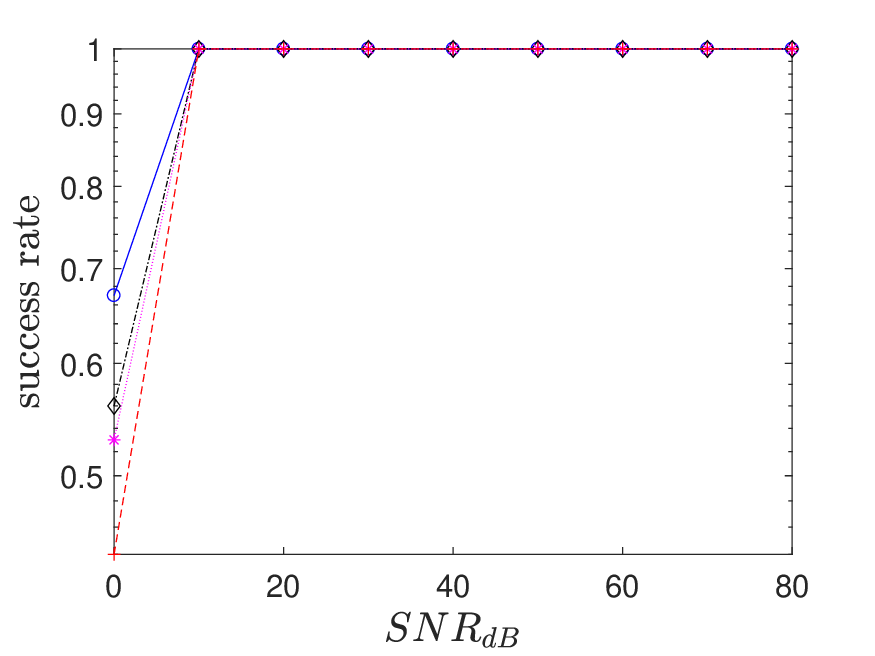}}
\caption{Algorithm \ref{alg:main}, Fourier basis, $M=10,\ D\in\{4,6,8,10\}, \ s=5 ,\ {\rm SNR}_{\rm dB} \in\{0,10,20,\cdots, 80\}$}
\label{fig:11}

\label{fig: red line: our method, blue line: CoSaMP}
\end{figure}

\begin{figure}[t!] 

\subfloat[\label{fig:12.1}]
  {\includegraphics[width=.5\linewidth]{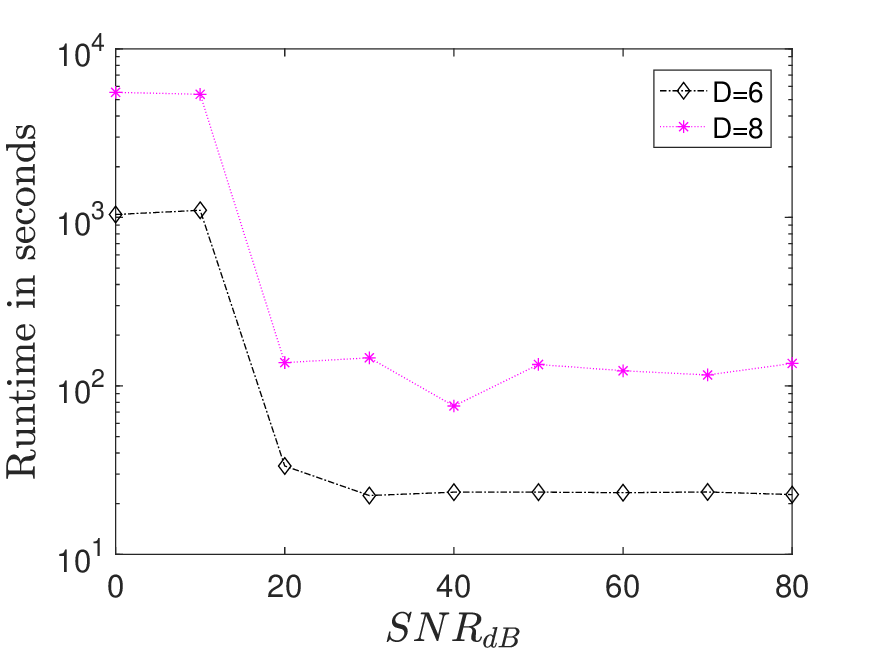}}\hfill
\subfloat[\label{fig:12.2}]
  {\includegraphics[width=.5\linewidth]{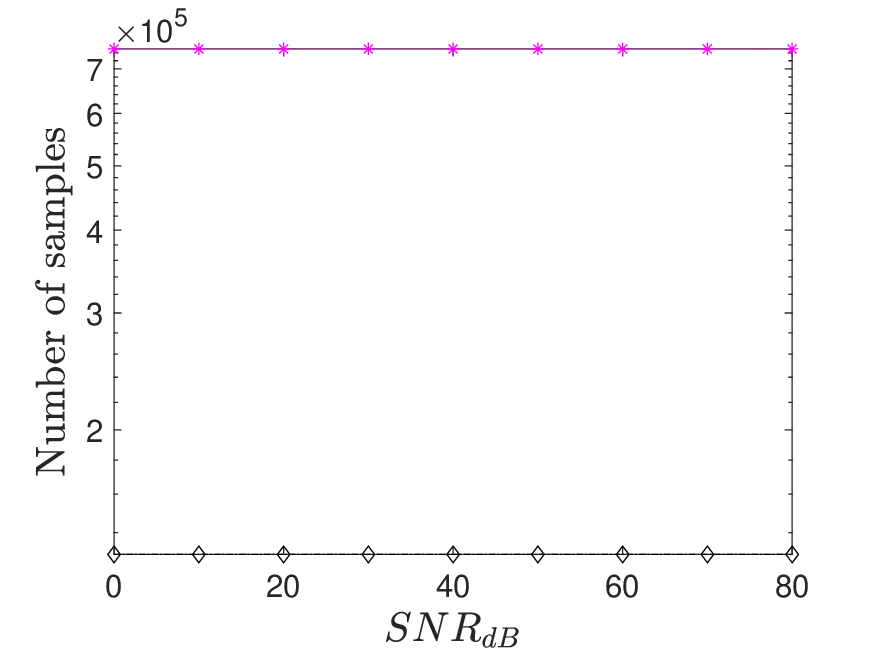}}\hfill
\subfloat[\label{fig:12.3}]
  {\includegraphics[width=.5\linewidth]{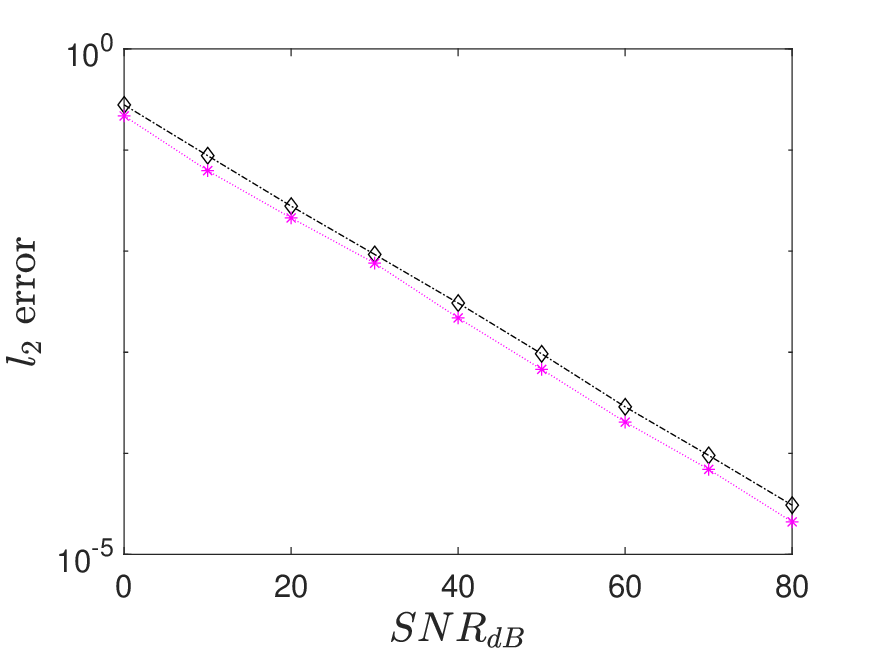}}
\subfloat[\label{fig:12.4}]
  {\includegraphics[width=.5\linewidth]{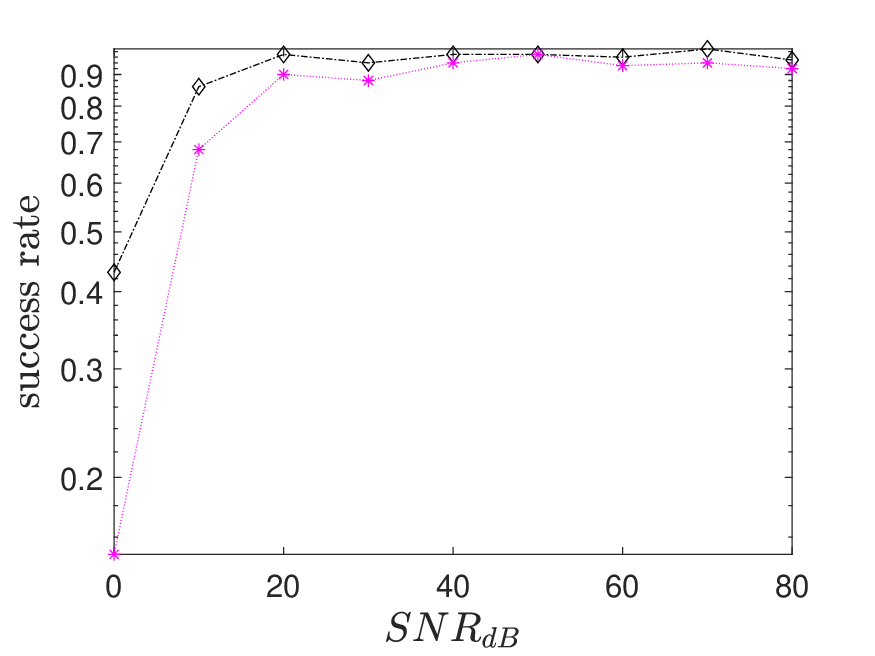}}
\caption{Algorithm \ref{alg:main}, Chebyshev basis, $M=10,\ D\in\{6,8\}, \ s=5 ,\ {\rm SNR}_{\rm dB} \in\{0,10,20,\cdots, 80\}$}
\label{fig:12}

\label{fig: red line: our method, blue line: CoSaMP}
\end{figure}

In Figures~\ref{fig:11}~and~\ref{fig:12} we further consider exactly sparse trial functions \eqref{equ:TestfuncEE} whose function evaluations are contaminated with Gaussian noise.  That is, we provide Algorithm~\ref{alg:main} with noisy samples 
$$\vect{y}' ~=~ \vect{y} + \vect{g}' ~=~ \vect{y} + \sigma  \frac{\| \vect{y} \|^2}{\| \vect{g} \|_2} \vect{g}$$
where $\vect{y}$ contains noiseless samples from each $f$ as per \eqref{Def:SamplesFromf}, $\vect{g} \sim \mathcal{N}(\vect{0},I)$, and $\sigma \in \mathbb{R}^+$ is used to control the Signal to Noise Ratio (SNR) defined herein by
$${\rm SNR}_{\rm dB} := 10 \log_{10} \left( \frac{\| \vect{y} \|^2_2}{\| \vect{g}' \|^2_2 }\right) ~=~ -10 \log_{10} \left( \sigma^2 \right).$$

Figures \ref{fig:11} and \ref{fig:12} show the performance of Algorithm \ref{alg:main} for the Fourier and Chebyshev product bases, respectively, as SNR varies. Figure \ref{fig:11.1} shows the average runtime for each $D\in\{2,4,6,8\}$ as ${\rm SNR}_{\rm dB}$ changes.  When ${\rm SNR}_{\rm dB}$ is close to $0$ (which means that the $\ell_2$-norm of noise vector is the same as the $\ell_2$-norm of sample vector), the runtime gets larger due to Algorithm \ref{alg:main} using a larger number of overall iterations.  The runtime also increases mildly as $D$ increases in line with our previous observations.  The sampling number in Figure \ref{fig:11.2} is set to be three times larger than the sampling number used in the noiseless cases.  Figure \ref{fig:12} shows the results of Algorithm \ref{alg:main} applied to functions which are sparse in the Chebyshev product basis. Similar to the Fourier case, the runtime grows as the noise level gets worse in Figure \ref{fig:12.1}.  Also, larger $D$ results in the larger runtime as previously discussed.  In Figure \ref{fig:12.2}, the sampling number is also set by tripling the sampling number used in noiseless cases. 

As above, in both Figures \ref{fig:11} and \ref{fig:12} the average $\ell_2$-error is computed by only considering the successful trials where every element of $f$'s support, $\mathcal{S}$, is found.  Here, however, the percentage of successful trials falls below $90 \%$ for lower SNR values.  The success rates (i.e., the percentage of successful trials at each data point) are therefore plotted in Figures \ref{fig:11.4} and \ref{fig:12.4}.  Both figures show that a smaller ${\rm SNR}_{\rm dB}$ results in a smaller success rate, as one might expect.  As ${\rm SNR}_{\rm dB}$ increases, however, the $\ell_2$-error decreases linearly for the successful trials.

\subsection{Some Additional Implementational Details}

In the line $13$ of Algorithm \ref{alg:main} solving the least square problem can be accelerated by the iterative algorithms such as the Richardson method or the conjugate gradient method when the size of the matrix $\Phi_T$ is large \cite{needell2009cosamp}. For our range of relatively low sparsities, however, there was not much difference in the runtime between using such iterative least square solving algorithms and simply multiplying $\vect{y^{\rm E}}$ by the Moore-Penrose inverse, $\Phi_T^{\dagger}:=(\Phi_T^*\Phi_T)^{-1}\Phi_T^*$.  Thus, we simply form and use the Moore-Penrose inverse for both CoSaMP and Algorithm \ref{alg:main} in our implementations below.  

Similarly, in our CoSaMP implementation the conjugate transpose of the measurement matrix, $\Phi$, of size $m\times M^D$ is simply directly multiplied by the updated sample vector $\vect{v}$ in each iteration in order to obtain the signal proxy used for CoSaMP's support identification procedure (recall that $d = D$ in all experiments below so that $\mathcal{I} = [M]^D$).  It is important to note that this matrix-vector multiplication can generally be done more efficiently if, e.g., one instead uses nonuniform FFT techniques \cite{greengard2004accelerating} to evaluate $\Phi^* \vect{y}$ for the types of high-dimensional Fourier and Chebyshev basis functions considered below.  However, such techniques are again not actually faster than a naive direct matrix multiply for the ranges of relatively low sparsities we consider in the experiments herein.\footnote{CoSaMP always uses only $m = \mathcal{O}(s \cdot D \log M)$ samples in the experiments herein which means that its measurement matrix's conjugate transpose, $\Phi^* \in \mathbb{C}^{M^D \times m}$, can be naively multiplied by vectors in only $\mathcal{O}(s \cdot D \log M \cdot M^D)$-time.  When $s$ is small this is comparable to the $\mathcal{O}( D \log M \cdot M^D)$ runtime complexity of a (nonuniform) FFT.}  Furthermore, such nonuniform FFT techniques will still exhibit exponential runtime and memory dependence on $D$ in the high-dimensional setting even for larger sparsity levels.  Thus, nonuniform FFTs were not utilized in our MATLAB implementation of CoSaMP.  

Again, we remind the reader that all the MATLAB codes used to produce the plots above is publicly available \cite{MarksCodePage}.  We invite the interested reader to download it and reproduce the plots herein at their leisure.

\section{Future Work} \label{Future Work}
\setcounter{equation}{0}

In this paper we develop a sublinear-time compressive sensing algorithm for rapidly learning functions of many variables that admit sparse representations in arbitrary Bounded Orthonormal Product (BOP) bases.  Our results are universal in the sense that we give randomized constructions for highly structured grids which are proven to allow for the swift recovery {\it of all} functions which are sufficiently sparse in a given BOP basis, with high probability.  This is the first method of its kind for general BOP bases.  As a result, there is much work to do before these preliminary results reach their full potential.

First, and perhaps most obviously, the theoretical guarantees developed herein only apply to exactly BOPB-sparse functions despite the fact that our numerical experiments suggest that the algorithm also works for nearly BOPB-sparse functions.  As a result, it should be possible to extend the main theorem herein to obtain best $s$-term approximation guarantees in the sense of Cohen et al. \cite{cohen2009compressed} for more arbitrary functions while simultaneously improving its complexity bounds (see, e.g., \cite{choi2019sparse} for preliminary results in this direction).  Specific complexity improvements that should be considered include an attempt at reducing the current cubic-in-$s$ sampling complexity of our main theoretical result.  This necessitates that a better pairing method be developed in Section~\ref{sec:proposed} that requires fewer function evaluations.

Additionally, different dimension matching techniques for improving the pairing and entry identification steps of our proposed support identification method could be considered.  Both steps are currently dimension incremental in the sense that all entries of the energetic index vectors are found one dimension at a time, and then extended into longer prefixes one dimension at a time.  However, there is nothing stopping either of these steps from being generalized so that several short (potentially overlapping) prefixes of energetic index vectors are found in parallel and then merged/combined in a different order.  This process could even be made adaptive to help eliminate interference between different index vector prefixes during a modified pairing phase's energy estimations.  For example, any time a set of energetic prefixes differs from all the others under consideration in its $j^{\rm th}$ entry one could compute an inner product in the $j^{\rm th}$ dimension of $h$ in a fashion similar to our current entry identification step in order to better isolate those prefixes' energy estimates from the others.  Such methods could potentially lead to more accurate pairing results in noisy conditions.

Finally, there is also a good deal of improvement possible in the numerical implementation of the methods developed herein.  In particular, Algorithm~\ref{alg:main} was implemented in a generic fashion in our Section~\ref{Empirical Evaluation} experiments.  For better results the implementation should be tuned to the particular BOPB being considered.  It is also important to note that Algorithm~\ref{alg:main} is inherently embarrassingly parallel in nature.  In particular, all of the energetic index vector entry sets $\mathcal{N}_j$ can be computed in parallel.  Similarly, during the pairing step of support identification several prefixes can be grown simultaneously from, e.g., both the front and end of the index vector until the energetic prefixes and suffixes meet.  Upon meeting in the middle, one additional energy estimate could then be done to correctly pair the proper prefixes and suffixes together.

\section{Acknowledgements} The authors thank Holger Rauhut for fruitful discussions on the topic.
Both Mark Iwen and Felix Krahmer acknowledge support by the TUM August-Wilhelm-Scheer (AWS) Visiting Professor Program that allowed for the initiation of this project.
Mark Iwen was supported in part by NSF DMS-1416752 and NSF CCF-1615489.  Bosu Choi was supported in part by NSF DMS-1416752.  Felix Krahmer was supported in part by the German Science foundation in the context of the Emmy Noether junior research group KR 4512/1-1.

\bibliography{BosuEDs_MarkFelix}
  
\end{document}